\documentclass[reqno]{amsart}

\usepackage[ansinew]{inputenc}
\usepackage[T1]{fontenc}

\usepackage{amsmath}
\usepackage{amssymb}

\usepackage[curve]{xypic}
\usepackage{amsthm}
\usepackage{lscape}

\usepackage{tikz}
\usepackage{tikz-cd}
\usetikzlibrary{shapes,arrows,calc}
\usepackage{graphicx}
\usepackage{enumitem}
\usepackage{todonotes}
\usepackage{breakurl}

\usepackage{hyperref}
\usepackage[capitalise]{cleveref}

\usepackage{blkarray} 

\usepackage{stmaryrd}

\newtheorem{satz}{Satz}[section]
\newtheorem{Theorem}[satz]{Theorem}
\newtheorem{Lemma}[satz]{Lemma}
\newtheorem{lemma}[satz]{Lemma}
\newtheorem{Prop}[satz]{Proposition}
\newtheorem{prop}[satz]{Proposition}
\newtheorem{Cor}[satz]{Corollary}
\newtheorem{Conj}[satz]{Conjecture}

\theoremstyle{definition}
\newtheorem{Definition}[satz]{Definition}
\newtheorem{Notation}[satz]{Notation}

\theoremstyle{remark}
\newtheorem{example}[satz]{Example}
\newtheorem{remark}[satz]{Remark}

\Crefname{Theorem}{Theorem}{Theorems}
\Crefname{Definition}{Definition}{Definitions}
\Crefname{Lemma}{Lemma}{Lemmas}
\Crefname{lemma}{Lemma}{Lemmas}
\Crefname{Notation}{Notation}{Notations}
\Crefname{Prop}{Proposition}{Propositions}
\Crefname{prop}{Proposition}{Propositions}
\Crefname{example}{Example}{Examples}
\Crefname{remark}{Remark}{Remarks}
\Crefname{Cor}{Corollary}{Corollaries}
\Crefname{Conj}{Conjecture}{Conjectures}
\Crefname{equation}{Equation}{Equations}

\renewcommand{\O}{\mathcal{O}}
\newcommand{\OO}{\mathcal{O}}
\renewcommand{\AA}{\mathcal{A}}
\newcommand{\F}{\mathcal{F}}
\renewcommand{\L}{\mathcal{L}}
\newcommand{\R}{\mathbb{R}}
\newcommand{\Z}{\mathbb{Z}}
\renewcommand{\H}{\mathbb{H}}
\newcommand{\Q}{\mathbb{Q}}

\newcommand{\MM}{\mathcal{M}}
\newcommand{\EE}{\mathcal{E}}
\newcommand{\FF}{\mathcal{F}}
\newcommand{\Fb}{\mathbb{F}}
\newcommand{\GG}{\mathbb{G}}
\newcommand{\cG}{\mathcal{G}}
\newcommand{\MMb}{\overline{\mathcal{M}}}
\newcommand{\mell}{\MM_{ell}}
\newcommand{\mcub}{\MM_{cub}}
\newcommand{\XX}{\mathcal{X}}
\newcommand{\YY}{\mathcal{Y}}
\newcommand{\C}{\mathbb{C}}

\newcommand{\bH}{\mathbb{H}}
\newcommand{\A}{\mathcal{A}_*}

\newcommand{\omline}{\underline{\omega}}
\newcommand{\aq}{\overline{a}}
\newcommand{\m}{\mathfrak{m}}
\newcommand{\p}{\mathfrak{p}}
\newcommand{\At}{\widetilde{A}}
\newcommand{\Gt}{\widetilde{\Gamma}}

\DeclareMathOperator{\Hom}{Hom}

\DeclareMathOperator{\Ext}{Ext}
\DeclareMathOperator{\Pic}{Pic}

\DeclareMathOperator{\MF}{mf}
\DeclareMathOperator{\MFC}{MF}
\DeclareMathOperator{\MFq}{mf}
\newcommand{\mf}{\operatorname{mf}_1(7)}
\DeclareMathOperator{\id}{id}

\DeclareMathOperator{\SL}{SL}
\DeclareMathOperator{\Spec}{Spec}
\DeclareMathOperator{\tensor}{\otimes}
\DeclareMathOperator{\Ell}{Ell}
\DeclareMathOperator{\Nat}{Nat}

\DeclareMathOperator{\Tate}{Tate}

\DeclareMathOperator{\tr}{tr}

\DeclareMathOperator{\Conv}{Conv}
\DeclareMathOperator{\ev}{ev}

\DeclareMathOperator{\QCoh}{QCoh}
\DeclareMathOperator{\modules}{\text{-}mod}
\DeclareMathOperator{\Tr}{Tr}
\DeclareMathOperator{\res}{res}

\DeclareMathOperator{\TMF}{TMF}
\DeclareMathOperator{\tmf}{tmf}
\DeclareMathOperator{\Tmf}{Tmf}

\newcommand{\zsk}{(\mathbb{Z}/7)^{\times}}

\renewcommand{\P}{\mathbb{P}}

\newenvironment{pf}{\begin{proof}[Proof]}{\end{proof}}

\makeatletter
\let\c@equation\c@satz
\numberwithin{equation}{section}
\makeatother

\allowdisplaybreaks

\author{Lennart Meier}
\address{Mathematisch Instituut, Universiteit Utrecht, Budapestlaan 6, 3584 CD Utrecht, Nederland}
\email{f.l.m.meier@uu.nl}
\author{Viktoriya Ozornova}
\address{Fakult\"at f\"ur Mathematik, Ruhr-Universit\"at Bochum, D-44780 Bochum, Germany}
\email{viktoriya.ozornova@rub.de}

\title{Rings of modular forms and a splitting of $\TMF_0(7)$}

\begin{document}

\begin{abstract}
Among topological modular forms with level structure, $\TMF_0(7)$ at the prime $3$ is the first example that had not been understood yet. We provide a splitting of $\TMF_0(7)$ at the prime 3 as $\TMF$-module into two shifted copies of $\TMF$ and two shifted copies of $\TMF_1(2)$. This gives evidence to a much more general splitting conjecture. Along the way, we develop several new results on the algebraic side. For example, we show the normality of rings of modular forms of level $n$ and introduce cubical versions of moduli stacks of elliptic curves with level structure. 
\end{abstract}

\maketitle
\tableofcontents

\section{Introduction}
The main objects of our study, spectra of \emph{topological modular forms}, are beasts in which arithmetic geometry and stable homotopy theory heavily intertwine. This entails that a huge part of our article is purely concerned with modular forms and different moduli stacks and a smaller part draws the consequences in homotopy theory. For the convenience of readers with different backgrounds, we have accordingly divided the introduction into two parts that separate our algebraic and homotopy-theoretic motivations and results. 

\addtocontents{toc}{\protect\setcounter{tocdepth}{1}}
\subsection*{Modular forms and moduli stacks}
Given a congruence subgroup $\Gamma \subset \SL_2(\Z)$, we consider the corresponding ring of holomorphic modular forms $\MF(\Gamma;\C)$. For every subring $R\subset \C$, we can further consider $\MF(\Gamma; R)$, the subring of modular forms whose $q$-expansion has coefficients in $R$. We will mainly be interested in the cases $\Gamma = \Gamma_0(n)$ and $\Gamma_1(n)$. While modular forms with respect to these groups have received an extraordinary amount of attention in number theory, the ring-theoretic properties of $\MF(\Gamma_0(n); R)$ and $\MF(\Gamma_1(n); R)$ have been far less studied. Exceptions are Deligne and Rapoport \cite{DeligneRapoport}, who show that these rings are finitely generated $\Z[\frac1n]$-algebras, and the work of Rustom \cite{RustomGenerators} \cite{RustomGeneratorsRelations}, where he provides both bounds on the degrees of generators and relations and an algorithm to determine the ring structure of $\MF(\Gamma_1(n); R)$ in weights at least $2$ if $R=\Z[\frac1n]$. 

Part of our algebraic aims in this article is to continue the study of the rings $\MF(\Gamma_1(n); R)$ and also of the corresponding moduli stacks $\MMb_1(n)$. These are the compactifications of the moduli stack $\MM_1(n)$ of elliptic curves with a chosen point of exact order $n$ over $\Z[\frac1n]$-algebras. The connection is that at least for $\frac1n\in R$ we can reinterpret the weight-$k$-part $\MF_k(\Gamma_1(n); R)$ as $H^0(\MMb_1(n)_R; \omline^{\tensor k})$, where $\omline$ denotes the pushforward of the sheaf of differentials of the universal generalized elliptic curve. Similarly, we can reinterpret $\MF_k(\Gamma_0(n); R)$ as $H^0(\MMb_0(n)_R; \omline^{\tensor k})$, where $\MMb_0(n)$ is the compactification of the moduli stack $\MM_0(n)$ of elliptic curves with chosen cyclic subgroup of order $n$ over $\Z[\frac1n]$-algebras. This allows to define modular forms for $\Z[\frac1n]$-algebras $R$ not contained in $\C$ as well.

Our first aim is to give equations for the universal elliptic curve over $\MM_1(n)$ in terms of modular forms. A more precise statement can be found in \cref{prop:holomorphicalphas} and its surrounding. 
\begin{Theorem}\label{thm:holomorphicTate}
For $n\geq 3$, the universal elliptic curve over $\MM_1(n)$ can be expressed via an equation
\[y^2 + \alpha_1 xy + \alpha_3 y = x^3 + \alpha_2 x^2,\]
where the $\alpha_i$ are \emph{holomorphic} modular forms for $\Gamma_1(n)$ of weight $i$, with explicit formulae for their $q$-expansions. 
\end{Theorem}

Our main example is the case $n=7$.

\begin{Theorem}\label{thm:Gamma17alph}
There is an isomorphism 
\[\MF(\Gamma_1(7);\Z) \cong \Z[z_1,z_2, z_3]/(z_1z_2+z_2z_3+z_3z_1)\] 
with $|z_i| = 1$. A generator $t$ of the group $\Gamma_1(7)\!\setminus\!\Gamma_0(7)\cong \zsk$ acts via 
\[t.z_3 = -z_1,\quad t.z_1 = -z_2 \quad \text{ and }\quad t.z_2 = -z_3.\]
The $\alpha_i$ are given by 
\[\alpha_1 = z_1-z_2 +z_3,\quad \alpha_2 = z_1z_2 +z_1z_3\quad\text{ and }\quad\alpha_3 = z_1z_3^2.\]
\end{Theorem}
While our computation of $\MF(\Gamma_1(7);\Z)$ is not difficult to obtain, the reader should compare the simplicity of its expression with the presentation Rustom \cite[Section 3.1]{RustomGeneratorsRelations} is forced to give by ignoring the elements of weight $1$. 

After these explicit computations, we also prove more structural results about rings of modular forms. 
\begin{Theorem}
For every $n\geq 2$, the ring $\MF(\Gamma_1(n);\Z[\frac1n])$ is normal.
\end{Theorem}
Using results from \cite{MeierDecomposition}, we give moreover a criterion for $\MF(\Gamma_1(n);\Z[\frac1n])$ being Cohen--Macaulay, which is satisfied for all $2\leq n \leq 28$. 

These commutative algebra results help us to develop \emph{cubical} analogues of the stacks $\MMb_1(n)$. To this purpose recall that $\MMb_1(n)$ is usually defined as the normalization of the compactified moduli stack of elliptic curves $\MMb_{ell}$ in $\MM_1(n)$. The stack $\MMb_{ell}$ embeds into the larger stack $\MM_{cub}$ of all curves that can locally be described by a cubic Weierstra{\ss} equation. We can define $\MM_1(n)_{cub}$ as the normalization of $\MM_{cub}$ in $\MM_1(n)$. We show: 

\begin{Theorem}
For $n\geq 2$, the stack $\MM_1(n)_{cub}$ is equivalent to the stack quotient $\Spec \MF(\Gamma_1(n);\Z[\frac1n])/\GG_m$. Moreover, the map $\MM_1(n)_{cub} \to \MM_{cub,\Z[\frac1n]}$ is finite and flat if $\MF(\Gamma_1(n);\Z[\frac1n])$ is Cohen--Macaulay.
\end{Theorem}

Our reason to consider these cubical stacks is that vector bundles seem to be easier to study on $\MM_{cub}$ than on $\MMb_{ell}$, a view inspired by work of Mathew \cite{Mathewtmf}. This line of thought is already implicit though in the classification of vector bundles over weighted projective lines in \cite[Proposition 3.4]{MeierVB}, where the idea was to extend them to vector bundles on the non-separated stack $\mathbb{A}^2/\GG_m$. For $\MM_{cub}$ this idea takes the form that there is a nice smooth cover $\Spec A \to \MM_{cub}$ with $A = \Z[a_1, a_2, a_3, a_4, a_6]$ given by Weierstra{\ss} curves and thus quasi-coherent sheaves on $\MM_{cub}$ become equivalent to comodules over a certain explicit Hopf algebroid $(A,\Gamma)$. For explicit calculations, this outweighs the disadvantage that $\MM_{cub}$ is neither Deligne--Mumford nor separated. 

Our wish for such explicit calculations was motivated by the results from \cite{MeierDecomposition}. There it was shown that the pushforward $(f_n)_*\OO_{\MMb_1(n)_{(p)}}$ along the map 
\[f_n\colon \MMb_1(n)_{(p)} \to \MMb_{ell,(p)}\]
splits for a prime $p$ under mild hypotheses into a few simple pieces and the same is true for the pushforward $(h_n)_*\OO_{\MMb_0(n)_{(p)}}$ along $h_n\colon \MMb_0(n)_{(p)} \to \MMb_{ell,(p)}$ if $p$ does not divide $\lvert (\Z/n)^\times \rvert$ or $p>3$. For $p=3$, the first case not covered is $\MMb_0(7)$. By extending the vector bundle $(h_7)_*\OO_{\MMb_0(7)_{(3)}}$ to $\MM_{cub,(3)}$ and performing computation with comodules over Hopf algebroids, we arrive at the following splitting result. 

\begin{Theorem}\label{thm:MainTheorem}
 The quasi-coherent sheaf $(h_7)_*\O_{\MMb_0(7)_{(3)}}$ on $\MMb_{ell,(3)}$ is a vector bundle of rank $8$, which can be decomposed as a sum 
 $$\O_{\MMb_{ell,(3)}} \oplus \omline^{-6} \oplus \left((f_2)_*\OO_{\MMb_1(2)_{(3)}}\otimes\omline^{-2}\right) \oplus \left((f_2)_*\OO_{\MMb_1(2)_{(3)}}\otimes\omline^{-4}\right).$$
\end{Theorem}

This obviously implies for every $\Z_{(3)}$-algebra $R$ a splitting
\begin{equation}\label{eq:modformsplitting}\MF_*(\Gamma_0(7);R) \cong \MF_*^R \oplus \MF_{*-6}^R \oplus \MF_{*-2}(\Gamma_1(2);R) \oplus \MF_{*-4}(\Gamma_1(2);R)\end{equation}
(with $\MF_*^R = \MF_*(\SL_2(\Z);R)$), but is a far stronger statement. One reason for our interest in this stronger statement will be apparent in the next subsection when we discuss a topological analogue of \eqref{eq:modformsplitting}. 

\subsection*{Topological modular forms}
The spectrum $\TMF$ of \emph{topological modular forms} was introduced by Goerss, Hopkins and Miller as a topological analogue of the ring $\MFC(\SL_2(\Z);\Z)$ of meromorphic modular forms. It is constructed as the global sections of a sheaf of ring spectra on the moduli stack of elliptic curves $\MM_{ell}$. As such it has the disadvantage that its homotopy groups are infinitely generated in most degrees, which is different for the refinement $\Tmf$ that is based on the compactified moduli stack $\MMb_{ell}$ instead. Its connective cover $\tmf$  can be seen as a topological analogue of the ring $\MF(\SL_2(\Z); \Z)$ of holomorphic modular forms. The unit map $\mathbb{S} \to \tmf$ identifies it moreover as a good approximation to the sphere spectrum in a certain range. We refer to \cite{TMFbook} as a basic reference for topological modular forms.

In many respects, $\TMF$ and $\Tmf$ can be seen as analogues of the real $K$-theory spectrum $KO$ at chromatic height $2$. The study of modules over $KO$ has been central in Bousfield's work on the classification of $E(1)$-local spectra at the prime $2$ \cite{Bou90}. In analogy, we expect modules over topological modular forms to play a central role in our understanding of $E(2)$-local spectra. For concrete incarnations of this philosophy after a further $K(2)$-localization see for example \cite{BehrensModular} or \cite{BobGoerss}. 

The $\TMF$-modules of relevance in this context are examples of \emph{topological modular forms with level structure}. Among the most important of these are the spectra $\TMF_1(n)$ and $\TMF_0(n)$, which are topological analogues of the rings of meromorphic modular forms for the congruence groups $\Gamma_1(n)$ and $\Gamma_0(n)$, respectively. More precisely, they arise as the global sections of sheaves of ring spectra on the stacks $\MM_1(n)$ and $\MM_0(n)$ introduced above, where we still implicitly invert $n$ everywhere. Hill and Lawson \cite{HillLawson} were able to define spectra $\Tmf_0(n)$ and $\Tmf_1(n)$ based on the compactified moduli $\MMb_0(n)$ and $\MMb_1(n)$ as well. While the $K(2)$-localizations of $\TMF_0(n)$ and $\Tmf_0(n)$ are equivalent, the latter appears to be a more powerful tool to understand the $E(2)$-local world and is also the first step to construct connective versions $\tmf_0(n)$ with interesting cohomological properties. 

Analogously to the algebraic splitting results mentioned above, the first-named author \cite{MeierTMFLevel} has proven splitting results for $\Tmf_1(n)$ and $\Tmf_0(n)$ in many cases if we localize at a prime $p$. If $p>3$, there is an explicit criterion when these modules are free over $\Tmf$. If $p=3$, the splittings have shifted copies of $\Tmf_1(2)_{(3)}$ as their summands. As $\pi_*\Tmf_1(2)_{(3)}$ is torsionfree, splittings into shifted copies of $\Tmf_1(2)_{(3)}$ can only exist if $\pi_*\Tmf_0(n)_{(3)}$ is also torsionfree, which is not expected if $3$ divides $|(\Z/n)^\times|$ and the criterion from \cite{MeierTMFLevel} does indeed not apply in this case. The first case where this occurs is $\Tmf_0(7)$, where we can prove the following modified splitting result. 

\begin{Theorem}\label{thm:TopMainResult}
The $\Tmf_{(3)}$-module $\Tmf_0(7)_{(3)}$ decomposes as 
$$\Tmf_{(3)} \oplus \Sigma^4\Tmf_1(2)_{(3)} \oplus \Sigma^8\Tmf_1(2)_{(3)}\oplus L,$$
where $L \in \Pic(\Tmf_{(3)})$, i.e.\ $L$ is an invertible $\Tmf_{(3)}$-module. 
The $\TMF_{(3)}$-module $\TMF_0(7)_{(3)}$ decomposes as
$$\TMF_{(3)} \oplus \Sigma^4\TMF_1(2)_{(3)} \oplus \TMF_1(2)_{(3)}\oplus \Sigma^{36}\TMF_{(3)}.$$
\end{Theorem}

This topological analogue of \eqref{eq:modformsplitting} is based on the algebraic splitting result \cref{thm:MainTheorem} and on computations by M.\ Olbermann. The necessary parts of the latter can be found in \cref{AppB}, whose results are completely those of Olbermann. These also allow to characterize $L$ precisely, using additionally the work of Mathew and Stojanoska \cite{MathewStojanoska}. 

While in itself, \Cref{thm:TopMainResult} might seem an isolated result, it is an important test case of the following more general conjecture for arbitrary $n\geq 1$.

\begin{Conj}\label{conj:TMFsplitting}
The $\TMF$-module $\TMF_0(n)_{(3)}$ decomposes into possibly shifted copies of $\TMF_1(2)_{(3)}$ and of $\TMF_{(3)}$.

If $H^1(\MMb_0(n); \omline)$ has no $3$-torsion, the $\Tmf$-module $\Tmf_0(n)_{(3)}$ decomposes into possibly shifted copies of $\Tmf_1(2)_{(3)}$ and into invertible $\Tmf_{(3)}$-modules. 
\end{Conj}

By the results of \cite{MathewMeier}, this conjecture actually implies the corresponding algebraic conjecture that $(f_n)_*\OO_{\MMb_1(n)_{(3)}}$ splits under the same condition on $H^1(\MMb_0(n); \omline)$ into copies of $\omline^{\tensor ?}$ and $(f_2)_*\OO_{\MMb_1(2)_{(3)}} \tensor \omline^{\tensor ?}$. There are also analogues of \cref{conj:TMFsplitting} for other primes. For primes $p>3$ it was already proven in \cite{MeierTMFLevel} that $\TMF_0(n)_{(p)}$ is (after $p$-completion) a free $\TMF_{(p)}$-module. For $p=2$, we expect a splitting into possibly shifted copies of $\TMF_1(3)_{(2)}$, $\TMF_0(3)_{(2)}$ and $\TMF_0(5)_{(2)}$. 

Lastly, we mention that we believe the stacks $\MM_1(n)_{cub}$ and $\MM_0(n)_{cub}$ to be important for the understanding of connective variants $\tmf_1(n)$ and $\tmf_0(n)$ of topological modular forms with level structure, a point we plan to come back to in future work.

\subsection*{Overview of the structure of the paper}
In rough outline, our strategy is to show \cref{thm:MainTheorem} in Sections \ref{sec:ModularForms} to \ref{sec:Comodule} and to deduce \cref{thm:TopMainResult} in the last section. As first preparation, \cref{sec:ModularForms} computes the ring $\MF(\Gamma_1(7), \Z)$ together with its natural action by $\Gamma_1(7)\!\setminus\Gamma_0(7) \cong (\Z/7)^\times$. Moreover, it provides explicit $q$-expansions of the generators. In \cref{sec:TateCurve} we study the equations of the universal elliptic curves over $\MM_1(n)$. In particular, we show \cref{thm:holomorphicTate} and the identification of the $\alpha_i$ in \cref{thm:Gamma17alph}, which is important input for our later computations. 

Our strategy to show Theorem \ref{thm:MainTheorem} is to show a statement about the corresponding comodules. The precise relationship between quasi-coherent sheaves and graded comodules will be recalled in \cref{sec:HopfStacks}. It is both easier and yields stronger results to use this relationship not for $\MMb_{ell}$ but for $\mcub$ instead. The latter stack has a presentation by the Hopf algebroid $(A,\Gamma)$ with $A = \Z[a_1,a_2,a_3,a_4,a_6]$ and at the prime $3$ actually also by a smaller Hopf algebroid $(\widetilde{A}, \widetilde{\Gamma})$. Thus, quasi-coherent sheaves on $\MM_{cub, (3)}$ become equivalent to graded $(\widetilde{A}, \widetilde{\Gamma})$-comodules. To formulate a version of Theorem \ref{thm:MainTheorem} on $\MM_{cub}$ we define and explore the cubical versions of $\MM_1(n)$ and $\MM_0(n)$ by a normalization procedure in \Cref{sec:Flatness}. 

The next step is to make the Hopf algebroids corresponding to $\MM_1(7)_{cub}$ and $\MM_0(7)_{cub}$ explicit. In \Cref{sec:Splitting}, we produce explicit $\widetilde{A}$-bases of $\widetilde{A}$-algebras $R_{\widetilde{A}}$ and $S_{\widetilde{A}}$, which are defined by 
$$\Spec R_{\widetilde{A}} \cong \MM_1(7)_{cub}\times_{\mcub} \Spec \widetilde{A}\quad \text{ and }\; \Spec S_{\widetilde{A}} \cong \MM_0(7)_{cub}\times_{\mcub} \Spec \widetilde{A}.$$ 
This allows us to prove in \Cref{sec:Comodule} a splitting of $S_{\widetilde{A}}$ as a graded comodule over $(\widetilde{A},\widetilde{\Gamma})$, which implies \Cref{thm:MainTheorem}. In \Cref{sec:TopConclusions} we apply standard techniques (the transfer and the descent spectral sequence) to deduce our topological main theorem. 

We continue with \cref{qExpAppendix}, which gives an exposition of the theory of modular forms with level over general rings and their $q$-expansions. The reason for the length of this appendix is the subtle difference between so-called arithmetic and naive level structures, which only agree in the presence of an $n$-th root of unity. To achieve a $q$-expansion principle in the form we need, we have to choose a slightly unusual identification of the sections of $\omline^{\tensor *}$ on $\MMb_1(n)_{\C}$ with holomorphic $\Gamma_1(n)$-modular forms in the classical sense. 

Lastly, \cref{AppB} explains the pieces of M.\ Olbermann's unpublished computation of $\pi_*\Tmf_0(7)$ we need to prove the precise form of \cref{thm:TopMainResult}.

\subsection*{Acknowledgements}
We thank Martin Olbermann and Dominik Absmeier for helpful discussions. In particular, we thank Martin for sharing his unpublished work with us, of which part resulted in \cref{AppB}.
  
Both authors thank SPP 1786 for its financial support. Moreover, the authors would like to thank the Isaac Newton Institute for Mathematical Sciences for support and hospitality during the programme \lq\lq Homotopy harnessing higher structures\rq\rq{} when work on this paper was undertaken. This work was supported by 
EPSRC grant numbers EP/K032208/1 and EP/R014604/1. 

\subsection*{Conventions}
All quotients of schemes by group schemes (like $\GG_m$) are understood to be stack quotients. Unless clearly otherwise, all rings and algebras are assumed to be commutative and unital. Tensor products of quasi-coherent sheaves are always over the structure sheaf.


\section{Modular forms of level $7$}\label{sec:ModularForms}
\addtocontents{toc}{\protect\setcounter{tocdepth}{2}}

Our goal in this section is to understand the ring of modular forms $\MF(\Gamma_1(7);\Z)$ with respect to the congruence group $\Gamma_1(7)\subset \SL_2(\Z)$ together with the action of $\left(\Z/7\right)^{\times}\cong\Z/6$. We refer for notation and background on modular forms and in particular on the $q$-expansion principle to \cref{qExpAppendix}. 

We begin with some recollections about moduli of elliptic curves. We denote by $\MM_{ell}$ the moduli stack of elliptic curves and by $\MMb_{ell}$ its compactification in the sense of $\mathfrak{M}_1$ of \cite[Chapter III]{DeligneRapoport}, i.e.\ the stack classifying generalized elliptic curves whose geometric fibers are elliptic curves or N\'eron $1$-gons. By $\MM_1(n)$ we denote the moduli stack of elliptic curves $E$ with chosen \emph{point $P\colon S \to E$ of exact order $n$} over schemes $S$ with $n$ invertible and by $\MM_0(n)$ the moduli stack of elliptic curves with chosen \emph{cyclic subgroup of order $n$} over such schemes. More precisely, we demand for $\MM_1(n)$ that for every geometric point $s\colon \Spec K \to S$ the pullback $s^*P$ spans a cyclic subgroup of order $n$ in $E(K)$ or, equivalently, that $P$ defines a closed immersion $(\Z/n)_S \to E$. Moreover, we call a group scheme over $S$ \emph{cyclic} if it is \'etale locally isomorphic to $(\Z/n)_S$.  

We can define compactifications $\MMb_1(n)$ and $\MMb_0(n)$ as the normalizations of $\MMb_{ell}$ in $\MM_1(n)$ and $\MM_0(n)$, respectively (for the definition of normalization see \cref{sec:Flatness}). The relevance for our purposes is that we have isomorphisms 
\begin{align*}\MF_k(\Gamma_1(n);R)&\cong H^0(\MMb_1(n)_R; \omline^{\tensor k}) \text{ and }\\
\MF_k(\Gamma_0(n);R) &\cong H^0(\MMb_0(n)_R; \omline^{\tensor k}).\end{align*}
for every $\Z[\frac1n]$-algebra $R$. The relevant case for us is the first one and is explained in \cref{sec:summary}. For $\Z[\frac1n]$-algebras $R$ that do not embed into $\C$, we will use the formula above as a definition of $\MF_k(\Gamma_1(n);R)$. In general, $\MF(\Gamma_1(n);R)$ differs from $\MF_1(n)\tensor R$, where we use $\MF_1(n)$ as an abbreviation for $\MF(\Gamma_1(n);\Z[\frac1n])$. But we have the following useful lemma.  

\begin{lemma}\label{lem:basechange}
Let $R \to S$ be a flat ring extension of $\Z[\tfrac{1}{n}]$-algebras. Then the canonical map $\MF(\Gamma_1(n);R) \tensor_R S \to \MF(\Gamma_1(n);S)$ is an isomorphism.
\end{lemma}
\begin{proof}
This is a variant of flat base change applied to the sheaf $\omline^{\tensor i}$ on $\MMb_1(n)$. If $\MMb_1(n)$ is a scheme, we can apply \cite[Lemma 5.2.26]{QingLiu} directly. For Deligne--Mumford stacks, the proof is the same using \'etale instead of Zariski coverings.
\end{proof}

To determine $\MF_k(\Gamma_1(n);R)$ for small $n$, we will use the following lemma.

\begin{lemma}
 The stack $\MMb_1(n)$ is equivalent to $\mathbb{P}^1_{\Z[\frac1n]}$ for $5\leq n\leq 10$ and $n=12$. For $n=7$, the line bundle $\omline$ corresponds under this equivalence to $\OO(2)$.
 \end{lemma}
\begin{proof}
 The first statement is proven in Section 2 of \cite{MeierDecomposition}. The Picard group of $\mathbb{P}^1_{\Z[\frac1n]}$ is isomorphic to $\Z$. Indeed, by \cite[Prop 6.5c]{Hartshorne}, we have a short exact sequence
 \[0 \to \Z \to \Pic \mathbb{P}^1_{\Z[\frac1n]} \to \Pic \mathbb{A}^1_{\Z[\frac1n]} \to 0,\]
 where the first map is split by degree and $\Pic \mathbb{A}^1_{\Z[\frac1n]} \cong \Pic \Z[\frac1n] = 0$ by \cite[Prop 6.6]{Hartshorne}. 
 
Thus, we have only to compute the degree of $\omline$ on $\MMb_1(7)$. In general, the degree of $\omline$ on $\MMb_1(n)$ is $\frac{1}{24}n^2\prod\limits_{p|n}(1-\frac{1}{p^2})$ \cite[Lemma 4.3]{MeierDecomposition}. So we conclude that for $n=7$ that the degree is $2$.  
\end{proof}

This directly implies that $\MF(\Gamma_1(7);\Z[\frac17])$ is isomorphic to the subalgebra of $\Z[\frac17][x,y]$ of polynomials of even degree. Our plan is to find explicit modular forms as generators and to use this to determine the action of $\left(\Z/7\right)^{\times}\cong\Z/6$ on $\MF(\Gamma_1(7); \Z[\frac17])$. A priori there are two different actions, one on the analytic side via an isomorphism $\Gamma_0(7)/\Gamma_1(7)\cong \left(\Z/7\right)^{\times}$ and one by the inverse of the $\left(\Z/7\right)^{\times}$-action on the torsion points of precise order $7$ in the modular interpretation. These actions and their comparison are discussed more precisely in \cref{rem:equivariance}. 

We will identify a basis of $\MF_1(\Gamma_1(7);\C)$ using Eisenstein series. According to \cite[Theorem 4.8.1]{DiamondShurman}, there is for every odd character $\varphi\colon \zsk \to \C^{\times}$ an Eisenstein series $E(\varphi)\in \MF_1(\Gamma_1(7);\C)$ and these are linearly independent. Fixing the generator $t = [3]$ in $\zsk$ induces an isomorphism $\zsk \cong \Z/6$. In terms of this generator, the three odd characters $\varphi_1,\varphi_2,\varphi_3\colon \zsk\to \C^{\times}$ are described by
\[
 \begin{aligned}
  \varphi_1(t)&=\zeta_6,\\
  \varphi_2(t)&=-1,\\
  \varphi_3(t)&=-\zeta_6+1,
 \end{aligned}
\]
where $\zeta_6=\exp(\frac{2\pi i}{6})$ is a sixth primitive root of unity. As $\dim_{\C}\MF_1(\Gamma_1(7);\C) = 3$ by our identification of the ring $\MF(\Gamma_1(7);\Z[\frac17])$ above, we conclude that $E(\varphi_1),\allowbreak E(\varphi_2),\allowbreak E(\varphi_3)$ are a basis of $\MF_1(\Gamma_1(7);\C)$. The $\zsk$-action on $E(\varphi)$ is described by the character $\varphi$. According to \cite[Section 4.8 and Formula (4.33)]{DiamondShurman} (or \cite{Buzzard}), the $q$-expansions of these Eisenstein series looks like follows:
\[
 E(\varphi_j)(\tau)=-\frac{1}{14}\sum_{n=1}^{6}n\varphi_j(n)+\sum_{k=1}^{\infty}\left(\sum_{l|k, l>0}\varphi_j(l)\right)q^k, \mbox{ with } q=\exp(2\pi i\tau).
\]
To actually obtain an integral basis of the weight $1$ modular forms, we consider the following three modular forms instead of the Eisenstein series.

\[
 \begin{aligned}
  z_1&=&\frac{1}{3}(3\zeta_6-1)E(\varphi_1)+\frac{2}{3}E(\varphi_2)+\frac{1}{3}(-3\zeta_6+2)E(\varphi_3),\\
  z_2&=& \frac{1}{3}(-\zeta_6-2)E(\varphi_1)+\frac{2}{3}E(\varphi_2)+\frac{1}{3}(\zeta_6-3)E(\varphi_3),\\
  z_3&=& \frac{1}{3}(-2\zeta_6+3)E(\varphi_1)+\frac{2}{3}E(\varphi_2)+\frac{1}{3}(2\zeta_6+1)E(\varphi_3).
 \end{aligned}
\]
Note that the base change matrix
\[
 \frac{1}{3}\begin{pmatrix}
             3\zeta_6-1& -\zeta_6-2&-2\zeta_6+3\\
	     2 & 2& 2\\
	    -3\zeta_6+2 & \zeta_6-3&2\zeta_6+1
            \end{pmatrix}
\]
has determinant $\frac{2}{27}\left(84\zeta_6-42\right)$, which is invertible in $\C$. Thus, $z_1, z_2, z_3$ form a new $\C$-basis of $\MF_1(\Gamma_1(7);\C)$. 

\begin{lemma}
 The $z_j$ have only integer coefficients in their $q$-expansion.
\end{lemma}
\begin{proof}
Denote the coefficient of $q^n$ in $z_j$ by $c_n(z_j)$. First, we compute $c_0(z_j)$. This calculation is somewhat different from the ones for higher coefficients:
\[
\begin{aligned}
 c_0(z_1)=&\frac{1}{3}(3\zeta_6-1)\cdot \left(-\frac{1}{14}\sum_{n=1}^{6}n\varphi_1(n)\right)+\frac{2}{3}\cdot \left(-\frac{1}{14}\sum_{n=1}^{6}n\varphi_2(n)\right)\\
 &+\frac{1}{3}(-3\zeta_6+2)\cdot \left(-\frac{1}{14}\sum_{n=1}^{6}n\varphi_3(n)\right)
 \end{aligned}
\]
Evaluating the sum for $\varphi_1$, we obtain
\[
 \sum_{n=1}^{6}n\varphi_1(n)=1+2\zeta_6^2+3\zeta_6+4\zeta_6^4+5\zeta_6^5+6\zeta_6^3 = -4\zeta_6-2,
\]
using that $\zeta_6^2=\zeta_6-1$ and $\zeta_6^3=-1$. 
For $\varphi_2$, we obtain 
\[
 \sum_{n=1}^{6}n\varphi_2(n)=1+2-3+4-5-6=-7.
\]
For $\varphi_3$, recall that $1-\zeta_6=\zeta_6^5$, so we obtain
\[
 \sum_{n=1}^{6}n\varphi_3(n)=1+2\zeta_6^4+3\zeta_6^5+4\zeta_6^2+5\zeta_6+6\zeta_6^3 = 4\zeta_6-6.
\]
Inserting this values into the formula for $c_0(z_1)$, we get
\[
\begin{aligned}
 c_0(z_1)=&\frac{1}{3}(3\zeta_6-1)\cdot \left(-\frac{1}{14}(-4\zeta_6-2)\right)+\frac{1}{3}\\
 &+\frac{1}{3}(-3\zeta_6+2)\cdot \left(-\frac{1}{14}(4\zeta_6-6)\right)=0.
 \end{aligned}
\]
Similarly, we obtain $c_0(z_2) = 0$ and $c_0(z_3) = 1$.

Now we will show that $c_k(z_j)$ for $k>0$ and $j\in\{1,2,3\}$ is always an integer. For $z_1$, we obtain
\[
\begin{aligned}
 c_k(z_1)&=&\frac{1}{3}(3\zeta_6-1)\cdot\left(\sum_{l|k, l>0}\varphi_1(l)\right)+\frac{2}{3}\cdot \left(\sum_{l|k, l>0}\varphi_2(l)\right)\\
 &&+ \frac{1}{3}(-3\zeta_6+2)\cdot\left(\sum_{l|k, l>0}\varphi_3(l)\right)\\
 &=&\sum_{l|k, l>0}\frac{1}{3}\left((3\zeta_6-1)\varphi_1(l)+2\varphi_2(l)+(-3\zeta_6+2)\varphi_3(l)\right).
 \end{aligned}
\]
where $l$ denotes also its congruence class in $\Z/7$. 

We give the values of the summands depending on $l$: 
\begin{center}
\begin{tabular}{c|c|c|c|c|c|c|c}
$l \mod 7$ & $0$ & $1$ & $2$ & $3$ & $4$ & $5$ & $6$\\ \hline
Summand & $0$ & $1$ & $-1$ & $-2$ & $2$ & $1$ & $-1$
\end{tabular}
\end{center}
 In particular, the sum we obtain has only integer summands, thus is itself an integer. The calculations for $c_k(z_2)$ and $c_k(z_3)$ are similar and we obtain $z_1,z_2,z_3\in \MF_1(\Gamma_1(7);\mathbb{Z})$. 

\end{proof}

We want to show that $z_1,z_2,z_3\in \MFq_1(\Gamma_1(7);\mathbb{Z})$ is a basis. For this, we consider the $q$-expansions of $z_1,z_2,z_3$ modulo $q^3$: 
\[
 \begin{aligned}
  z_1&\equiv& &&q& &\mod q^3\\
  z_2&\equiv& &&-q&+q^2&\mod q^3\\
  z_3&\equiv& 1&&+2q&+3q^2&\mod q^3
 \end{aligned}
\]
The right-hand sides form obviously a $\Z$-basis of $\Z\llbracket q\rrbracket/(q^3)$. Thus, $\MFq_1(\Gamma_1(7);\mathbb{Z}) \to \Z\llbracket q\rrbracket/(q^3)$ is surjective. Since the composite
\[\MF_1(\Gamma_1(7);\Z) \to \MF_1(\Gamma_1(7); \C) \to \C\llbracket q \rrbracket /(q^3)\]
is as a composition of an injection and a surjection between $3$-dimensional $\C$-vector spaces an injection as well, we conclude that $\MFq_1(\Gamma_1(7);\mathbb{Z}) \to \Z\llbracket q\rrbracket/(q^3)$ is actually an isomorphism. This implies that $z_1,z_2,z_3\in \MFq_1(\Gamma_1(7);\mathbb{Z})$ is indeed a basis. The same is thus true in $\MF_1(\Gamma_1(7);\mathbb{Z}[\frac17])$.

Our next goal is to understand all of $\mf$ in terms of $z_i$'s. We will prove the following proposition:
\begin{Prop}
\label{prop:IdentifyMF}
 There is an isomorphism of rings
\[
 \mathbb{Z}\left[\tfrac{1}{7}\right][z_1, z_2, z_3]/(z_1z_2+z_2z_3+z_3z_1)\to \MF(\Gamma_1(7); \mathbb{Z}\left[\tfrac{1}{7}\right]).
\]
\end{Prop}

\begin{pf} 
 We will first show using $q$-expansions that the relation $z_1z_2+z_2z_3+z_3z_1=0$ is satisfied in $\MFq_2(\Gamma_1(7); \Z)$. Recall from above that $z_3$, $z_1$ and $z_1+z_2$ are modular forms whose $q$-expansions begin with $1$, $q$ and $q^2$, respectively. Thus, the $q$-expansions of the modular forms
 \begin{equation}\label{eq:zbasis}
     z_3^2, z_1z_3, z_3(z_1+z_2), z_1(z_1+z_2)\text{ and }(z_1+z_2)^2
 \end{equation} 
 begin with $1, q, q^2, q^3$ and $q^4$, respectively. As by the dimension formulae from \cite[Section 3.9]{DiamondShurman} the vector space $\MF_2(\Gamma_1(7); \C)$ has dimension $5$, we immediately see that the modular forms in \eqref{eq:zbasis} form a basis of $\MF_2(\Gamma_1(7); \C)$ and that hence a modular form in $\MF_2(\Gamma_1(7); \C)$ is determined by its $q$-expansion modulo $q^5$. Knowing the $q$-expansions of the $z_i$ modulo $q^3$, determines the $q$-expansion of $z_1z_2+z_2z_3+z_3z_1$ modulo $q^5$ and one deduces easily that it is actually zero modulo $q^5$. 

We have noted before that 
\[\MF(\Gamma_1(7); \Z[\tfrac{1}{7}]) \cong H^0(\MMb_1(7); \omline^{\tensor *}) \cong H^0(\mathbb{P}^1_{\Z[\tfrac17]}; \OO(2*)).\]
 Thus, this ring is abstractly isomorphic to polynomials of even degree in two variables of degree $1$ (and the degree of the modular form is half the degree of the polynomial). The ring of such polynomials is generated by the three monomials of degree $2$ with one quadratic relation between those. Thus, the ring $\MF(\Gamma_1(7); \Z\left[\tfrac17\right])$ is generated in degree $1$ and hence by the elements $z_1,z_2, z_3$. We obtain a surjective map 
$$\mathbb{Z}[\tfrac17][z_1, z_2, z_3]/(z_1z_2+z_2z_3+z_3z_1)\to \MF(\Gamma_1(7); \mathbb{Z}[\tfrac17]),$$
which has to be an isomorphism as counting ranks shows. 
\end{pf}

Next, we want to identify the $(\Z/7)^{\times}$ action on the left-hand side in terms of the generator $t = [3] \in (\Z/7)^\times$, where we use the conventions from \cref{rem:equivariance}. 
\begin{lemma}
 The action of $(\Z/7)^\times$ on $\MF(\Gamma_1(7); \mathbb{Z}[\tfrac17])$ is given by $t.z_1 = -z_3$ and $t.z_2 = -z_1$ and $t.z_3 = -z_2$. 
\end{lemma}
\begin{proof}
 From the known action on the Eisenstein series, so we calculate as follows:
\[
 \begin{aligned}
  t.z_1&=\frac{1}{3}(3\zeta_6-1)\varphi_1(t)E(\varphi_1)+\frac{2}{3}\varphi_2(t)E(\varphi_2)+\frac{1}{3}(-3\zeta_6+2)\varphi_3(t)E(\varphi_3)\\
&= \frac{1}{3}(3\zeta_6-1)\zeta_6 E(\varphi_1)-\frac{2}{3}E(\varphi_2)+\frac{1}{3}(-3\zeta_6+2)(-\zeta_6+1)E(\varphi_3)\\
&= \frac{1}{3}(2\zeta_6-3)E(\varphi_1)-\frac{2}{3}E(\varphi_2)-\frac{1}{3}(2\zeta_6+1)E(\varphi_3)\\
&= -z_3.
 \end{aligned}
\]
Similarly, one obtains $t.z_2 = -z_1$ and $t.z_3 = -z_2$. 
\end{proof}
Note that the resulting action on $\Z[z_1,z_2,z_3]$ makes it isomorphic as a $\Z/6\cong \Z/2\times\Z/3$-representation to $\Z^{\mbox{sign}}\otimes \Z[z_1,z_2,z_3]$, where $\Z^{\mbox{sign}}$ is the sign representation of $\Z/2$ and $\Z/3$ acts on $\Z[z_1,z_2,z_3]$ by permuting the variables as indicated above. We record without proof the following consequence. 

\begin{prop} \label{Z6invariants}
If we denote by $\sigma_1,\sigma_2,\sigma_3 \in \Z[z_1,z_2,z_3]$ the elementary symmetric polynomials, we obtain that the invariants 
\[\MF(\Gamma_0(7);\Z\left[\tfrac17\right]) \cong H^0(\Z/6, \Z[z_1,z_2,z_3]/\sigma_2)\] 
are the even degrees of the free $\Z[\sigma_1,\sigma_3]$-module on $1$ and $z_1^2z_2+z_2^2z_3+z_3^2z_1$. As explained in \cref{sec:analytic}, this implies 
\[\MFC(\Gamma_0(7);\Z\left[\tfrac17\right]) \cong H^0(\Z/6, \Z[z_1,z_2,z_3]/\sigma_2)[\Delta^{-1}].\] 
\end{prop}


\section{$q$-expansions of the coefficients of universal elliptic curves}
\label{sec:TateCurve}
The aim of this section is to obtain $q$-expansions for the coefficients of Weierstra\ss{} equations for the universal elliptic curves with a $\Gamma_1(n)$-level structure. It is easy to see that for such curves there exists a Weierstra\ss{} equation whose coefficients are meromorphic $\Gamma_1(n)$-modular forms. Our computations show in particular that we can find a Weiersra\ss{} equation in Tate normal form whose coefficients are indeed holomorphic. In the special case of $n=7$, this will allows us to identify them with polynomials in the modular forms $z_1, z_2, z_3$ from the previous section. 

\subsection{Coordinates for generalized elliptic curves}
In this section, we will need some results about Weierstra{\ss} equations for generalized elliptic curves in the sense of \cite[Definition II.1.12]{DeligneRapoport}. Actually for the main results of this section, \cref{prop:holomorphicalphas} and \cref{Alphas}, we could work with the usual smooth elliptic curves instead, but first of all we will show here a similar statement to \cref{prop:holomorphicalphas} without the necessity of calculations with $q$-expansion and furthermore we will need the specific form of \cref{prop:coordinatesm1n} in the proof of \cref{prop:normal}. In the following theorem we will summarize the necessary input from \cite[\S 1]{dapres}.

\begin{Theorem}\label{thm:coordinates}
Let $S$ be a scheme and let $p\colon \EE \to S$ be a generalized elliptic curve whose geometric fibers are either smooth or N\'eron $1$-gons. Let furthermore $(e)$ be the relative Cartier divisor defined by the unit section $e\colon S \to \EE$.
\begin{enumerate}[label=\alph*)]
    \item The sheaves $p_*\OO_{\EE}(ke)$ are locally free of rank $k$ for $k>0$ and are related by short exact sequences 
\begin{align}\label{eq:omegashort}
    0 \to p_*\OO_{\EE}((k-1)e) \to p_*\OO_{\EE}(ke) \to \omega_{\EE}^{\tensor (-k)} \to 0
\end{align}
for $k>1$. The morphisms $\OO_S \to p_*\OO_{\EE} \to p_*\OO_{\EE}(e)$ are isomorphisms. Moreover, $R^1p_*\OO(ke) = 0$ for $k>0$.
\item Zariski locally one can choose a trivialization of $\omega_{\EE}$ and splittings of \eqref{eq:omegashort} for $k=2$ and $3$ and these choices define a trivialization 
\[(1,x,y)\colon \OO_S^3 \xrightarrow{\cong} p_*\OO(3e).\]
This gives rise to a closed embedding $\EE \to \P^2_S$ with image cut out by a cubic equation 
\[y^2+a_1xy+a_3y = x^3+a_2x^2+a_4x+a_6
\]
with $a_1,a_2,a_3,a_4,a_6\in \Gamma(\OO_S)$ and the $a_i$ are completely determined by the previous choices.
\end{enumerate}
\end{Theorem} 

\begin{remark}\label{rem:grading}
We want to discuss the role of gradings in the preceding theorem. Let $p\colon \EE \to S$ be a generalized elliptic curve as in the theorem with chosen splittings of \eqref{eq:omegashort} for $k=2$ and $k=3$. This yields an isomorphism
\[p_*\OO_{\EE}(3e) \cong \OO_S \oplus \omega_{\EE}^{\tensor (-2)} \oplus \omega_{\EE}^{\tensor (-3)}.\]
Let $q\colon T \to S$ be a morphism with a trivialization $\OO_T \xrightarrow{\omega} \omega_{\EE_T}$, where $p^{T}\colon \EE_T \to T$ is the pullback of $p$ along $q$. By \cite[Proposition 4.37]{VistoliDescent}, the natural map $q^*p_*\OO_{\EE}(3e) \to p^T_*\OO_{\EE_T}(3e)$ is an isomorphism. The resulting isomorphism $\OO_T^3 \to p^T_*\OO_{\EE_T}(3e)$ sends the standard basis to elements $(1,x,y)$ and we get associated quantities $a_i \in \Gamma(\OO_T)$. Changing the trivialization to $\lambda\omega$ for some $\lambda \in \GG_m(T)$ produces new coordinates $(1,x',y') = (1, \lambda^{-2}x, \lambda^{-3}y)$ with associated quantities $a_i' = \lambda^{-i}a_i$. 

In particular, we can consider the $\GG_m$-torsor $q\colon T \to S$ given by the relative spectrum $T = \underline{\Spec}\left(\bigoplus_{i\in\Z}\omega_{\EE}^{\tensor i}\right)$.  This comes with a canonical trivialization of $q^*\omega_{\EE}$ and thus by \eqref{eq:canonicaliso} also of $\omega_{\EE_T}$; this produces elements $a_i \in \Gamma(\OO_T)$. The computation above shows that the degree of $a_i$ is $i$, where the grading on $\Gamma(\OO_T)$ comes from the $\GG_m$-action on $T$ or equivalently from the identification with $\left(\bigoplus_{i\in\Z} \omega_{\EE}^{\tensor i}\right)(S)$.
\end{remark}

The following standard fact will be needed for the proof of \cref{prop:coordinatesm1n}.
\begin{lemma}\label{lem:separated}
Let $R$ be a nonnegatively graded ring, $Z$ the vanishing locus of the ideal generated by the positive degree homogeneous elements and $U$ its complement in $\Spec R$. Then $U/\GG_m$ is separated. 
\end{lemma}
\begin{proof}
By the valuative criterion it suffices to show that for every valuation ring $V$ with field of fractions $K$, the map $p_V\colon (U/\GG_m)(V) \to (U/\GG_m)(K)$ of groupoids is fully faithful \cite[Proposition 7.8]{Laumon}. As every $\GG_m$-torsor over $V$ is trivial, the groupoid $(U/\GG_m)(V)$ can be up to equivalence described as follows: It has as objects $\mathbb{G}_m$-equivariant maps $\Spec V \times \mathbb{G}_m \to U$ and morphisms are $\mathbb{G}_m$-equivariant endomorphisms of $\Spec V\times \mathbb{G}_m$ over $\Spec V$ and $U$. More concretely, its objects can be described as ring morphisms $f\colon R \to V$ such that $f(r)$ is invertible for some $r$ homogeneous of positive degree. A morphism $g \to f$ is given in this language by an element $\lambda\in \GG_m(V)$ such that $g(r) = \lambda^if(r)$ for all homogeneous $r$ of degree $i$. The description of $(U/\GG_m)(K)$ is analogous. As $\GG_m(V)$ includes into $\GG_m(K)$, we see that $p_V$ is faithful. Now suppose that $f,g\colon R \to V$ are two objects in $(U/\GG_m)(V)$ and $\lambda \in \GG_m(K)$ satisfies $g(r) = \lambda^if(r)$ for all homogeneous $r$ of degree $i$. We can choose an $r\in R$ such that $f(r)$ is invertible in $V$ and thus $\lambda^i \in V$ and hence $\lambda \in V$ as $V$ is normal. Repeating this argument for an $r' \in R$ such that $g(r')$ is invertible yields that $\lambda^{-1}\in V$ and hence $\lambda \in \GG_m(V)$. Thus $p_V$ is full. 
\end{proof}

\begin{prop}\label{prop:coordinatesm1n}
Let $n\geq 2$ and
\[\pi\colon \MMb_1^1(n) = \underline{\Spec}\left(\bigoplus_{i\in\Z}\omline^{\tensor i}\right) \to \MMb_1(n)\]
be the $\GG_m$-torsor trivializing $\omline$. Let $p\colon \EE \to \MMb_1(n)$ be the generalized elliptic curve classified by the map $\MMb_1(n) \to \MMb_{ell}$ that is induced by the forgetful map $\MM_1(n) \to \MM_{ell}$.

Then there are
\[a_1, a_2, a_3, a_4, a_6\in \MF_1(n) = H^0(\MMb_1^1(n), \OO_{\MMb_1^1(n)})\]
with $|a_i| = i$ such that the generalized elliptic curve $\pi^*\EE$ is defined by the cubic equation
\begin{equation}\label{eq:cubic} y^2 +a_1xy +a_3y = x^3+a_2x^2+a_4x + a_6.\end{equation}
Consider moreover the canonical morphism $\MMb_1^1(n) \to \Spec \MF_1(n)$. It is an open immersion onto the complement of the common vanishing locus of $c_4$ and $\Delta$ for the usual quantities $c_4$ and $\Delta$ associated with $a_1,\dots, a_6$. Moreover, $c_4, c_6$ and $\Delta$ coincide with the images of the classes with the same name along $\MF_1(1) \to \MF_1(n)$.
\end{prop}
\begin{proof}
To apply \cref{thm:coordinates} and \cref{rem:grading}, it suffices to show that the inclusions 
\[p_*\OO_{\EE}((k-1)e) \to p_*\OO_{\EE}(ke)\]
split for $k>1$. By induction, we can assume that $p_*\OO_{\EE}((k-1)e)$ is via such splittings isomorphic to $\OO_{\MMb_1(n)} \oplus \omline^{\tensor (-2)} \oplus \cdots \oplus \omline^{\tensor (-k+1)}$. The vector bundle $p_*\OO_{\EE}(ke)$ is an extension of $p_*\OO_{\EE}((k-1)e)$ and $\omline^{\tensor (-k)}$ and thus we have to show the vanishing of a class $\chi$ in 
\[\Ext_{\OO_{\MMb_1(n)}}(\omline^{\tensor (-k)}, p_*\OO_{\EE}((k-1)e)) \cong H^1(\MMb_1(n); \omline^{\tensor k} \oplus \omline^{\tensor (k-2)} \oplus \cdots \oplus \omline)\]
The vanishing result \cite[Proposition 2.14]{MeierDecomposition} implies that the projection to 
\[\Ext_{\OO_{\MMb_1(n)}}(\omline^{\tensor (-k)}, p_*\OO_{\EE}((k-1)e)/p_*\OO_{\EE}((k-2)e)) \cong H^1(\MMb_1(n); \omline)\]
is an isomorphism and thus it suffices to show the vanishing of the projection $\chi'$ of $\chi$.
As $\EE$ is the pullback of the universal generalized elliptic curve $\EE^{uni} \to \MMb_{ell}$ along $\MMb_1(n) \to \MMb_{ell}$, the class $\chi'$  actually lies in the image from  
\[\Ext_{\OO_{\MMb_{ell}}}(\omline^{\tensor (-k)}, p_*\OO_{\EE^{uni}}((k-1)e)/p_*\OO_{\EE^{uni}}((k-2)e)) \cong H^1(\MMb_{ell}; \omline).\]
As shown in \cite[Proposition 2.16]{MeierDecomposition}, the map $H^1(\MMb_{ell};\omline) \to H^1(\MMb_1(n); \omline)$ is trivial and thus the inclusions $p_*\OO_{\EE}((k-1)e) \to p_*\OO_{\EE}(ke)$ are split for $k>1$. We see that $\pi^*\EE$ is indeed given by a cubic equation as in \eqref{eq:cubic} with $|a_i|=i$.

If $S \to \MMb_{ell}$ classifies a generalized elliptic curve $E\to S$ given by a cubic equation as in \eqref{eq:cubic}, then the images of $c_4,\Delta \in H^0(\MMb_{ell};\omline^{\tensor *})$ in $H^0(S; \omega_E^{\tensor *})$ coincide with the corresponding polynomials in the coefficients $a_i$ of the Weierstra{\ss} equation; thus $c_4,\Delta\in H^0(S;\OO_S)$ have an unambiguous meaning. Moreover, a curve on $S$ defined by a cubic equation as in \eqref{eq:cubic} defines a generalized elliptic curve only if $c_4$ and $\Delta$ vanish nowhere simultaneously \cite[Proposition III.1.4]{SilvermanAEC}. 

As $\omline$ is ample on $\MMb_1(n)$ by \cite[Lemma 5.11]{MeierTMFLevel}, the pullback $\pi^*\omline$ on $\MMb_1^1(n)$ is both trivial and ample and thus $\MMb_1^1(n)$ is quasi-affine, i.e.\ the canonical morphism $\MMb_1^1(n) \to \Spec H^0(\MMb_1^1(n), \OO_{\MMb_1^1(n)}) = \Spec \MF_1(n)$ is an open immersion \cite[Propositions 13.83 and 13.80]{GoertzWedhorn}. Moreover, $\pi^*\EE$ being a generalized elliptic curve implies that the immersion $\MMb_1^1(n)\to \Spec \MF_1(n)$ has image in the complement $U$ of the common vanishing locus of $c_4$ and $\Delta$. Note that the inclusion $\MMb_1^1(n) \to U$ is $\GG_m$-equivariant. 

By \cref{lem:separated}, $U/\GG_m$ is separated. As $\MMb_1(n)$ is proper over $\Spec \Z[\frac1n]$ and $U/\GG_m$ is separated, we obtain analogously to \cite[Corollary II.4.6]{Hartshorne} that the open immersion $\MMb_1(n) \hookrightarrow U/\GG_m$ is proper. Thus, the image is closed. As $\MF_1(n)$ is an integral domain by \cite{MeierDecomposition}, we see tha $U$ and hence $U/\mathbb{G}_m$ are connected. We deduce that $\MMb_1(n) \hookrightarrow U/\GG_m$ is an isomorphism. 
\end{proof}

At least over a field, the following lemma is well-known.

\begin{lemma}
Let $\EE \to S:=\Spec R$ be an elliptic curve given by a Weierstra{\ss} equation, where $R$ is graded. Assume that the degrees of the associated quantities $a_i \in \Gamma(\OO_S)=R$ have degree $i$. Let there furthermore be a section $P\colon S \to \EE$ of exact order $n \geq 3$. Then there are coordinates for $\EE$ such that the associated Weierstra{\ss} equation is of the form 
\[y^2+\alpha_1xy + \alpha_3y = x^3 +\alpha_2x^2,\]
$P$ corresponds to the point $(0,0)$ and $\lvert\alpha_i\rvert = i$.
\end{lemma}

This special form of the Weierstra{\ss} equation is called \emph{Tate normal form} or also sometimes \emph{Kubert--Tate normal form}. 

\begin{proof}
We perform a similar transformation as in the proof of \cite[Theorem 1.1.1]{BehrensOrmsby}. First, observe from formulae in \cite[Section III.1]{SilvermanAEC} that any Weierstra\ss{} equation of the form
\[
 y^2+a_1xy+a_3x=x^3+a_2x^2+a_4x+a_6
\]
can be transformed so that the chosen torsion point $(x_0,y_0)$ on this curve is moving to $(0,0)$. This transformation has the transformation parameter 
\[
 \begin{aligned}
  r= x_0, \quad
  t= y_0, \quad
  s= 0. 
 \end{aligned}
\]

Thus we may assume $a_6=0$ and the torsion point to be $(0,0)$. Over any field $K$, it follows from \cite[Remark 4.2.1]{HusemollerEllCurves} that if $a_3=0$ over this field, then $(0,0)$ would be either singular or a $2$-torsion point, contradicting the assumption that it is a torsion point of strict order $n\geq 3$. Thus, $a_3$ is invertible in $R$ since it maps to a non-zero element in every field for any ring map $R \to K$.

This allows to define a transformation with transformation parameters
\[
 \begin{aligned}
  r= 0, \quad
  t= 0, \quad
  s= \frac{a_4}{a_3}, 
 \end{aligned}
\]
and the resulting coefficients are 
\[
y^2+\alpha_1xy + \alpha_3y = x^3 +\alpha_2x^2.
\] 
\end{proof}

\begin{remark}\label{rmk:generalTateNF}
If we start with the Weierstra\ss{} equation $ y^2+a_1xy+a_3x=x^3+a_2x^2+a_4x+a_6$ as above, we want to record for later use the values of $\alpha_i$ obtained by the procedure in the proof of the lemma above. Denote by $s'$ the auxiliary quantity (the invertibility of the denominator follows from the lemma above)
\[
s'=\frac{a_4+2a_2x_0-a_1y_0+3x_0^2}{a_3+a_1x_0+2y_0}.
\]
Then we obtain
\begin{alignat*}{2}
\alpha_1&=&&\frac{a_1a_3+2a_4+(a_1^2+4a_2)x_0+6x_0^2}{a_3+a_1x_0+2y_0}\\
\alpha_2&=&&a_2+3x_0-a_1s'-(s')^2\\
\alpha_3&=&&a_3+a_1x_0+2y_0.
\end{alignat*}
\end{remark}

\subsection{The $q$-expansions of the coefficients of Weierstras{\ss} equations}
\label{subsec:qAlphai}
In the last subsection, we showed that the universal elliptic curve over $\MM_1^1(n)$ has a Weierstra{\ss} equation in Tate normal form and that the corresponding quantities $\alpha_i$ are elements in \[H^0(\MM_1^1(n);\OO_{\MM_1^1(n)})\cong \bigoplus_{i\in\Z}H^0(\MM_1(n);\omline^{\tensor i})\]
of degree $i$, i.e.\ meromorphic modular forms of weight $i$. In this subsection, we will show that the $\alpha_i$ are actually \emph{holomorphic} modular forms and provide a general formula for the $q$-expansion. This will allow us to identify them with explicit polynomials in the $z_i$ in the case $n=7$. Our first goal will be to prove a criterion to check whether a meromorphic $\Gamma_1(n)$-modular form is actually holomorphic. 

Recall from \cref{sec:analytic} that meromorphic $\Gamma_1(n)$-modular form $g$ of weight $k$ (in the analytic sense) is holomorphic if 
for every $\gamma=\begin{pmatrix}a & b\\ c &d\end{pmatrix}\in \SL_2(\Z)$, the map $g[\gamma]_k$, given by 
\[
z \mapsto (cz+d)^{-k}g\left(\frac{az+b}{cz+d}\right)
\]
is holomorphic at $\infty$. 

Similarly to e.g.\ \cite[Section 4]{DiamondIm}, we use certain operators $[W_m]$ to rephrase holomorphy of a $\Gamma_1(n)$-modular form. We will see in the proof of \cref{lem:holomorphicitycriterion} that these operators are closely related to the equivalence $\varphi\colon \MM_1(n) \to \MM_{\mu}(n)$ in \cref{lem:phiequivalence} (see also \cite[Section VII.6]{LangMod}). 
\begin{lemma}\label{lem:Atkin}
Let $g\colon \mathbb{H}\to \mathbb{C}$ be a meromorphic $\Gamma_1(n)$-modular form of weight $k$. Then the function
\begin{alignat*}{2}
W_ng:=g[W_n]_k\colon&\mathbb{H}\to \mathbb{C}\\
 &\tau \mapsto (n\tau)^{-k}g\left(-\frac{1}{n\tau}\right)
\end{alignat*}
is is a meromorphic $\Gamma_1(n)$-modular form again, where $W_n=\begin{pmatrix}0 &-1\\n &\hphantom{-}0\end{pmatrix}$. Moreover, $g$ is holomorphic as a modular form if and only if $W_ng$ is holomorphic as a modular form. 
\end{lemma}
\begin{pf}
  Given any matrix $\gamma=\begin{pmatrix}a & b\\ c &d\end{pmatrix}\in\Gamma_1(n)$, observe that its conjugate 
  \[W_n\gamma W_n^{-1}=\begin{pmatrix} \hphantom{-}d & -\frac{c}{n}\\ -bn &\hphantom{-}a\end{pmatrix} = W_n^{-1}\gamma W_n\]
  also lies in $\Gamma_1(n)$ and hence $W_n^{-1}\Gamma_1(n)W_n = \Gamma_1(n)$. Thus our lemma becomes a special case of \cite[Exercise 1.2.11(c)]{DiamondShurman}.
\end{pf}

For the next lemma recall the Tate curve $\Tate(q^n)$ from the discussion after \cref{TateWeierstrass}, as well as the description of torsion points on this curve. 
\begin{lemma}\label{lem:holomorphicitycriterion}
Let $g \in \Nat_k^{\C}(\Ell^1_{\Gamma_1(n)}(-), \Gamma(-))$ be a Katz modular form over $\mathbb{C}$ and let $\beta_1(g)\in \MFC(\Gamma_1(n);\C)$ be its associated modular form as in \cref{sec:naive}. Assume that the evaluation at the Tate curve $\Tate(q^n)$ with its invariant differential $\eta^{can}$ over $\Conv_{q^n}$ yields a power series (as opposed to a general Laurent series) for any choice of torsion point $(X(\zeta^dq^{c}, q^{n}), Y(\zeta^dq^{c}, q^{n}))$ with $\zeta = e^{\frac{2\pi i}n}$. Then $\beta_1(g)$ is actually holomorphic. 
\end{lemma}
\begin{proof}
Throughout this proof, let $\tau$ be an arbitrary point in the upper half-plane. 
By definition, $\beta_1(g)(\tau) = g(\C/\Z+n\tau\Z, dz, \tau)$. We claim that 
\[(W_n\beta_1(g))(\tau) = (-n)^{-k}g(\C/\Z+\tau\Z, dz, \frac1n).\] Indeed, multiplication by $-\tau$ induces an isomorphism from the elliptic curve $\C/\Z+n\frac{-1}{\hphantom{-}n\tau}\Z$ to $\C/\Z+\tau\Z$. Thus, $(\C/\Z+n\frac{-1}{\hphantom{-}n\tau}\Z, dz, \frac{-1}{n\tau})$ is isomorphic to $(\C/\Z+\tau\Z, (-\tau)^{-1}dz, \frac1n)$. We obtain 
\[(W_n\beta_1(g))(\tau) = (n\tau)^{-k}\beta_1(g)\left(-\frac{1}{n\tau}\right) = (n\tau)^{-k}(-\tau)^kg\left(\C/\Z+\tau\Z, dz, \frac1n\right)
\]
as was to be shown.

The assignment $g'\colon \mathbb{H} \to \mathbb{C}$, given by $\tau\mapsto g(\C/\Z+\tau\Z, dz, \frac1n)$, can be checked to define a  meromorphic $\Gamma_1(n)$-modular form. In particular, it suffices by \cref{lem:Atkin} and our previous computation to show that $g'$ is holomorphic. By definition, this means that $g'[\gamma]_k$ are holomorphic at $\infty$ for every $\gamma \in \SL_2(\Z)$. If we write $\gamma = \begin{pmatrix}a&b\\c&d\end{pmatrix}$ for such a $\gamma$, we have 
\begin{align*}g'[\gamma]_k(\tau) &= (c\tau+d)^{-k}g\left(\C/\Z+\frac{a\tau+b}{c\tau+d}\Z, dz, \frac1n\right) \\
&= g\left(\C/\Z+\frac{a\tau+b}{c\tau+d}\Z, (c\tau+d)dz, \frac1n\right) \\
&= g\left(\C/\Z+\tau\Z, dz, \frac{c\tau+d}n\right)\\
&= g\left(\C^\times/q_0^{n\Z}, \frac{du}u, \zeta^dq_0^c\right)
\end{align*}
with $q_0 = e^{\frac{2\pi i\tau}n}$.  
If we push the Tate curve $\Tate(q^n)$ forward along the evaluation map $\ev_{q_0}\colon \Conv_{q^n}\to \mathbb{C}$, the resulting elliptic curve with chosen torsion point is isomorphic to $\left(\C^\times/q_0^{n\Z}, \frac{du}u, \zeta^dq_0^c\right)$. By naturality, the values of $g$ are also related via $\ev_{q_0}$. 
Because $g(E_{q^n}, \eta^{can}, (X(\zeta^dq^{c}, q^{n}), Y(\zeta^dq^{c}, q^{n})))$ is a power series in $q^n$ by assumption, we see thus that $g'[\gamma]$ is holomorphic at $\infty$.
\end{proof}

From now on we will assume $n\geq 3$. Our aim now is to compute the $q$-expansions of the coefficients $\alpha_i$ of the Tate normal form 
\[
 y^2+\alpha_1xy+\alpha_3y=x^3+\alpha_2x^2
\]
for the Tate curve $\Tate(q^n)$ given by $y^2+xy=x^3+a_4(q^n)x+a_6(q^n)$ over $\Conv_{q^n}$ with a chosen $n$-torsion point $(x_0,y_0)$. 
 The values from \cref{rmk:generalTateNF} specialize to 
\[
 \begin{aligned}
  s'&=&& \frac{a_4(q^n)-y_0+3x_0^2}{x_0+2y_0},\\
  \alpha_1&=&& \frac{x_0+6x_0^2+2a_4(q^n)}{x_0+2y_0},\\
  \alpha_2&=&& 3x_0-s'-(s')^2,\\
  \alpha_3&=&& x_0+2y_0.
 \end{aligned}
\] 

Next, we want to specify the torsion point $(x_0,y_0)$ on the Tate curve $\Tate(q^n)$. 
We use methods from \cite[Section V.3]{SilvermanAdvanced}, to simplify the expressions for $X(vq^k,q^n)$ and $Y(vq^k,q^n)$ in our case, where $v$ and $q$ are complex numbers with $|v| = 1$ and $|q|<1$ and $0\leq k <n$. First, we reindex the sum over positive natural numbers:
\[
 \begin{aligned}
  X(vq^k,q^n)&=&&\sum_{m\in\Z} \frac{vq^{mn+k}}{(1-vq^{mn+k})^2}-2s_1(q^n),\\
&=&& \frac{vq^k}{(1-vq^k)^2}+\sum_{m\geq 1} \left( \frac{vq^{mn+k}}{(1-vq^{mn+k})^2}+ \frac{v^{-1}q^{mn-k}}{(1-v^{-1}q^{mn-k})^2}\right.\\
&&&\left.-2 \frac{q^{mn}}{(1-q^{mn})^2}\right).
 \end{aligned}
\]
Recall the following formulae for $|x|<1$, obtained e.g.\ by differentiating the geometric series:
\[
 \frac{x}{(1-x)^2} = \sum_{l\geq 1} lx^l\mbox{ and }\frac{x^2}{(1-x)^3}=\sum_{l\geq 1} \frac{l(l-1)}{2} x^l \mbox{ and }\frac{x}{(1-x)^3}=\sum_{l\geq 0}\frac{l(l+1)}{2} x^l .
\]

\noindent Inserting this into the expression for $X(vq^k,q^n)$, we obtain for $k>0$
\[
  X(vq^k,q^n)= \sum_{l\geq 1}lv^lq^{kl}+\sum_{m\geq 1}\sum_{l\geq 1} \left( lv^lq^{(mn+k)l}+ lv^{-l}q^{(mn-k)l} 
 -2 lq^{mnl}\right). 
\]

\noindent For $k=0$ and $v\neq 1$ we obtain similarly
 \begin{align*}
    X(v,q^n)
      = \frac{v}{(1-v)^2} + \sum_{m>0 } \left(\sum_{l|m} l(v^l+v^{-l}-2)\right)q^{mn}.
\end{align*}

\noindent For $Y(vq^k,q^n)$ we get
\[
 \begin{aligned}
  Y(vq^k,q^n)&=&&\sum_{m\in\Z} \frac{v^2q^{2(mn+k)}}{(1-vq^{mn+k})^3}+s_1(q^n),\\
&=&& \frac{v^2q^{2k}}{(1-vq^k)^3}+\sum_{m\geq 1} \left( \frac{v^2q^{2(mn+k)}}{(1-vq^{mn+k})^3}- \frac{v^{-1}q^{mn-k}}{(1-v^{-1}q^{mn-k})^3}\right.\\
&&&\left.\vphantom{\frac{v^2q^{2(mn+k)}}{(1-vq^{mn+k})^3}}+ \frac{q^{mn}}{(1-q^{mn})^2}\right).
 \end{aligned}
\]

\noindent Using again the formulae derived from geometric series, we obtain for $k>0$
\[
 \begin{aligned}
  Y(vq^k,q^n) &=&& \sum_{l\geq 2}\frac{(l-1)l}{2}v^lq^{kl}+\sum_{m\geq 1}\sum_{l\geq 1} \left( \frac{(l-1)l}{2}v^lq^{(mn+k)l} \right.\\
&&&\left.
- \frac{l(l+1)}{2}v^{-l}q^{(mn-k)l}+ lq^{mnl}\right).
 \end{aligned}
\]
For $k=0$ and $v\neq 1$, we obtain instead
\begin{align*}
    Y(v,q^n)
          = \frac{v^2}{(1-v)^3} + \sum_{m>0 } \left(\sum_{l|m} \left(\frac{l(l-1)}{2}v^l-\frac{l(l+1)}{2}v^{-l}
         +1\right)\right)q^{mn}.
\end{align*}

\begin{lemma}\label{lem:powerseries}
As before let $|q|<1$, $|v| =1$ and $0\leq k <n$. In terms of the $X(vq^k, q^n)$ and $Y(vq^k, q^n)$ the Laurent series $\alpha_1, \alpha_2$ and $\alpha_3$ are actually power series as well if we assume $v\neq \pm 1$ in case that $k=0$ or $k = \frac{n}2$.
\end{lemma}
\begin{proof}
Note that in each of the cases above both $X(vq^k, q^n)$ and $Y(vq^k, q^n)$ are not just Laurent series in $q$, but actually power series.
 In particular, so is $\alpha_3=X+2Y$. Given the expressions for $\alpha_1$ and $\alpha_2$, we only need to check that 
\[
   s'= \frac{a_4(q^n)-X(vq^k, q^n)+3X(vq^k, q^n)^2}{X(vq^k, q^n)+2Y(vq^k, q^n)}
\]
is a power series to obtain that $\alpha_1$ and $\alpha_2$ are power series as well. In our Tate curve, we have 
\[
  a_4(q^n)= -5s_3(q^n) = -5 \sum_{m\geq 1}\sigma_3(m)q^{mn}
\]
so this power series has $n>k$ as lowest exponent of $q$. Thus, the lowest power of $q$ occuring in the numerator is the same as for $X$ (unless the numerator is $0$ and thus $s'=0$). It thus suffices to show that the lowest term of the power series for $X$ has at least the order of the lowest term of the power series defining $X+2Y$. In the table below we will compute the lowest term in the power series defining $X$, $Y$ and $X+2Y$ in the different cases.

\begin{center}
\begin{tabular}{|c|c|c|c|}
\hline
 & $X$ & $Y$ & $X+2Y$ \\[0.5ex] \hline
  &&&\\ 
$k=0$ & $\frac{v}{(1-v)^2}$ & $\frac{v^2}{(1-v)^3}$ & $\frac{v+v^2}{(1-v)^3}$ \\[3ex] \hline 
  &&&\\
$0<k<\frac{n}{2}$ & $vq^k$ & higher term & $vq^k$\\[3ex] \hline
  &&&\\
$k=\frac{n}{2}$ & $(v+v^{-1})q^{\frac{n}{2}}$ & $-v^{-1}q^{\frac{n}{2}}$ & $(v-v^{-1})q^{\frac{n}{2}}$\\[3ex] \hline
  &&&\\
$\frac{n}{2}<k<n$ & $v^{-1}q^{n-k}$ & $-v^{-1}q^{n-k}$ & $-v^{-1}q^{n-k}$\\[3ex] \hline
\end{tabular}
\end{center}
Note that $\frac{v+v^2}{(1-v)^3} = 0$ only if $v =-1$ and $v - v^{-1} = 0$ only if $v = \pm 1$. 
\end{proof}

\begin{Prop}\label{prop:holomorphicalphas}
If $n\geq 3$, the universal elliptic curve over $\MM_1^1(n)$ has a Weierstra{\ss} equation of the form
\[y^2 +\alpha_1xy + \alpha_3y = x^3 + \alpha_2x^2, \]
where the $\alpha_i$ are \emph{holomorphic} modular forms in $\MF(\Gamma_1(n);\Z[\frac1n])$ of degree $i$. 
\end{Prop}
\begin{proof}
\cref{lem:powerseries} checks exactly the holomorphy criterion \cref{lem:holomorphicitycriterion} once we observe that $P = (X(\zeta^dq^k, q^{n}), Y(\zeta^dq^k, q^{n}))$ cannot be a point of exact order $n$ if $k=0$ or $k=\frac{n}2$ and $\zeta^d = \pm 1$ because this would imply that $P$ is of order $2$.  
\end{proof}

According to the conventions from \cref{sec:qexpansions} we obtain the $q$-expansions of the $\alpha_i$ by specializing to the torsion point $(X(q, q^{n}), Y(q, q^{n}))$ above and use our explicit expressions of the $\alpha_i$ in terms of $X$ and $Y$. 

In our case of $n=7$ we obtain the following $q$-expansions for $X$ and $Y$:
\begin{align*}
    X &= q +2q^2+3q^3+4q^4+5q^5+7q^6+5q^7+9q^8+\cdots\\
Y &= q^2+3q^3+6q^4+10q^5+14q^6+22q^7+28q^8 +\cdots 
\end{align*}
As in the proof of \cref{prop:IdentifyMF}, the form of the $q$-expansions of $z_1,z_2$ and $z_3$ implies that there are elements of $\MF_k(\Gamma_1(7);\Z[\frac17])$ with $q$-expansions of the form $q^i + \text{higher terms}$ where $i$ runs over all integers in $[0,2k]$. As $\MF_k(\Gamma_1(7);\Z[\frac17])$ is free of rank $2k+1$, these elements form automatically a basis and thus every element in $\MF_k(\Gamma_1(7);\Z[\frac17])$ is determined by its $q$-expansion modulo $q^{2k+1}$. Comparing the $q$-expansions of the $z_i$ with those of the $\alpha_i$ using \texttt{MAGMA} implies the following theorem. 

\begin{Theorem}
\label{Alphas}
The elliptic curve classified by the composition 
\[\Spec \MFC(\Gamma_1(7), \Z[\tfrac17]) \to \MM_1(7) \to \MM_{ell}\]
is given by the equation
\[y^2 +\alpha_1xy +\alpha_3y = x^3 + \alpha_2x^2\]
with
\begin{equation*}
\begin{aligned}
 \alpha_1&=z_1-z_2+z_3,\\
 \alpha_2&=z_1z_2+z_1z_3,\\
 \alpha_3&=z_1z_3^2. 
\end{aligned}
\end{equation*}
\end{Theorem}

\section{Graded Hopf algebroids and stacks}\label{sec:HopfStacks}
In this section, all gradings can be taken to be either over $\Z$ (as convenient in the algebraic setting) or over $2\Z$ (as convenient in the topological setting) if this choice is done consistently. In either case, we assume our graded rings to be commutative and not just graded commutative.  All comodules over Hopf algebroids, a notion which we will recall in this section, are chosen to be left comodules. 

\subsection{General theory} \label{subsec:GeneralHopfAlgebroids}
We begin with the relationship between graded and ungraded Hopf algebroids. Let $(B,\Sigma)$ be a graded Hopf algebroid, i.e.\ a cogroupoid object in the category of graded rings. To such a graded Hopf algebroid, we can associate an ungraded Hopf algebroid $(B,\Sigma[u^{\pm 1}])$ as follows. The structure maps $\eta_L$ and $\varepsilon$ are essentially unchanged.  The right unit $\eta_R^{(B,\Sigma[u^{\pm 1}])} \colon B\to \Sigma[u^{\pm 1}]$ is given by
\[
\eta_R^{(B,\Sigma[u^{\pm 1}])}(x) = u^i\eta_R^{(B,\Sigma)}(x)
\]
if $x\in B$ is a homogeneous element of degree $i$. 
The comultiplication 
\[\psi^{(B,\Sigma[u^{\pm 1}])}\colon \Sigma[u^{\pm 1}]\to \Sigma [u^{\pm 1}]\otimes_B \Sigma [u^{\pm 1}]\]
is given by 
\[
\psi^{(B,\Sigma[u^{\pm 1}])}(s) = u^i\psi^{(B,\Sigma)}(s)
\]
for homogeneous elements $s\in \Sigma$ of degree $i$ and  $\psi(u) = 1\tensor u + u\tensor 1$. One can show that the category of graded comodules over $(B,\Sigma)$ is equivalent to that of comodules over $(B,\Sigma[u^{\pm 1}])$. 

In the following, we will assume that $(B,\Sigma)$ is \emph{flat}, i.e.\ that $\Sigma$ is flat as a $B$-module. We observe that $\Sigma$ is flat as a $B$-module with respect to the left module structure given by $\eta_L$ if and only if it is flat as a $B$-module with respect to the right module structure given by $\eta_R$. This can be shown using the conjugation.

\begin{Definition}
 The \emph{associated stack} for a graded Hopf algebroid $(B,\Sigma)$ is the stack $\XX(B,\Sigma)$ associated to the (ungraded) Hopf algebroid $(B,\Sigma[u^{\pm 1}])$ defined above, i.e.\ the stackification of the presheaf of groupoids represented by the groupoid scheme $(\Spec B, \Spec \Sigma[u^{\pm 1}])$ on the fpqc site of schemes. 
\end{Definition}

As explained in \cite[Section 3]{Naumann} the stack $\XX = \XX(B,\Sigma)$ is automatically algebraic in the sense of op.\ cit.\ and actually an Artin stack if $\Sigma$ is a finitely presented $B$-algebra (see \cite[Th\'eor\`eme 10.1]{Laumon}). Moreover, it comes with a map $\Spec B \to \XX$ and the pullback $\Spec B \times_{\XX}\Spec B$ can be identified with $\Spec \Sigma[u^{\pm 1}]$. The pullback 
\[\QCoh(\XX) \to \QCoh(\Spec B) \simeq B\modules\]
refines to an equivalences from $\QCoh(\XX)$ to left $(B,\Sigma[u^{\pm 1}])$-comodules by \cite[Section 3.4]{Naumann}; cf.\ also \cite[Theorem 2.2]{HoveyMoritaHopf}. Given an $\FF\in\QCoh(\XX)$, the comodule structure 
\[\FF(B) \to \Sigma[u^{\pm 1}] \tensor_B \FF(B)\]
is given by composing $\FF(\eta_L^{B,\Sigma[u^{\pm 1}]})\colon \FF(B) \to \FF(\Sigma[u^{\pm 1}])$ with the inverse of the isomorphism $\Sigma[u^{\pm 1}] \tensor_B \FF(B) \xrightarrow{\cong} \FF(\Sigma[u^{\pm 1}])$ induced by $\FF(\eta_R^{B,\Sigma[u^{\pm 1}]})$. 

\begin{example} \label{GradedModulesQCoh}
Let $B$ be a graded ring viewed as a graded Hopf algebroid $(B,B)$. Its associated stack is $\Spec B/\GG_m$ with the $\GG_m$-action corresponding to the grading. Graded $B$-modules are the same as graded comodules over $(B,B)$  and are thus equivalent to $\QCoh(\Spec B/\GG_m)$.
\end{example}

Composing the equivalences betweeen $\QCoh(\XX)$ and $(B,\Sigma[u^{\pm 1}])$-comodules and between the latter and graded $(B,\Sigma)$-comodules, we obtain the following. 
\begin{Prop}\label{associatedstackQCoh}
The map $(\id_B, \varepsilon)\colon (B,\Sigma) \to (B,B)$ of graded Hopf algebroids induces a map $f^B\colon\Spec B/\GG_m \to \XX$ and the pullback functor $(f^B)^*\colon \QCoh(\XX) \to \QCoh(\Spec B/\GG_m)$ refines to an equivalence between quasi-coherent sheaves on $\XX$ and graded comodules over $(B,\Sigma)$. 
\end{Prop}
It is easy to check that this equivalence is monoidal, when we put the following tensor product on graded comodules: Let $(M, \psi_M)$ and $(N, \psi_N)$ be two graded $(B,\Sigma)$-comodules. Then we set $(M, \psi_M) \tensor (N, \psi_N) = (M\tensor_A N,\psi_{M\tensor N})$, where $\psi_{M\tensor N}$ denotes the composite
\[ M\tensor_B N \xrightarrow{\psi_M \tensor \psi_N} \Sigma\tensor_B M \tensor_B \Sigma \tensor_B N \to \Sigma\tensor_B M \tensor_B N,  \]
with the last map being induced by the multiplication on $\Sigma$. Here we use that the associated graded comodule to some $\FF\in \QCoh(\XX)$ can be described completely analogously to the ungraded case above, only replacing $\Sigma[u^{\pm 1}]$ by $\Sigma$ and observing that $\FF(B)$ has a natural grading. 

Under the equivalences of \cref{GradedModulesQCoh} and \cref{associatedstackQCoh}, the pullback functor $\QCoh(\XX) \to \QCoh(\Spec B/\GG_m)$ translates into the forgetful functor from graded $(B,\Sigma)$-comodules to graded $B$-modules. This forgetful functor has a right adjoint, sending a graded $B$-module $M$ to the \emph{extended} graded $(B,\Sigma)$-comodule $\Sigma\tensor_B M$ with the comodule structure 
\[\psi \tensor \id_M\colon \Sigma\tensor_B M \to \Sigma\tensor_B\Sigma\tensor_B M\]
\cite[Definition A1.2.1]{RavenelAlt}. Under the equivalence above this translates into the right adjoint $f^B_*\colon \QCoh(\Spec B/\GG_m) \to \QCoh(\XX)$ to $(f^B)^*$. 

For the next lemma, consider a graded ring $R$. We recall that an $S$-valued point of $\Spec R/\GG_m$ corresponds to a $\GG_m$-torsor $T\to S$ together with a $\GG_m$-equivariant map $T \to \Spec R$. Here, we consider the $\GG_m$-action on $\Spec R$ corresponding to its grading. In particular the universal $R$-valued point $\Spec R \to \Spec R/\GG_m$ corresponds to the \emph{trivial} $\GG_m$-torsor $\Spec R[u^{\pm 1}] \to \Spec R$ together with the $\GG_m$-equivariant map $\Spec R[u^{\pm 1}] \to \Spec R$ induced by 
\[R \to R[u^{\pm 1}], \qquad x \mapsto xu^n \text{ for }x \text{ of degree }n. \] 
Here, we recall that the trivial $\GG_m$-torsor is induced by the obvious inclusion $i\colon R \to R[u^{\pm 1}]$ and the $\GG_m$-action on $\Spec R[u^{\pm 1}]$ is given by the grading on $R[u^{\pm 1}]$ where $|u| = 1$ and $i$ is the inclusion of the degree $0$ elements. For an arbitrary $S \to \Spec R/\GG_m$, we obtain $T$ as the pullback of this universal $\GG_m$-torsor $\Spec R \to \Spec R/\GG_m$. 

\begin{lemma}
The pullback $\Spec B\times_{\XX} \Spec B/\GG_m$ can be identified with $\Spec \Sigma$, where the first projection corresponds to $\eta_L$ and the second to the composite of the map $\eta_R\colon \Spec \Sigma \to \Spec B$ with the projection $\Spec B \to \Spec B/\GG_m$.
Moreover, the map $\Spec B/\GG_m \to \XX$ is fpqc. 
\end{lemma}
\begin{proof}
Consider the trivial $\GG_m$-torsor $i\colon \Sigma \to \Sigma[u^{\pm 1}]$ as above. The right unit $\eta_R^u = \eta_R^{(B,\Sigma[u^{\pm 1}])}$ defines a grading preserving map from $B$ to $\Sigma[u^{\pm 1}]$. We obtain a diagram
\[
\begin{tikzcd}
\Spec \Sigma[u^{\pm 1}] \arrow[r, "i"] \arrow[d, "\eta_R^u"] & \Spec \Sigma \arrow[d] \arrow[r, "\eta_L"] & \Spec B \arrow[d]\\
\Spec B \arrow[r] & \Spec B/\GG_m \arrow[r] & \XX
\end{tikzcd}
\]
As noted already above, the outer rectangle is a pullback square. The left square is a pullback square as it is a map of $\GG_m$-torsors. By fpqc descent along $\Spec B \to \Spec B/\GG_m$ the right square is a pullback square as well. As by \cite[Section 3.3]{Naumann}, the map $\Spec B \to \XX$ is fpqc, $\Sigma$ is flat over $B$ by assumption and moreover $\varepsilon$ is a retraction of $\eta_L$, we see that $\Spec B/\GG_m \to \XX$ is fpqc as well.

It remains to identify the map $\Spec \Sigma \to \Spec B/\GG_m$. Using the discussion before this lemma, we see that the composite $\Spec \Sigma \xrightarrow{\eta_R} \Spec B \to \Spec B/\GG_m$ corresponds to the trivial $\GG_m$-torsor $i\colon \Spec \Sigma[u^{\pm 1}] \to \Spec \Sigma$ together with the $\GG_m$-equivariant map $\eta_R^u\colon \Spec \Sigma[u^{\pm 1}] \to \Spec B$, exactly as claimed. \end{proof}
Dually, we can of course also identify $\Spec B/\GG_m \times_{\XX} \Spec B$ with $\Spec \Sigma$ and obtain analogous descriptions of the projections. In particular, we see that 
\[\Spec B/\GG_m \times_{\XX} \Spec B\quad \text{ and }\,\Spec B \times_{\XX} \Spec B/\GG_m\] 
are equivalent over $\Spec B/\GG_m \times \Spec B/\GG_m$. In the next lemma we investigate what happens after base change along a morphism $B \to C$. 

\begin{Lemma}\label{lem:pullback}\label{lem:associated}
 Let $B \to C$ and $B\to D$ be grading preserving ring morphisms. This induces a morphism $f^C\colon \Spec C/\GG_m \to \Spec B/\GG_m \to \XX$ and similarly for $D$.
 \begin{enumerate}
     \item The pullback $\Spec B \times_{\XX} \Spec C/\GG_m$ is equivalent to $\Spec \Sigma \tensor_B C$,
 where $\Sigma$ is a $B$-module via the right unit $\eta_R$. 
 \item The quasi-coherent sheaf $f^C_*\OO_{\Spec C/\GG_m}$ corresponds under the equivalence from \cref{associatedstackQCoh} to the extended $(B,\Sigma)$-comodule structure on $\Sigma\tensor_B C$.
 \item The pullbacks $\Spec C\times_{\XX}\Spec D/\GG_m$ and $\Spec D/\GG_m \times_{\XX} \Spec C$ are equivalent to $\Spec \Omega$ with $\Omega = C\tensor_B \Sigma \tensor_B D$ and $\Omega = D\tensor_B \Sigma \tensor_B C$, respectively. If $C=D$ and the maps $B\to C$ coincide, then $(C,\Omega)$ obtains the structure of a graded Hopf algebroid.
 \item The stack associated with $(C,\Omega)$ is equivalent to $\XX$ if $f^C$ is fpqc.
 \end{enumerate}
 \end{Lemma}
\begin{proof}
Recall that $\Spec B \times_{\XX}\Spec B/\GG_m \simeq \Spec \Sigma$. Pulling back along the map $\Spec C/\GG_m \to \Spec B/\GG_m$ gives the equivalence in the first item. The second item follows by the remarks above as we can factor $f^C$ into the map $\Spec C/\GG_m \to \Spec B/\GG_m$ and $f^B$. Pulling back the equivalence from the first item along $\Spec D \to \Spec B$ gives the equivalences in the third item. The pullback $\Spec C \times_{\XX} \Spec C$ is of the form $\Omega[u^{\pm 1}]$ for analogous reasons. By \cite[Section 3.3]{Naumann}, $\XX$ is the stack associated with the (ungraded) Hopf algebroid $(C,\Omega[u^{\pm 1}])$ and thus also with the graded Hopf algebroid $(C,\Omega)$. 
\end{proof}

We remark that with notation as in the preceding lemma, the identity on $\Omega = C\tensor_B\Sigma\tensor_B D$ provides us with an equivalence between $\Spec C\times_{\XX}\Spec D/\GG_m$ and $\Spec C/\GG_m \times_{\XX} \Spec D$ that is compatible with the projections to $\Spec C/\GG_m$ and to $\Spec D/\GG_m$.

\subsection{Stacks related to elliptic curves}
\label{StacksHopfElliptic}

Recall that we are working with the moduli stack of elliptic curves $\MM_{ell}$. Let 
\[
A:=\mathbb{Z}[a_1,a_2, a_3, a_4,a_6] \mbox{ and } \Gamma:=A[r,s,t].
\]
There is an element $\Delta \in A$ corresponding to the discriminant for cubical curves; see e.g.\ \cite[Section III.1]{SilvermanAEC} for a precise formula. The stack $\MM_{ell}$ is equivalent to the stack associated with the graded Weierstra\ss{} Hopf algebroid $(A[\Delta^{-1}], \Gamma[\Delta^{-1}])$ in the sense recalled in \cref{subsec:GeneralHopfAlgebroids}. For the precise structure maps, see e.g.\ \cite{Bauertmf}; note the name comes from the fact that this Hopf algebroid is related to Weierstra\ss{} equations for elliptic curves and the right unit $\eta_R$ comes from change-of-coordinates formulas for these. Our grading convention is that $|a_i| = i$.

One observes that the structure maps do not use the fact that $\Delta$ was inverted, so one can consider the graded Weierstra\ss{} Hopf algebroid $(A, \Gamma)$.

\begin{Definition}
 Let \emph{the moduli stack of cubical curves} $\MM_{cub}$ be the Artin stack associated with the graded Weierstra\ss{} Hopf algebroid $(A, \Gamma)$.
\end{Definition}

The name is justified, as there is a modular interpretation for this stack, see \cite[Section 3.1]{Mathewtmf}, and in particular the morphism $\Spec A \to \mcub$ classifies the Weierstra\ss{} cubical curve over $A$ given by the usual equation
\[
y^2+a_1xy+a_3y=x^3+a_2x^2+a_4x+a_6.
\]

We record the relationship with the moduli stack of elliptic curves.
\begin{lemma}\label{MellopenMcub}
There is a pullback square
\[
\begin{tikzcd}
\Spec A[\Delta^{-1}] \arrow[r] \arrow[d] & \Spec A \arrow[d]\\
\mell \arrow[r] & \mcub
\end{tikzcd}
\]
In particular, $\MM_{cub}$ contains $\MM_{ell}$ as an open substack and the inclusion is an affine morphism.
\end{lemma}

\begin{proof}
The existence of the pullback square corresponds to the fact that a cubical curve given by a Weierstra\ss{} equation is an elliptic curve if and only if its discriminant $\Delta$ is invertible. 

As noted in the last subsection, $\Spec A \to \MM_{cub}$ is fpqc. Since both being open immersion and being affine can be checked after faithfully flat base change, the remaining claims follow. 
\end{proof}

\begin{Definition}
 We define the line bundle $\omline$ on $\MM_{cub}$ to be the one corresponding to the shift $A[1]$ under the equivalence between quasi-coherent sheaves on $\MM_{cub}$ and graded $(A,\Gamma)$-comodules. Here, $A[1]_i = A_{i+1}$ and the comodule structure is induced by $\eta_L\colon A[1] \to \Gamma[1] \cong \Gamma \tensor_A A[1]$.
\end{Definition}
By \cite[(1.2)]{dapres}, our definition of $\omline$ agrees with the more geometric definition of Deligne; in particular, our definition restricts to the corresponding definition on $\MMb_{ell}$ we give in the appendix. By the transformation formulae in \cite[Section III.1]{SilvermanAEC}, it is easy to see that $\eta_L$ and $\eta_R$ agree on $\Delta \in A_{12}$. Thus, $\Delta$ defines a map $A \to A[12]$ of graded $(A,\Gamma)$-comodules and hence a section of $\omline^{\tensor 12}$ on $\MM_{cub}$. \Cref{TateWeierstrass} implies that this $\Delta$ corresponds after restriction to $\MM_{ell}$ and under the isomorphism $H^0(\MM_{ell};\omline^{\tensor 12}) \cong \MFC_{12}(\SL_2(\Z); \Z)$ indeed to the modular form $\Delta$ that we introduced via its $q$-expansion in \cref{sec:analytic}.

Next, we will need the following easy lemma. 

\begin{lemma}\label{SpecAsmooth}
The morphism $\Spec A/\GG_m \to \MM_{cub}$ is smooth.
\end{lemma}
\begin{proof}
As noted in the last subsection, $\Spec A \to \MM_{cub}$ is fpqc. Moreover, $\Gamma = A[r,s,t]$ is smooth over $A$ and $\Spec \Gamma \simeq \Spec A \times_{\MM_{cub}} \Spec A/\GG_m$. As smoothness can be tested after base change along an fpqc morphism \cite[\href{https://stacks.math.columbia.edu/tag/02VL}{Tag 02VL}]{stacks-project}, we obtain the claim. 
\end{proof}

When working at the prime $3$, it turns out to be more convenient to work with a different smooth cover of $\MM_{cub} =\MM_{cub, \mathbb{Z}_{(3)}}$, namely with $\Spec \widetilde{A} \to \MM_{cub}$, where $\widetilde{A} :=\mathbb{Z}_{(3)}[\aq_2,\aq_4,\aq_6]$ and the morphism is given by composition of the canonical morphism $\Spec A\to \MM_{cub}$ with the one induced by the ring map $A \to \widetilde{A}$, given by 
\[
a_1, a_3\mapsto 0 \mbox{ and } a_i \mapsto \overline{a}_i \mbox{ for } i\in \{2,4,6\}.
\]
The cubical curve corresponding to the morphism $\Spec \widetilde{A}\to \MM_{cub}$ is
\[
y^2=x^3+\aq_2x^2+\aq_4x+\aq_6.
\]
Note that there are different conventions for simplifying the elliptic or cubical curves when $2$ is inverted. In particular, the convention used in \cite{SilvermanAEC} differs from ours. 

We want to show that this smooth cover induces a different Hopf algebroid $(\widetilde{A}, \widetilde{\Gamma})$ with $\widetilde{\Gamma}:=\widetilde{A} \otimes_A \Gamma \otimes_A \widetilde{A}$ representing  $\MM_{cub}$. For more details on the explicit description of $(\widetilde{A}, \widetilde{\Gamma})$, see \cite[Section 3]{Bauertmf}. We will first prove that it is indeed a presentation for $\mcub$, and then recall some of the structure maps we will be using later.

\begin{lemma}\label{AtildevsQCoh}
At the prime $3$, the stack associated to the graded Hopf algebroid $(\widetilde{A}, \widetilde{\Gamma})$ is equivalent to $\mcub$. In particular, there is an equivalence between quasi-coherent sheaves on $\MM_{cub}$ and graded  $(\widetilde{A},\widetilde{\Gamma})$-comodules. Moreover, the morphism $\Spec \widetilde{A}/\mathbb{G}_m \to \mcub$ is a smooth cover.
\end{lemma}

\begin{proof}
We would like to apply \cref{lem:associated}. Thus, we only have to check that the composition $\Spec\widetilde{A}/\mathbb{G}_m \to \Spec A/\mathbb{G}_m \to \mcub$ is fpqc. As explained in \cref{subsec:GeneralHopfAlgebroids}, the map $\Spec A \to \mcub$ is fpqc, so by faithfully flat descent, it is enough to check that the pullback map $\Spec A\times_{\mcub} \Spec \widetilde{A}/\mathbb{G}_m \to \Spec A$ is smooth and surjective. By \cref{lem:associated}, we can identify the source with $\Spec \Gamma \otimes_A \widetilde{A}$. By inspection, 
we arrive at the isomorphism of $A$-modules
\[
\Gamma \otimes_A \widetilde{A} \cong A[r,s,t]/(\eta_R(a_1), \eta_R(a_3)).
\]
Using the right unit formulae
\begin{align*}
    \eta_R(a_1) &=a_1+2s,\\
    \eta_R(a_3) &=a_3+a_1r+2t,
\end{align*}
and the fact that we inverted $2$, we get $\Gamma \otimes_A \widetilde{A} \cong A[r]$. In particular, this is a smooth $A$-algebra and the claim follows.
\end{proof}

By the proof of the previous lemma we obtain 
\[\widetilde{\Gamma}\cong \tilde{A}\tensor_A A [r] \cong {\widetilde{A}}[r].\] 
The structure formulae for the Hopf algebroid $(A,\Gamma)$ determine under this identification the formulae
\[
 \begin{aligned}
 \eta_R(\aq_2)&=\aq_2+3r,\\
  \eta_R(\aq_4)&=\aq_4+2r\aq_2+3r^2,\\
  \eta_R(\aq_6)&=\aq_6+r\aq_4+r^2\aq_2+r^3,
 \end{aligned}
\]
whereas $\eta_L$ is the canonical inclusion of $\widetilde{A}$.

\begin{remark}
Note that there is an even easier smooth cover of $\MM_{cub}$ coming from $\MM_1(2)_{cub}$ and a corresponding graded Hopf algebroid yielding $\MM_{cub}$ again. For our approach, it has the following disadvantage: there is no meaningful module structure on $\MF_1(7)$ over the corresponding ring. On the contrary, $\MF_1(7)$ is a $\widetilde{A}$-module corresponding to a cubical curve discussed later.
\end{remark}


\section{The definition and properties of \texorpdfstring{$\MM_1(7)_{cub}$ and $\MM_0(7)_{cub}$}{cubical moduli stacks}}
\label{sec:Flatness}
Fix throughout the section an integer $n\geq 2$ and the notation 
$$A = \mathbb{Z}\left[\tfrac{1}{n}\right][a_1,a_2, a_3, a_4,a_6].$$

We want to extend the moduli stacks $\MM_0(n) = \MM_0(n)_{\mathbb{Z}\left[\frac{1}{n}\right]}$ and $\MM_1(n) = \MM_1(n)_{\mathbb{Z}\left[\frac{1}{n}\right]}$ to algebraic stacks that are finite over $\MM_{cub}$ via a normalization construction. 

\subsection{Definition and basic properties of \texorpdfstring{$\MM_1(n)_{cub}$ and $\MM_0(n)_{cub}$}{cubic moduli stacks}}
Let us recall the notion of normalization. Let $\XX$ be an Artin stack and $\AA$ a quasi-coherent sheaf of $\OO_\XX$-algebras. Let $\AA' \subset \AA$ be the presheaf that evaluated on any smooth $\Spec C\to \XX$ consists of those elements in $\AA(\Spec C)$ that are integral over $C$. This is an fpqc (and in particular \'etale) sheaf because being integral for an element can be tested fpqc-locally (as generating a finite module can be checked fpqc-locally). Thus, we obtain a sheaf on the subsite of the lisse-\'etale site of $\XX$ consisting of affine schemes (see \cite[Section 12]{Laumon} or \cite[Tag 0786]{stacks-project} for the definition). 

\begin{lemma}\label{lem:AAstrich}This construction has the following properties. 
\begin{enumerate}
    \item If $\XX = \Spec D$ is affine, then $\AA'$ is the quasi-coherent sheaf associated with the $D$-algebra that is the normalization of $D$ in $\AA(\Spec D)$.
    \item If $p\colon \YY \to \XX$ is a smooth morphism of Artin stacks, then the map $p^*(\AA') \to (p^*\AA)'$ is an isomorphism. 
    \item The sheaf $\AA'$ is quasi-coherent for general Artin stacks $\XX$.
\end{enumerate}
\end{lemma}
\begin{proof}
Let $\XX = \Spec D$ be affine and let $D \to C$ be a smooth map of rings. By \cite[Tag 03GG]{stacks-project} and using that $\AA$ is quasi-coherent we obtain that the canonical map $\AA'(\Spec D)\tensor_D C \to \AA'(\Spec C)$ is an isomorphism. This implies that $\AA'$ is quasi-coherent in this case. 

Let now $p\colon \YY \to \XX$ be a smooth morphism of Artin stacks with $\XX$ general again and let $\Spec C \xrightarrow{q} \YY$ be smooth as well. Then both $p^*(\AA')(\Spec C)$ and $(p^*\AA)'(\Spec C)$ are computed as the normalization of $C$ in $(q^*p^*\AA)(C) = \AA(C)$. 

Finally, let 
\[
\begin{tikzcd}
\Spec C \arrow[rr,"r"]\arrow[rd, "p" swap] && \Spec D\arrow[ld, "q"]\\
&\XX&
\end{tikzcd}
\]
be a $2$-commutative diagram where the vertical maps are smooth. To show the quasi-coherence of $\AA'$, we need to show that the natural map 
\begin{eqnarray}\label{eq:quasicoherence}\AA'(\Spec D)\tensor_D C \to \AA'(\Spec C)\end{eqnarray}
is an isomorphism. From the above, $q^*\AA'$ is quasi-coherent on $\Spec D$. Hence $(r^*q^*\AA')(\Spec C)$ can be computed as $(q^*\AA')(\Spec D)\tensor_D C = \AA'(\Spec D)\tensor_D C$. Thus we can identify the map \eqref{eq:quasicoherence} with the evaluation of the isomorphism $r^*q^*\A' \cong p^*\AA'$ on $\Spec C$.
\end{proof}

\begin{Definition}
We define the \emph{normalization of $\XX$ in $\AA$} to be the relative $\Spec$ of $\AA'$ over $\XX$. For a quasi-compact and quasi-separated morphism $p\colon \YY \to \XX$, we define the \emph{normalization} of $\XX$ in $\YY$ to be the 
normalization of $\XX$ in $p_*\OO_\YY$, where $p_*$ denotes the pushforward of quasi-coherent sheaves as in \cite[Tag 070A]{stacks-project}.
\end{Definition} 

Directly from \cref{lem:AAstrich} and \cite[Proposition 13.1.9]{Laumon} we obtain:
\begin{Lemma}\label{lem:smoothbasechange}
 Relative normalization commutes with smooth base change.
\end{Lemma}

Recall that the compactifications $\MMb_0(n)$ and $\MMb_1(n)$ are defined as the normalizations of $\MMb_{ell}$ in $\MM_0(n)$ and $\MM_1(n)$, respectively. This motivates the following definition. 
\begin{Definition}
We define $\MM_0(n)_{cub}$ and $\MM_1(n)_{cub}$ as the normalizations of $\MM_{cub}$ in $\MM_0(n)$ and $\MM_1(n)$.
\end{Definition} 

 It remains an open problem to provide modular interpretations for $\MM_1(n)_{cub}$ and $\MM_0(n)_{cub}$ similar to those in \cite{Conrad} and \cite{Cesna} for $\MMb_1(n)$ and $\MMb_0(n)$. Moreover, we warn the reader that there is no reason to expect the map $\MM_1(n)_{cub} \to \MM_0(n)_{cub}$ to be the stack quotient by the natural $(\Z/n)^\times$-action on the source.

Note that the normalization maps are by definition affine and in the next lemma we will even show finiteness. 

\begin{Lemma}\label{lem:finitem1n}
 The maps $\MM_1(n)_{cub} \to \MM_{cub}$ and $\MM_0(n)_{cub} \to \MM_{cub}$ are finite.
\end{Lemma}
\begin{proof}
Let $\XX \to \MM_{ell}$ be an affine map of finite type from a reduced Artin stack. Note that reducedness is local in the smooth topology \cite[Tag 034E]{stacks-project} so that we can define an Artin stack to be reduced if it has a smooth cover by a reduced scheme. We want to show that the normalization $\XX'$ of $\MM_{cub}$ in $\XX$ is finite over $\MM_{cub}$. The relevant cases for us are $\XX = \MM_1(n)$ and $\XX = \MM_0(n)$. 

 Let $\Spec A \to \MM_{cub}$ be the usual smooth cover. Denote by $T$ the global sections of the pullback $\XX\times_{\MM_{cub}} \Spec A$, which is an affine scheme. By \cref{lem:smoothbasechange}, the pullback $\XX' \times_{\MM_{cub}} \Spec A$ is equivalent to the spectrum of the normalization $A'$ of $A$ in $T$. As finiteness can be checked after faithfully flat base change, it thus suffices to show that $A'$ is a finite $A$-module. 
 
 By \cite[\href{https://stacks.math.columbia.edu/tag/03GR}{Tag 03GR}]{stacks-project}, we just have to check that $A$ is a Nagata ring, $\Spec T \to \Spec A$ is of 
finite type and $T$ is reduced. As $A$ is a polynomial ring over a quasi-excellent ring, it is quasi-excellent again and hence Nagata \cite[Tag 07QS]{stacks-project}. The 
second point is clear by base change. For the last one note that $\Spec T$ is equivalent to $\Spec A[\Delta^{-1}] \times_{\MM_{ell}} \XX$ by \cref{MellopenMcub}, and also that this pullback is affine. Moreover, $\Spec A[\Delta^{-1}]\to \MM_{ell}$ is smooth by \cref{SpecAsmooth}. 
Since being reduced is local in the smooth topology, we conclude that $T$ is reduced. Hence, we obtain finiteness of $A'$ over $A$ and thus of $\XX'$ over $\MM_{cub}$.
\end{proof}

\subsection{Commutative algebra of rings of modular forms}
For structural results about $\MM_1(n)_{cub}$ we need some information about the commutative algebra of rings of modular forms. We will use the abbreviation $\MF_1(n)$ for $\MF(\Gamma_1(n);\Z[\tfrac{1}{n}])$. Recall from \cref{sec:summary} that $\MF_k(\Gamma_1(n);R) \cong H^0(\MMb_1(n)_R; \omline^{\tensor k})$ for every subring $R\subset \C$ and thus we set in general 
\[\MF_k(\Gamma_1(n);R) = H^0(\MMb_1(n)_R; \omline^{\tensor k})\]
for every $\Z[\frac1n]$-algebra $R$. 

Our first goal is to investigate when $\MF_1(n)$ is Cohen--Macaulay. This will be later the key to show that for many $n$ the map $\MM_1(n)_{cub} \to \MM_{cub}$ is flat. Recall that a noetherian commutative ring $R$ is called \emph{Cohen--Macaulay} if for every maximal ideal $\m\subset R$ the depth of the ideal $\m R_{\m}$ equals the Krull dimension of $R_{\m}$. 

On the other hand, $\MF_1(n)$ is a graded ring and it might appear more natural to consider a graded notion of being Cohen--Macaulay, where graded always means $\Z$-graded for us. An ideal is called \emph{graded} if it is generated by homogeneous elements; this results in the notion of a graded prime ideal and thus also in the notion of graded Krull dimension. Moreover, a graded ideal is called \emph{graded maximal} if it is maximal among all graded ideals. Note that with these definitions, the following lemma is not completely obvious, but also not hard to show: 
\begin{lemma}
Any graded maximal ideal in a graded ring is a graded prime ideal. 
\end{lemma}

Given a graded prime ideal $\p$ in a graded ring $R$, we denote by $R_{(\p)}$ the localization at the multiplicatively closed subset of all homogeneous elements not in $\p$. The \emph{graded depth} of an ideal $I\subset R$ is the maximal length of a regular sequence of homogeneous elements in $I$. 

\begin{Definition}
 A graded noetherian ring $R$ is \emph{graded Cohen--Macaulay} if for every graded maximal ideal $\m\subset R$ the graded depth of $\m R_{(\m)}$ equals the graded Krull dimension of $R_{(\m)}$.
\end{Definition} 
Note that the graded Krull dimension of $R_{(\m)}$ agrees with the Krull dimension of $R_{\m}$ by \cite[Theorem 1.5.8]{BrunsHerzog}.

\begin{lemma}\label{lem:CMgraded}
Let $R$ be a graded noetherian ring. Then $R$ is Cohen--Macaulay if it is graded Cohen--Macaulay. 
\end{lemma}
\begin{proof}
According to \cite[Exercise 2.1.27]{BrunsHerzog}, $R$ is Cohen--Macaulay if and only if $R_{(\p)}$ is Cohen--Macaulay for every graded prime ideal $\p\subset R$. As by \cite[Theorem 2.1.3]{BrunsHerzog} being Cohen--Macaulay is preserved by localizations, this happens if and only if $R_{(\m)}$ is Cohen--Macaulay for every graded maximal ideal $\m \subset R$. Thus, we may assume that $R$ contains a unique graded maximal ideal $\m$. 

By \cite[Exercise 2.1.27]{BrunsHerzog} again, it suffices to show that $R_{\m}$ is Cohen--Macaulay. Let $d$ be the dimension of $R_{\m}$. By assumption, there is a $R_{(\m)}=R$-regular sequence $x_1,\dots, x_d$ in $\m$. A straightforward check shows that the same sequence is also regular on $R_{\m}$. This completes the proof. 

\end{proof}

\begin{remark}\label{rem:regular}
By \cite[Corollary 2.2.6]{BrunsHerzog} every regular ring is Cohen--Macaulay, but the ring $\MF_1(n)$ is not regular in general, even over the complex numbers. Indeed, a maximal ideal of $\MF_1(n)\tensor \C$ is generated by all elements of positive degree and thus needs at least $\dim_{\C}\MF_1(n)_1\tensor \C$ many generators. On the other hand, $\MF_1(n)\tensor \C$ has Krull dimension $2$ as in the proof of \cite[Theorem 5.14]{MeierDecomposition} and thus $\MF_1(n)\tensor \C$ can only be regular if $\MF_1(n)_1\tensor \C$ is of dimension at most $2$. This does not happen for $n \geq 7$ as the dimension of $\MF_1(n)_1 \tensor \C$ is at least half the number of regular cusps, i.e.\ at least $\frac14\sum_{d|n}\varphi(d)\varphi(\frac{n}d)$ by \cite[Theorem 3.6.1 and Figure 3.3]{DiamondShurman}, where $\varphi$ denotes Euler's totient function. On the other hand, the rings $\MF_1(n)$ are polynomial for $2\leq n\leq 6$ and in particular regular. 

There is the related property of being Gorenstein, which is stronger than Cohen--Macaulay, but weaker than regular. The values of $n$ for which $\MF_1(n)$ is Gorenstein were found by Dimitar Kodjabachev in his thesis. 
\end{remark}

\begin{Prop}\label{prop:CM}
The ring $\MF_1(n)$ is a Cohen--Macaulay ring if and only if the map $\MF_1(n)_1 \to \MF_1(\Gamma_1(n); \Fb_l)$ is surjective for all primes $l$ not dividing $n$. This happens if and only if $H^1(\MMb_1(n);\omline)$ is torsionfree. 
\end{Prop}
\begin{proof}
This follows from \cite[Theorem 5.14]{MeierDecomposition} and \cref{lem:CMgraded}. 
\end{proof}

\begin{example}\label{exa:CM}
As noted in \cite[Remark 3.14]{MeierDecomposition}, the condition of \cref{prop:CM} is equivalent to the existence of a \emph{cusp form} in $\MF_1(\Gamma_1(n); \Fb_l)$ that is not liftable to a cusp form in $\MF_1(n)_1$. For $n\leq 28$, Buzzard \cite{BuzzardWeight1} shows that only for $n=23$ there is a nonvanishing cusp form in $\MF(\Gamma_1(n);\Fb_l)$. But \cite[Corollary 5.8]{MeierTMFLevel} shows that on $\MMb_1(23)$ there is an isomorphism $\Omega^1_{\MMb_1(23)/\Z[\frac1n]} \cong \omline$. By \cite[Proposition 2.11]{MeierDecomposition}, this implies that we can identify the reduction map $\MF_1(23)_1 \to \MF_1(\Gamma_1(23); \Fb_l)$ with the surjection $\Z[\frac1{23}] \to \Fb_l$. (A similar argument also appears in \cite{BuzzardWeight1}.) 

Thus, $\MF_1(n)$ is Cohen--Macaulay for all $2\leq n\leq 28$. It is not Cohen--Macaulay for example for $n=74$ or $n=82$ (see \cite[Remark 3.14]{MeierDecomposition}). 
\end{example}

In the context of normalizations it is furthermore an important question whether the rings $\MF_1(n)$ are normal. 

\begin{Prop}\label{prop:normal}
The ring $\MF_1(n)$ is normal for every $n\geq 2$. 
\end{Prop}
\begin{proof}
 We begin with some recollections on the general commutative algebra of $R$. By \cite[Proposition 2.13]{MeierDecomposition} the ring $R$ is an integral domain so that $(0)$ is its only prime ideal of height $0$. Moreover, $R/p$ is an integral domain for every prime number $p$ not dividing $n$ because $\MF(\Gamma_1(n); \mathbb{F}_p)$ is an integral domain by \cite[Proposition 2.13]{MeierDecomposition} and $R/p$ embeds into $\MF(\Gamma_1(n); \mathbb{F}_p)$ \cite[Remark 3.14]{MeierDecomposition}. The same \cite[Proposition 2.13]{MeierDecomposition} also says that $R_0 = \Z[\frac1n]$. We also note that $R$ is noetherian, e.g.\ as it is finitely generated over the noetherian ring $\MF(\SL_2(\Z);\Z[\frac1n]) \cong \Z[\frac1n][c_4,c_6,\Delta]/(1728\Delta = c_4^3-c_6^2)$ \cite[Corollary 1.4]{MeierDecomposition}. Thus Serre's normality criterion \cite[Th\'eor\`eme 5.8.6]{EGAIV} says that it suffices to show that $R = \MF_1(n)$ satisfies R1 and S2.

Condition R1 says that for every prime ideal $\mathfrak{p}\subset R$ of height $1$, the localization $R_{\mathfrak{p}}$ is regular. Let $\MMb_1^1(n) \to \MMb_1(n)$ be the $\mathbb{G}_m$-torsor over which the line bundle $\omline$ trivializes. By \cref{prop:coordinatesm1n}, the scheme $\MMb_1^1(n)$ is quasi-affine and can be identified as the complement of the common vanishing locus $V(c_4,\Delta)$ of $c_4$ and $\Delta$ in $\Spec R$. 

We claim that no point $\mathfrak{p}$ of height $1$ of $\Spec R$ can be in $V(c_4,\Delta)$, essentially as this is closed of codimension $2$. More precisely, we claim that $c_4,\Delta$ form a regular sequence and thus $c_4,\Delta \in\mathfrak{p}$ would imply that the height of $\mathfrak{p}$ is at least $2$ \cite[Proposition 1.2.14]{BrunsHerzog}. As $R$ is torsionfree and $\Delta$ is not divisible by any non-unit in $\Z[\frac1n]$ (as can be seen by its $q$-expansion $q-24q^2+\cdots$), it suffices to show that $c_4, \Delta$ forms a regular sequence in $R_{\Q}$. By \cite[Theorem 1.1]{MeierDecomposition}, $R_{\Q}$ is free of finite rank over the ring $S = \Q[c_4,c_6]$ of modular forms and we see indeed that $c_4$ and $\Delta = \frac1{1728}(c_4^3-c_6^2)$ form a regular sequence in $S$ and hence in $R_{\Q}$.  

We conclude that $\mathfrak{p} \in \MMb_1^1(n)$ and the localization $R_{\mathfrak{p}}$ can be identified with a stalk in $\MMb_1^1(n)$. But $\MMb_1^1(n)$ is smooth over $\Z[\frac1n]$ (as $\MMb_1(n)$ is smooth over $\Z[\frac1n]$ and $\mathbb{G}_m$-torsors are smooth) and thus regular. 

Condition S2 says that the depth of $R_{\mathfrak{p}}$ is for every prime ideal $\mathfrak{p}\subset R$ at least the height of $\mathfrak{p}$ or at least $2$. As $R$ is an integral domain, $\mathfrak{p}$ has at least depth $1$ if it is nonzero. Thus, we need to show that every prime ideal $\mathfrak{p}$ of height at least $2$ has depth at least two. 

Assume first that $p \in \mathfrak{p}$ for some prime number $p$. By Krull's principal ideal theorem, $\mathfrak{p}$ cannot be principal and thus it contains some $x \notin (p)$.  As $R/p$ is an integral domain, $p,x$ is a regular sequence of length $2$ in $\mathfrak{p}$. 

Assume now that no prime $p$ is in $\mathfrak{p}$. Then $\mathfrak{p}R_{\Q}$ is still a prime ideal and $\mathfrak{p}R_{\Q}\cap R = \mathfrak{p}$. Moreover $R_{\mathfrak{p}}$ is a $\Q$-algebra, namely the localization of $R_{\Q}$ at $\mathfrak{p}R_{\Q}$. As noted above, $R_{\Q}$ is finite and free over $S = \Q[c_4,c_6]$. By \cite[Theorem A.6]{BrunsHerzog}, the ideal $\mathfrak{p} \cap S$ has still height $2$. As $S$ is regular we can choose a regular sequence of length $2$ in $\mathfrak{p} \cap S$ that is automatically also regular in $R_{\Q}$ and hence in $R_{\mathfrak{p}}$.
\end{proof}

\begin{example}\label{exa:normal}
The ring $\MF_1(7) \cong \Z[\frac17][z_1,z_2,z_3]/(z_1z_2+z_2z_3+z_3z_1)$ is no longer regular by Remark \ref{rem:regular}, but still Cohen--Macaulay by Example \ref{exa:CM} and normal by the proposition above. 
\end{example}

\subsection{The flatness \texorpdfstring{of $\MM_1(n)_{cub} \to \MM_{cub}$}{results}}
The two aims of this section are to show that $\MM_1(n)_{cub}$ agrees with $\Spec \MF_1(n)/\GG_m$ and that under some conditions $\MM_1(n)_{cub} \to \MM_{cub}$ is flat. These claims will be special cases of the more general criterion \cref{lem:CMFlatness}. 

We will still use our convention that we work implicitly over $\Z[\frac1n]$ for a fixed $n\geq 2$, i.e.\ that $\MM_{cub}$ means $\MM_{cub,\Z[\frac1n]}$ and $A = \Z[\frac1n][a_1,a_2,a_3,a_4,a_6]$ etc. Note moreover that $\MM_{cub} \to \Spec \Z[\frac1n]$ is smooth as $\Spec A \to \Spec \Z[\frac 1n]$ is. In particular, $\MM_{cub}$ is reduced by \cite[Tag 034E]{stacks-project}.

\begin{Prop}\label{lem:finitenessandnormalization}\label{lem:CMFlatness}
Let $R$ be a graded ring with a graded ring map $A \to R$. Assume that the induced map
$$\Spec R[\Delta^{-1}]/\GG_m \to \Spec A[\Delta^{-1}]/\GG_m \to \MM_{ell}$$
is surjective and $R_A = \Gamma\tensor_A R$ is a finitely generated $A$-module. 
\begin{enumerate}[label=(\alph*)]
    \item\label{normal} If $R$ is a normal domain, the normalization of $\MM_{cub}$ in $\Spec R[\Delta^{-1}]/\GG_m$ is equivalent to $\Spec R/\GG_m$. 
    \item\label{CM} If $R$ is Cohen--Macaulay, then $R_A$ is a flat $A$-module and thus the morphism $\Spec R/\GG_m \to \MM_{cub}$ is flat as well.
\end{enumerate}
\end{Prop}
\begin{proof}
We begin with part \ref{normal}. Let $\Spec C \to \MM_{cub}$ be any smooth map. As $\Spec R / \GG_m \to \MM_{cub}$ is by \cref{lem:pullback} and fpqc descent an affine morphism, the fiber product $\Spec C \times_{\MM_{cub}} \Spec R/\GG_m$ is affine again and we denote it by $\Spec R_C$. With the same notation, the fiber product $\Spec C \times_{\MM_{cub}} \Spec R[\Delta^{-1}]/\GG_m$ is equivalent to $\Spec R_C[\Delta^{-1}]$. By the definition of the normalization, we thus must show that the normalization of $C$ in $R_C[\Delta^{-1}]$ equals $R_C$. 

As a first step, we will show that the natural maps $C \to R_C[\Delta^{-1}]$ (and hence $C \to R_C$) are injections. We can work locally and assume that the pullback of $\omline$ to $\Spec C$ is trivialized so that $\Delta\in H^0(\MM_{cub};\omline^{\tensor 12})$ defines an element of $C$, well-defined up to multiplication by a unit. Note further that $\Spec C$ is noetherian and hence it is the disjoint union of finitely many affine components; thus we can assume moreover that $\Spec C$ is connected. As $\Spec R_C[\Delta^{-1}]/\GG_m\to \MM_{ell}$ is surjective by assumption, its base change $\Spec R_C[\Delta^{-1}] \to \Spec C[\Delta^{-1}]$ along $\Spec C[\Delta^{-1}] \to \MM_{ell}$ is surjective as well. We see that the kernel of $C[\Delta^{-1}] \to R_C[\Delta^{-1}]$ lies in the intersection of all prime ideals, i.e.\ in the nilradical. As $\MM_{cub}$ is reduced and hence $C$ is reduced as well, it follows that $C[\Delta^{-1}] \to R_C[\Delta^{-1}]$ is injective. Thus, it suffices to show that $C \to C[\Delta^{-1}]$ is injective. Note that $\Spec C$ is smooth over $\Z[\frac1n]$ as $\MM_{cub}$ is. Thus $\Spec C$ is regular and its connectedness implies that $\Spec C$ is irreducible by \cite[Corollary 10.14]{Eisenbud} and \cite[Exercise 3.16]{GoertzWedhorn}; hence $C$ is an integral domain. 

Furthermore, we show that $\Delta \in C$ is nonzero. Indeed, we can check this in $\Spec C_A \cong \Spec C \times_{\MM_{cub}} \Spec A$. The non-vanishing locus of $\Delta$ in $A$ is dense (as $A$ is an integral domain) and $\Spec C_A \to \Spec A$ is smooth and hence open. Thus, the image of $\Delta$ in $C_A$ is nonzero and $\Delta\neq 0$ in $C$ as well. This implies that $C \to C[\Delta^{-1}]$ and hence $C\to R_C[\Delta^{-1}]$ are indeed injective. 

According to \cref{lem:pullback} again, the pullback $\Spec A \times_{\MM_{cub}} \Spec R/\GG_m$ is equivalent to $\Spec R_A$. Thus, $\Spec R/\GG_m$ is finite over $\MM_{cub}$. We see that $R_C$ is finite over $C$ and hence every element of $R_C$ is integral over $C$. As $R$ is normal and $R_C$ is smooth over $R$, also $R_C$ is normal \cite[Tag 033C]{stacks-project}. Thus, every element that is integral over $C$ (and hence $R_C$) in $R_C[\Delta^{-1}]$ is already in $R_C$. Thus, $R_C$ is the normalization of $C$ in $R_C[\Delta^{-1}]$ and we obtain part \ref{normal}. 

For part \ref{CM}, we observe first that the flatness of $\Spec R_A \to \Spec A$ indeed implies the flatness of $\Spec R/\GG_m \to \MM_{cub}$ as the former map is by \cref{lem:pullback} the base change of the latter map along the fpqc morphism $\Spec A \to \MM_{cub}$. Moreover it suffices to show the flatness of the base change $R_{A_{\mathfrak{p}}} = R_A \tensor_A A_{\mathfrak{p}}$ over $A_{\mathfrak{p}}$ for all prime ideals $\mathfrak{p}$ in $A$. We want to envoke Hironaka's flatness criterion \cite[25.16]{Nagata}, by which a ring extension $A_{\mathfrak{p}}\subset B$ is automatically flat if $B$ is Cohen--Macaulay and finite as an $A_{\mathfrak{p}}$-module. The ring $R_A$ is Cohen--Macaulay as it is by base change smooth over $R$ and being Cohen--Macaulay is local in the smooth topology \cite[\href{https://stacks.math.columbia.edu/tag/036B}{Tag 036B}]{stacks-project}. Moreover, any localization of a Cohen--Macaulay ring is Cohen--Macaulay again \cite[Theorem 2.1.3]{BrunsHerzog} and thus $R_{A_{\p}}$ is indeed Cohen--Macaulay. To apply Hironaka's criterion, it suffices thus to show that the map $A_{\mathfrak{p}} \to R_{A_{\p}}$ is injective. 

We argue similarly to part \ref{normal}. As $A$ is an integral domain, the map $A_{\mathfrak{p}} \to A_{\mathfrak{p}}[\Delta^{-1}]$ is an injection. Thus, it suffices to show the injectivity of $A_{\mathfrak{p}}[\Delta^{-1}] \to R_{A_{\mathfrak{p}}}[\Delta^{-1}]$. But $\Spec R_{A_{\mathfrak{p}}}[\Delta^{-1}] \to \Spec A_{\mathfrak{p}}[\Delta^{-1}]$ is surjective as the base change of $\Spec R/\GG_m \to \MM_{ell}$ along $\Spec A_{\mathfrak{p}}[\Delta^{-1}]\to \MM_{ell}$, and the reducedness of $A$ implies thus that $A_{\mathfrak{p}}[\Delta^{-1}] \to R_{A_{\mathfrak{p}}}[\Delta^{-1}]$ is indeed injective.  
\end{proof}

We will apply this criterion to $R = \MF_1(n)$. While it was already shown that $\MF_1(n)$ is always normal and often Cohen--Macaulay in the last section, the next thing to check is the finiteness of $R_A$ over $A$ in this case. A crucial ingredient is the following proposition. 

\begin{prop}\label{prop:Conrad}
The maps $\MMb_1(n) \to \MMb_{ell}$ and $\MMb_0(n) \to \MMb_{ell}$ are finite and flat.  In particular, also $\MM_1(n) \to \MM_{ell}$ and $\MM_0(n) \to \MM_{ell}$ are finite and flat. The degree $d_n$ of $\MMb_1(n) \to \MMb_{ell}$ satisfies $d_n = n^2\prod_{p|n}(1-\frac1{p^2})$ and $\MMb_0(n) \to \MMb_{ell}$ is of degree $\frac{d_n}{\varphi(n)}$ for $\varphi$ Euler's totient function. 
\end{prop}
\begin{proof}
 The first part is contained in Theorem 4.1.1 of \cite{Conrad}. For the formula for $d_n$ note that the degree of $\MMb_1(n) \to \MMb_{ell}$ agrees with that of $\MMb_1(n)_{\C} \to \MMb_{ell,\C}$ as $\MMb_{ell}$ is connected. As recalled in \cref{sec:Comparisons}, for $n\geq 5$ the analytification of $\MMb_1(n)_{\C}$ agrees with $X_1(n)$ and as the generic point of $\MMb_{ell,\C}$ has automorphism group of order $2$, the degree $d_n$ is twice the degree of $X_1(n) \to X_1(1)$, which is computed in \cite[Sections 3.8+3.9]{DiamondShurman}. The cases $n=2,3$ and $4$ are easily computed by hand. 
 
The degree of $\MMb_0(n) \to \MMb_{ell}$ agrees with that of $\MM_0(n) \to \MM_{ell}$. As $\MM_1(n) \to \MM_0(n)$ is a $(\Z/n)^\times$-Galois cover, it has degree $\varphi(n)$. The formula for the degree of $\MM_0(n) \to \MM_{ell}$ follows. 
\end{proof}

Before we come to the next proposition, consider again the $\GG_m$-torsor $\MMb_1^1(n) \to \MMb_1(n)$ that trivializes $\omline$. As $H^0(\MMb_1^1(n), \OO_{\MMb_1^1(n)}) = \MF_1(n)$, we obtain a $\GG_m$-equivariant map $\MMb_1^1(n) \to \Spec \MF_1(n)$ and thus $\MMb_1(n) \to \Spec \MF_1(n)/\GG_m$, where the $\GG_m$-action on $\Spec \MF_1(n)$ corresponds to the standard grading on the ring of modular forms. 

\begin{lemma}\label{cubic}
  The map $\MMb_1(n) \to \MMb_{ell} \to \MM_{cub}$ factors over $\Spec \MF_1(n)/\GG_m$, resulting in the following commutative square:
\[
 \begin{tikzcd}
 \MMb_1(n) \arrow[r, "j"]\arrow[d, "f" swap] & \Spec \MF_1(n)/\GG_m \arrow[d, "\widetilde{f}"]\\
 \MMb_{ell} \arrow[r, "i" swap] & \MM_{cub}
  \end{tikzcd}
\]
\end{lemma}
\begin{proof}
\cref{prop:coordinatesm1n} yields a commutative square
\[
\begin{tikzcd}
 \MMb_1^1(n) \arrow[r]\arrow[dd] & \Spec \MF_1(n) \arrow[d]\\
 &\Spec A \arrow[d]\\
 \MMb_{ell} \arrow[r] & \MM_{cub}
 \end{tikzcd}
\]

Quotiening by $\GG_m$ gives the result since the map $A \to \MF_1(n)$ is grading preserving (i.e.\ $|a_i| = i$). 
\end{proof}

\begin{prop}\label{prop:finiteness}
The $A$-module $R_A = \Gamma \tensor_A \MF_1(n)$ is finite. 
\end{prop} 
\begin{proof}
Set $R = \MF_1(n)$. We will use the graded ring map $R \to A$ classifying a Weierstra{\ss} equation for the universal elliptic curve over $\MMb_1(n)$ and recall the resulting commutative square from \cref{cubic}:
\[
\begin{tikzcd}
 \MMb_1(n) \arrow[r, "j"]\arrow[d, "f" swap] & \Spec R/\GG_m \arrow[d, "\widetilde{f}"]\\
 \MMb_{ell} \arrow[r, "i" swap] & \MM_{cub}
 \end{tikzcd}
\]
As a quasi-coherent sheaf on $\Spec R/\GG_m$ is determined by its graded global sections as explained in \cref{GradedModulesQCoh}, we see that $j_*\OO_{\MMb_1(n)}$ is exactly $\OO_{\Spec R/\GG_m}$. Consider the cartesian square
\[
\begin{tikzcd}
 U \arrow[r, "k"]\arrow[d, "q" swap] & \Spec A \arrow[d, "p"] \\
 \MMb_{ell} \arrow[r, "i" swap] & \MM_{cub},
\end{tikzcd}
\]
where $k$ is an open immersion onto the complement of the common vanishing locus $V(c_4,\Delta)$ of $c_4$ and $\Delta$ by \cite[Proposition III.1.4]{SilvermanAEC}.
 As $\Spec R_A \cong \Spec A \times_{\MM_{cub}} \Spec R/\GG_m$ by \cref{lem:pullback}, we see that $R_A$ are the global sections of 
$$p^*\widetilde{f}_*\OO_{\Spec R/\GG_m} \cong p^*\widetilde{f}_*j_*\OO_{\MMb_1(n)} \cong p^*i_*f_*\OO_{\MMb_1(n)}.$$
As $p$ is flat, we have an isomorphism $p^*i_*f_*\OO_{\MMb_1(n)} \cong k_*q^*f_*\OO_{\MMb_1(n)}$. As $f$ is finite flat by the last proposition, $q^*f_*\OO_{\MMb_1(n)}$ is a vector bundle. We claim that for every reflexive sheaf $\FF$ on $U$ the pushforward $k_*\FF$ is reflexive and hence coherent. In particular, this would imply that $\Gamma(k_*q^*f_*\OO_{\MMb_1(n)}) = R_A$ is a finitely generated $A$-module if we apply the claim to $\FF = q^*f_*\OO_{\MMb_1(n)}$. 

To finish the proof, let $\FF$ be a reflexive sheaf on $U$. It is possible to extend $\FF$ to a reflexive sheaf $\EE$ on $\Spec A$ (see e.g.\ \cite[Lemma 3.2]{MeierVB}). By \cite[Proposition 1.6]{HartshorneReflexive}, we see that $k_*\FF \cong k_*k^*\EE \cong \EE$ as $A$ is normal and the complement $V(c_4,\Delta)$ of $U$ has codimension $2$. Thus, $k_*\FF$ is reflexive. 
\end{proof}

\begin{Theorem}\label{prop:flatnessm1n}\label{prop:Gamma1(n)flatness}
For every $n\geq 2$, we have an equivalence 
\[\MM_1(n)_{cub}\simeq \Spec \MF_1(n)/\GG_m\]
and more generally $\MM_1(n)_{cub, B}$ is equivalent to $\Spec \MF(\Gamma_1(n); B)/\GG_m$ for every flat $\Z[\frac1n]$-algebra $B$. 

Moreover, if $\MF_1(n)_1 \to \MF_1(\Gamma_1(n); \Fb_l)$ is surjective for all primes $l$ not dividing $n$, the map $\MM_1(n)_{cub}\to \MM_{cub}$ is flat. This is in particular true for all $n\leq 28$.
\end{Theorem}
\begin{proof}
The stack $\MM_1^1(n)$ is representable by an affine scheme for $n\geq 2$ (see e.g\ \cite[Proposition 2.4, Example 2.5]{MeierDecomposition}). As $\MFC_1(n) = \MFC(\Gamma_1(n);\Z[\frac1n])$ coincides with the global sections of $\OO_{\MM_1^1(n)}$, we obtain $\MM_1^1(n) \simeq \Spec \MFC_1(n)$ and thus $\MM_1(n) = \Spec \MFC_1(n)/\GG_m$ for all $n\geq 2$. As $\MM_1(n) \to \MM_{ell}$ is surjective, we see that $\Spec \MFC_1(n)/\GG_m \to \MM_{ell}$ is surjective as well. Thus the identification $\MF_1(n)[\Delta^{-1}] \cong \MFC_1(n)$ from \cref{sec:analytic} together with the commutative diagram from the proof of \cref{cubic} shows the first condition of \cref{lem:finitenessandnormalization}.

\Cref{prop:finiteness} and \cref{prop:normal} imply that we can apply the criterion \cref{lem:finitenessandnormalization} to conclude that $\MM_1(n)_{cub}\simeq \Spec \MF_1(n)/\GG_m$. This implies $\MM_1(n)_{cub, B} \simeq\Spec \MF(\Gamma_1(n); B)/\GG_m$ for any flat $\Z[\frac1n]$-algebra $B$ as $\MF(\Gamma_1(n); B) \cong \MF_1(n)\tensor B$ by \cref{lem:basechange}.

The statement about flatness follows again from \cref{lem:finitenessandnormalization} if we use \cref{prop:finiteness}, \cref{prop:CM} and \cref{exa:CM}.
\end{proof}

Whether a corresponding flatness result exists for $\MM_0(n)_{cub}$ remains open, although our main algebraic result \cref{thm:MainTheorem} provides such a result in a special case. 

Next we will fix the notation already introduced in the proof of \cref{lem:finitenessandnormalization}, specializing to $R = \MF_1(n)$.  
\begin{Notation}\label{notation}
For a given $n$ and an $A$-algebra $C$, we define $R_C$ and $S_C$ by the equivalences 
\begin{alignat*}{2}
\Spec R_C &\simeq&& \Spec C\times_{\MM_{cub}} \MM_1(n)_{cub}\\
\Spec S_C &\simeq&& \Spec C\times_{\MM_{cub}}\MM_0(n)_{cub} .
\end{alignat*}
Here we use that by construction $\MM_1(n)_{cub} \to \MM_{cub}$ and $\MM_0(n)_{cub} \to \MM_{cub}$ are affine. 
We will show in the next lemma that $S_C = (R_C)^{(\Z/n)^\times}$. 
\end{Notation}

As normalization commutes with smooth base change by \cref{lem:smoothbasechange}, the stack $\MM_1(n)_{cub}\times_{\MM_{cub}} \MM_{ell}$ is the normalization of $\MM_{ell}$ in $\MM_1(n)$, which is $\MM_1(n)$ itself as $\MM_1(n) \to \MM_{ell}$ is finite by \cref{prop:Conrad}. Thus
\[\Spec C \times_{\MM_{cub}} \MM_1(n) \simeq \Spec R_C \times_{\MM_{cub}} \MM_{ell}.\]
As $R_C$ is an $A$-algebra and $\Spec A \times_{\MM_{cub}}\MM_{ell} \simeq \Spec A[\Delta^{-1}]$ by \cref{MellopenMcub}, we obtain an equivalence $\Spec C \times_{\MM_{cub}} \MM_1(n) \simeq \Spec R_C[\Delta^{-1}]$ that forms a commutative square with defining equivalence of $R_C$ and the obvious maps. Similarly, we obtain $\Spec C \times_{\MM_{cub}} \MM_0(n) \simeq \Spec S_C[\Delta^{-1}]$ with the analogous property. 

\begin{Lemma} \label{SCasFP}
Let $C$ be an $A$-algebra such that the composite $\Spec C \to \Spec A \to \MM_{cub}$ is smooth.
\begin{enumerate}
    \item The map $S_C \to R_C^{(\Z/n)^\times}$ is an isomorphism.
    \item The ring of invariants $R_{C[\Delta^{-1}]}^{(\Z/n)^\times}$ is projective over $C[\Delta^{-1}]$. Its rank is precisely the degree of the map $\MMb_0(n)\to \MMb_{ell}$.  
\end{enumerate}
\end{Lemma}
\noindent A formula for the degree of $\MMb_0(n)\to \MMb_{ell}$ was recalled in \cref{prop:Conrad}.
\begin{proof}
We start by analyzing the situation after inverting $\Delta$. As the map $\MM_1(n) \to \MM_0(n)$ is a $(\Z/n)^\times$-torsor, the pullback 
\[\Spec C \times_{\MM_{cub}}\MM_1(n) \to \Spec C \times_{\MM_{cub}} \MM_0(n)\]
is a $(\Z/n)^\times$-torsor as well. This map can be identified with $\Spec R_C[\Delta^{-1}] \to \Spec S_C[\Delta^{-1}]$ and thus the map $S_C[\Delta^{-1}] \to R_C[\Delta^{-1}]^{(\Z/n)^\times}$ is an isomorphism. As $\Spec S_C[\Delta^{-1}] \to \Spec C[\Delta^{-1}]$ agrees with the base change of $\MM_0(n) \to \MM_{ell}$ along $\Spec C[\Delta^{-1}]\to \MM_{ell}$, we see that $S_C[\Delta^{-1}]$ is projective module over $C[\Delta^{-1}]$, whose rank agrees with the degree of $\MM_0(n) \to \MM_{ell}$ (or equivalently of $\MMb_0(n) \to \MMb_{ell}$).

Next we want to show that the map $S_C \rightarrow (R_C)^{(\Z/n)^\times}$ is an isomorphism. By \cref{lem:smoothbasechange}, $S_C$ consists of those elements in $S_C[\Delta^{-1}]$ that are integral over $C$ and in particular $S_C \to S_C[\Delta^{-1}]$ is an injection. For analogous reasons $R_C \to R_C[\Delta^{-1}]$ is injective as well. 
Thus, the two maps in the composition
\begin{align*}S_C \to R_C^{(\Z/n)^\times} \to R_C^{(\Z/n)^\times}[\Delta^{-1}]\cong S_C[\Delta^{-1}]\end{align*}
are injections. Hence, it remains to show that every element in $R_C^{(\Z/n)^\times}$ is integral over $C$ to obtain that $S_C \to R_C^{(\Z/n)^\times}$ is surjective as well. 

As $R_A$ is a finite $A$-module by \cref{prop:Gamma1(n)flatness}, $R_C \cong C \tensor_A R_A$ is a finite $C$-module. Moreover, $C$ is noetherian as it is smooth and hence finitely presented over the stack $\MM_{cub}$ and the latter is noetherian because $A$ is. Hence, $(R_C)^{(\Z/n)^\times}$ is finite over $C$ and thus every element of it is indeed integral over $C$. 
\end{proof}

\section{Computation of invariants}
\label{sec:Splitting}
In order to understand $\MM_0(7)_{cub}$ and $\MM_1(7)_{cub}$, we will perform in this section crucial preliminary calculations. 
Recall that at the prime $3$, there is a smooth cover 
$\Spec \widetilde{A} \to \MM_{cub}$, where $ \widetilde{A} :=\mathbb{Z}_{(3)}[\aq_2,\aq_4,\aq_6]$ as discussed in \cref{StacksHopfElliptic}. This defines rings $R_{\widetilde{A}}$ and $S_{\widetilde{A}}$ as in \cref{notation}. \cref{SCasFP} identifies $S_{\widetilde{A}}$ with the invariants $R_{\widetilde{A}}^{(\Z/7)^\times}$. Here and in the following we consider the case $n=7$ so that \[\Spec R_{\widetilde{A}} \simeq \Spec \widetilde{A}\times_{\MM_{cub}}\MM_1(7)_{cub}.\]
Our main goal in this section is to compute explicitly $R_{\widetilde{A}}$ together with its $(\Z/7)^\times$-action and especially the invariants $S_{\widetilde{A}}$. Later we will see that the pushforward of $\OO_{\MM_0(n)_{cub}}$ to $\MM_{cub}$ corresponds under the equivalence from \cref{AtildevsQCoh} to an $(\widetilde{A},\widetilde{\Gamma})$-comodule structure on $S_{\widetilde{A}}$. Thus, the computation of $S_{\widetilde{A}}$ will be key to our splitting result \cref{thm:MainTheorem}. Throughout this section, we localize implicitly at the prime $3$.

 We recall from \cref{StacksHopfElliptic} the graded Hopf algebroid $(\widetilde{A}, \widetilde{\Gamma})$. We have an isomorphism $\widetilde{\Gamma}\cong {\widetilde{A}}[r]$ and $\eta_R$ is determined under this identification by
\[
 \begin{aligned}
 \eta_R(\aq_2)&=\aq_2+3r,\\
  \eta_R(\aq_4)&=\aq_4+2r\aq_2+3r^2,\\
  \eta_R(\aq_6)&=\aq_6+r\aq_4+r^2\aq_2+r^3,
 \end{aligned}
\]
whereas $\eta_L$ is the canonical inclusion of $\widetilde{A}$.

We transform the Tate normal form of the universal cubical curve over 
\[
\MF_1(7)\cong\Z_{(3)}[z_1,z_2,z_3]/(z_1z_2+z_2z_3+z_3z_1)
\]
into the Weierstra\ss{} form
 \[
y^2=x^3+\kappa(\aq_2)x^2+\kappa(\aq_4)+\kappa(\aq_6),
\]
determining a map $\kappa\colon\widetilde{A}\to \mf$ which makes the diagram 
\[
\begin{tikzcd}
\Spec \MF_1(7) \arrow[r, "\kappa"] \arrow[d] & \Spec \widetilde{A}\arrow[d]\\
\MM_1(7)_{cub}\arrow[r] & \MM_{cub}
\end{tikzcd}
\]
commutative. Using \cref{Alphas}, we compute this map $\kappa\colon\widetilde{A}\to \mf$  to be given by 
\begin{equation*}\label{Betas}
 \begin{aligned}
  \aq_2&\mapsto\frac{1}{4}\alpha_1^2+\alpha_2=\frac{1}{4}(z_1-z_2+z_3)^2-z_2z_3,\\
  \aq_4&\mapsto\frac{1}{2}\alpha_1\alpha_3=\frac{1}{2}z_1z_3^2(z_1-z_2+z_3),\\
  \aq_6&\mapsto\frac{1}{4}\alpha_3^2=\frac{1}{4}z_1^2z_3^4.
 \end{aligned}
\end{equation*}

Using \cref{lem:pullback} and \cref{AtildevsQCoh}, the map $\widetilde{A} \to \mf$ allows us to rewrite $R_{\widetilde{A}}$ as follows:
\[
R_{\widetilde{A}}\cong\widetilde{\Gamma} { }_{\eta_R}\underset{\widetilde{A}}{\otimes} \mf .
\] 

\begin{prop}\label{RAtildefree}
 $R_{\widetilde{A}}$ is a free $\widetilde{A}$-module of rank $48$.
 \end{prop}
 \begin{proof}
Recall that $\Spec R_{\widetilde{A}}$ is the pullback $\Spec \widetilde{A}\times_{\MM_{cub}}\MM_1(7)_{cub}$. The map $\MM_1(7)_{cub} \to \MM_{cub}$ is finite and flat by \cref{lem:finitem1n}, \cref{prop:flatnessm1n} and \cref{exa:normal}. Thus, $R_{\widetilde{A}}$ is a finite projective module over $\widetilde{A}$. As $\widetilde{A}$ is a polynomial ring over a discrete valuation ring, the Quillen--Suslin Theorem \cite{QuillenProjective}, \cite{SuslinProjective} implies that $R_{\widetilde{A}}$ is already free. Its rank coincides with the degree of the map $\MM_1(7)_{cub} \to \MM_{cub}$ that coincides with that of the restriction $\MMb_1(7) \to \MMb_{ell}$. By \cref{prop:Conrad}, this is $7^2-1 = 48$. 
 \end{proof}

We want to identify $R_{\widetilde{A}}$ with  $\mf[r]$. The tensor product $R_{\widetilde{A}}\cong\widetilde{\Gamma} { }_{\eta_R}\underset{\widetilde{A}}{\otimes} \mf$ can be described as
\[
R_{\widetilde{A}} \cong \mf[\aq_2, \aq_4, \aq_6, r]/(\eta_R(\aq_2)=\kappa(\aq_2), \eta_R(\aq_4)=\kappa(\aq_4), \eta_R(\aq_6)=\kappa(\aq_6)). 
\]
Looking closely at the formulae, we can eliminate $\aq_2, \aq_4, \aq_6$ and this yields a ring isomorphism to $\mf[r]$. The resulting composite
\[\lambda\colon \widetilde{A} \xrightarrow{\eta_L\tensor 1} \widetilde{\Gamma}\tensor_{\widetilde{A}} \mf \cong \mf[r] \]
defines a rather complicated $\widetilde{A}$-module structure. Concretely it is given by:
\[
 \begin{aligned}
\aq_2&\mapsto&&\frac{1}{4}(z_1-z_2+z_3)^2-z_2z_3-3r,\\
 \aq_4 &\mapsto&& \frac{1}{2}z_1z_3^2(z_1-z_2+z_3)-2r(\frac{1}{4}(z_1-z_2+z_3)^2-z_2z_3-3r) -3r^2,\\
  \aq_6&\mapsto&&\frac{1}{4}z_1^2z_3^4-r(\frac{1}{2}z_1z_3^2(z_1-z_2+z_3)
  -2r(\frac{1}{4}(z_1-z_2+z_3)^2-z_2z_3-3r)-3r^2)\\
  &&&-r^2(\frac{1}{4}(z_1-z_2+z_3)^2-z_2z_3-3r) -r^3. 
 \end{aligned}
\]
The map $\lambda$ corresponds to the projection $\Spec \widetilde{A}\times_{\MM_{cub}}\MM_1(7)_{cub} \to \Spec \widetilde{A}$ under the identification of the source with $\Spec \mf[r]$.

  Our next aim is to make \cref{RAtildefree} explicit. More precisely, we claim that there is an $\widetilde{A}$-basis of $R_{\widetilde{A}}$ of the form $X\sqcup Xr\sqcup Xr^2$, where $X$ is a $16$-element subset of the image of $\mf$ in $R_{\widetilde{A}}$. To prove this, we will use the following graded version of the Nakayama lemma.

\begin{lemma} \label{GradedNakayama}
Let $R$ be a nonnegatively graded commutative ring such that $R_0$ is local with maximal ideal $\mathfrak{m}_0$. Let $\mathfrak{m}$ be the homogeneous ideal generated by $\mathfrak{m}_0$ and the ideal of all homogeneous elements of positive degree. Let furthermore $M$ and $N$ be nonnegatively graded $R$-modules that are finitely generated over $R_0$ in every degree. Then a map $M \to N$ of graded $R$-modules is surjective if $M/\mathfrak{m} \to N/\mathfrak{m}$ is surjective.
\end{lemma}
\begin{proof}
It suffices to show that $N = 0$ if $N/\mathfrak{m} = 0$. By the usual Nakayama lemma it suffices to show that $N/(\mathfrak{m}_0)$ is zero. Assume otherwise and let $i$ be the minimal non-vanishing degree of $N/(\mathfrak{m}_0)$ and hence of $N$. As $(\mathfrak{m}N)_i = \mathfrak{m}_0N_i$, we have $(N/(\mathfrak{m}_0))_i \cong (N/\mathfrak{m})_i$. Thus $(N/(\mathfrak{m}_0))_i $ vanishes as well.
\end{proof}
 
 Recall the notation $\sigma_1=z_1+z_2+z_3$ and $\sigma_3=z_1z_2z_3$ for elementary symmetric polynomials in $z_i$. 
 \begin{lemma}
The subset 
  \[
  \begin{aligned}
   X=&\{1\}\cup\{\sigma_1, z_2, z_3\}\cup\{ \sigma_1^2, \sigma_1z_2, \sigma_1z_3, z_2z_3\}\cup\{\sigma_1^3, \sigma_1^2z_2, \sigma_1^2z_3, \sigma_3\}\\
   &\cup\{  \sigma_1^4, \sigma_1^3z_2, \sigma_1^3z_3\} \cup \{\sigma_1^4z_2\}.
   \end{aligned}
  \]
  of $R_{\widetilde{A}}$ gives an $\widetilde{A}$-basis of $R_{\widetilde{A}}$ of the form $X\sqcup Xr\sqcup Xr^2$.
  \end{lemma}

\begin{proof}
As we already know by \cref{RAtildefree} that $R_{\widetilde{A}}$ is a free $\widetilde{A}$-module of rank $48$ and $\widetilde{A}$ is noetherian, it is enough to show that $X\sqcup Xr\sqcup Xr^2$ is a generating system (since it has precisely $48$ elements).  

We want to apply the graded Nakayama \cref{GradedNakayama} to the ideal $I = (3,\aq_2,\aq_4, \aq_6)$ in the ring $\widetilde{A}$. 
Thus, it is enough to show that the images of $X\sqcup Xr\sqcup Xr^2$ form a basis of $R_{\widetilde{A}}/I$. This is done by the following \texttt{MAGMA} code. 

\begin{verbatim}
F3:=FiniteField(3);
M<z1, z2, z3, r>:=PolynomialRing(F3,4);
ka2:=(z1-z2+z3)^2/4-z2*z3;
ka4:=z1*z3^2*(z1-z2+z3)/2;
ka6:=z1^2*z3^4/4;
la2:=ka2-3*r;
la4:=ka4-2*r*la2-3*r^2;
la6:=ka6-r*la4-r^2*la2-r^3;
RAtildeModI:=quo<M|z1*z2+z2*z3+z3*z1,la2,la4,la6>;
RAtildeModIasVSp, pr:=VectorSpace(RAtildeModI);
Dimension(RAtildeModIasVSp);
sigma1:=z1+z2+z3;
sigma3:=z1*z2*z3;
X:={1, sigma1, z2, z3, sigma1^2, sigma1*z2, sigma1*z3, z2*z3,
	sigma1^3, sigma1^2*z2, sigma1^2*z3, sigma3, 
	sigma1^4, sigma1^3*z2, sigma1^3*z3, sigma1^4*z2};
Xr:={x*r: x in X};
Xr2:={x*r^2: x in X};
IsIndependent(pr(X) join pr(Xr) join pr(Xr2));
\end{verbatim}

Here \texttt{ka2} denotes $\kappa(\aq_2)$ and $\texttt{la2}$ denotes $\lambda(\aq_2)$ etc. We first check the quotient $R_{\widetilde{A}}/I$ to be $48$-dimensional as an $\mathbb{F}_3$vector space and then show that $X\sqcup Xr\sqcup Xr^2$ is linearly independent. 
\end{proof}

Now that we have some understanding of $R_{\widetilde{A}}$ as an ${\widetilde{A}}$-module, we can look at the $(\Z/7)^{\times}$-action on it and the invariants under this action. 
Recall we have chosen the generator $\tau = [3] \in \left(\Z/7\right)^{\times}$ and shown it to act on the $z_i \in \MF_1(7)$ via
\[
 \begin{aligned}
  \tau(z_1)=-z_3,\\
  \tau(z_2)=-z_1,\\
  \tau(z_3)=-z_2. 
 \end{aligned}
\]
This grading-preserving action induces an action on $R_{\widetilde{A}}$ by the identification of its spectrum with $\Spec \widetilde{A}\times_{\MM_{cub}}(\Spec \mf/\GG_m)$ via \cref{prop:Gamma1(n)flatness} and thus on $\mf[r]$ as well. By definition, the projections onto both factors are $(\Z/7)^\times$-equivariant with the trivial action on $\widetilde{A}$ and the action above on $\mf$. Thus, the obvious inclusion $\mf \to \mf[r]$ is $(\Z/7)^\times$-equivariant and so is $\lambda\colon \widetilde{A} \to \mf[r]$. In particular, $\tau(\lambda(\aq_2)) = \lambda(\aq_2)$ enforces $\tau(r)=r+z_2z_3$. 

The computation of invariants relies again on \texttt{MAGMA} computations. We will list now some elements which can be checked to be invariant, and the remainder of the section is devoted to the proof that these elements actually form a basis of $S_{\widetilde{A}}$ as an $\widetilde{A}$-module, in particular proving that this module is free. 

We consider the elements $1, \sigma_1^2, \sigma_1^4, \sigma_3^2$ as well as 
\[
\begin{aligned}
 n_4&:=\sigma_1^2r -z_1^3z_3-z_1z_2^3-z_1^2z_3^2,\\
 \sigma_1^2n_4 &=\sigma_1^4r-\sigma_1^2\cdot (z_1^3z_3+z_1z_2^3+z_1^2z_3^2),\\
 n_6 &:=\sigma_1^2r^2-2z_1^3z_3r-2z_1z_2^3r-2z_1^2z_3^2r+2z_1^3z_3^3-z_1^2z_3^4\\
 &=2n_4r-\sigma_1^2r^2+2z_1^3z_3^3-z_1^2z_3^4,\\
 \sigma_1^2n_6 &=\sigma_1^2\cdot(\sigma_1^2r^2-2z_1^3z_3r-2z_1z_2^3r-2z_1^2z_3^2r+2z_1^3z_3^3-z_1^2z_3^4),
 \end{aligned}
\]
which we claim to elements in $S_{\widetilde{A}}$.

Indeed, to check that the non-obvious elements $n_4$ and $n_6$ are invariant, we use \texttt{MAGMA}. For the example of $n_4$, we have used the following code: 
\begin{verbatim}
QQ:=RationalField();
M<z1, z2, z3, r>:=PolynomialRing(QQ,4);
RAtildeQ:=quo<M|z1*z2+z2*z3+z3*z1>;
tau:=hom<M ->RAtildeQ| -z3, -z1, -z2, r+z2*z3>;
proj:=hom<M ->RAtildeQ| z1, z2, z3, r>;
sigma1:=z1+z2+z3;
n4:=sigma1^2*r -z1^3*z3-z1*z2^3-z1^2*z3^2;
tau(n4)-proj(n4);
\end{verbatim}

Our aim is to prove the following proposition. 
\begin{prop}\label{SBbasis}
 The elements
 \[
  1, \sigma_1^2, \sigma_1^4, n_4, \sigma_1^2n_4, n_6, \sigma_1^2n_6, \sigma_3^2
 \]
 form a $\widetilde{A}$-basis of $S_{\widetilde{A}}$; in particular, $S_{\widetilde{A}}$ is a free $\widetilde{A}$-module of rank $8$. 
\end{prop}

\begin{proof}
The proof will proceed in several steps. 

\begin{enumerate}[wide, labelwidth=!, labelindent=0pt, label=\textbf{Step \arabic{enumi}:}]
    \item As a first step, we compute using \texttt{MAGMA} the following expressions for the listed invariants in terms of the basis $X\sqcup Xr\sqcup Xr^2$: 
\begin{align*}
 1=& 1\\
 \sigma_1^2=& \sigma_1^2\\
 \sigma_1^4=& \sigma_1^4\\
 n_4=&\frac{1}{2}\sigma_1^3z_3+4\sigma_1^2r-6\sigma_1z_3r
-2\overline{a}_2\sigma_1z_3+\overline{a}_2\sigma_1^2-4\overline{a}_2^2+12\overline{a}_4\\
n_6=&-\frac{33}{32}  \sigma_1^4 r+\frac{3}{8}  \sigma_1^3  z_2 r+\frac{13}{4}  \sigma_1^3  z_3 r
+\frac{233}{8}  \sigma_1^2 r^2-\frac{21}{4}  \sigma_1  z_2 r^2-42  \sigma_1  z_3 r^2\\
&+\frac{3}{2}  z_2  z_3 r^2-18 \overline{a}_6
-\frac{7}{8}  \overline{a}_4  \sigma_1^2-\frac{1}{4}  \overline{a}_4  \sigma_1  z_2- \overline{a}_4  \sigma_1  z_3+\frac{13}{2}  \overline{a}_4  z_2  z_3
+\frac{123}{2}  \overline{a}_4 r\\
&-\frac{11}{2}  \overline{a}_2^3+\frac{11}{4}  \overline{a}_2^2  \sigma_1^2-\frac{1}{2}  \overline{a}_2^2  \sigma_1  z_2-3  \overline{a}_2^2  \sigma_1  z_3-2  \overline{a}_2^2  z_2  z_3-\frac{41}{2}  \overline{a}_2^2 r\\
&+\frac{37}{2}  \overline{a}_2 \overline{a}_4-\frac{11}{32}  \overline{a}_2  \sigma_1^4+\frac{1}{8} \overline{a}_2  \sigma_1^3  z_2+\frac{3}{4}  \overline{a}_2  \sigma_1^3  z_3+\frac{67}{4} \overline{a}_2  \sigma_1^2 r\\
&-\frac{7}{2} \overline{a}_2  \sigma_1  z_2 r-24  \overline{a}_2 \sigma_1  z_3 r+\overline{a}_2  z_2 z_3 r\\
\sigma_1^2n_4=&
-8 \sigma_1^4 r+6 \sigma_1^3 z_2 r+24 \sigma_1^3 z_3 r+252 \sigma_1^2 r^2-336 \sigma_1 z_3 r^2\\
&+24 \overline{a}_4 \sigma_1 z_2-16 \overline{a}_4 \sigma_1 z_3+96 \overline{a}_4 z_2 z_3+576 \overline{a}_4 r\\
&-64 \overline{a}_2  ^3+28 \overline{a}_2  ^2 \sigma_1^2-8 \overline{a}_2  ^2 \sigma_1 z_2
-32 \overline{a}_2  ^2 \sigma_1 z_3-32 \overline{a}_2  ^2 z_2 z_3-192 \overline{a}_2  ^2 r\\
&+192 \overline{a}_2   \overline{a}_4-3 \overline{a}_2   \sigma_1^4+2 
 \overline{a}_2   \sigma_1^3 z_2+8 \overline{a}_2   \sigma_1^3 z_3+168 \overline{a}_2   \sigma_1^2 r-224 \overline{a}_2   \sigma_1 z_3r\\
\sigma_3^2=&  \frac{81}{64}\sigma_1^4r-\frac{3}{16}\sigma_1^3z_2r-\frac{33}{8}\sigma_1^3 z_3 r
-\frac{537}{16} \sigma_1^2 r^2+\frac{69}{8} \sigma_1 z_2 r^2+51 \sigma_1 z_3 r^2-\frac{3}{4} z_2 z_3 r^2\\
&-9\overline{a}_6-\frac{17}{16} \overline{a}_4 \sigma_1^2 +\frac{17}{8} \overline{a}_4 \sigma_1 z_2  +\frac{1}{2} \overline{a}_4 \sigma_1 z_3  -\frac{13}{4} \overline{a}_4 z_2 z_3
 -\frac{267}{4} \overline{a}_4 r\\
 &+\frac{27}{4} \overline{a}_2^3-\frac{27}{8} \overline{a}_2^2 \sigma_1^2+\frac{1}{4} \overline{a}_2^2 \sigma_1 z_2+
 \frac{11}{2}\overline{a}_2^2 \sigma_1 z_3
+\overline{a}_2^2 z_2 z_3 
+\frac{89}{4} \overline{a}_2^2 r \\
&-\frac{77}{4} \overline{a}_2 \overline{a}_4
+\frac{27}{64} \overline{a}_2 \sigma_1^4
-\frac{1}{16} \overline{a}_2 \sigma_1^3 z_2
-\frac{11}{8} \overline{a}_2 \sigma_1^3 z_3\\
&-\frac{179}{8} \overline{a}_2 \sigma_1^2 r+\frac{23}{4} \overline{a}_2 \sigma_1 z_2 r +34 \overline{a}_2 \sigma_1 z_3 r-\frac{1}{2} \overline{a}_2 z_2 z_3 r\\
\sigma_1^2n_6=& \frac{5933}{3488} \sigma_1^4 r^2+\frac{7599}{872}  \sigma_1^3  z_2 r^2-\frac{255}{872}  \sigma_1^3  z_3 r^2\\
&+\frac{2997}{218}  \overline{a}_6  \sigma_1^2-\frac{11475}{109} \overline{a}_6  \sigma_1  z_2+\frac{816}{109} \overline{a}_6  \sigma_1  z_3\\
&-\frac{2339}{436}  \overline{a}_4  \sigma_1^4 +\frac{4267}{1744}  \overline{a}_4  \sigma_1^3  z_2 +\frac{21951}{1744}  \overline{a}_4  \sigma_1^3  z_3
+\frac{52113}{436}  \overline{a}_4  \sigma_1^2 r\\
&-\frac{28203}{436}  \overline{a}_4  \sigma_1  z_2 r-\frac{64187}{436}  \overline{a}_4  \sigma_1  z_3 r+\frac{2397}{109}  \overline{a}_4  z_2  z_3 r
-\frac{13005}{218}  \overline{a}_4 r^2 \\
&+\frac{16659}{109}  \overline{a}_4^2
+\frac{11279}{1744}  \overline{a}_2  \sigma_1^4 r+\frac{789}{436}  \overline{a}_2  \sigma_1^3  z_2 r-\frac{7061}{436}  \overline{a}_2  \sigma_1^3  z_3 r-168  \overline{a}_2  \sigma_1^2 r^2 \\
&+224  \overline{a}_2  \sigma_1  z_3 r^2
+\frac{15373}{436}  \overline{a}_2  \overline{a}_4  \sigma_1^2-\frac{1077}{436}  \overline{a}_2  \overline{a}_4  \sigma_1  z_2-\frac{17833}{436}  \overline{a}_2  \overline{a}_4  \sigma_1  z_3\\
&-\frac{6177}{109}  \overline{a}_2  \overline{a}_4  z_2  z_3-\frac{46191}{109}  \overline{a}_2  \overline{a}_4 r
+\frac{13485}{3488}  \overline{a}_2^2  \sigma_1^4-\frac{2059}{1744}  \overline{a}_2^2  \sigma_1^3  z_2\\
&-\frac{16675}{1744}  \overline{a}_2^2  \sigma_1^3  z_3-\frac{66203}{436}  \overline{a}_2^2  \sigma_1^2 r+\frac{9401}{436}  \overline{a}_2^2  \sigma_1  z_2 r+\frac{86505}{436}  \overline{a}_2^2  \sigma_1  z_3 r\\
&-\frac{799}{109}  \overline{a}_2^2  z_2  z_3 r+\frac{4335}{218}  \overline{a}_2^2 r^2-\frac{51561}{218}  \overline{a}_2^2  \overline{a}_4
+\frac{13485}{218}  \overline{a}_2^4-\frac{13485}{436}  \overline{a}_2^3  \sigma_1^2\\
&+\frac{2059}{436}  \overline{a}_2^3  \sigma_1  z_2+\frac{16675}{436}  \overline{a}_2^3 \sigma_1  z_3+\frac{2059}{109}  \overline{a}_2^3 z_2 z_3+\frac{15397}{109} \overline{a}_2^3 r
\end{align*}
 \item We want to show that the 8 invariants listed in the statement of the proposition are $\widetilde{A}$-linearly independent elements of $R_{\widetilde{A}}$. 
 Since $\widetilde{A}$ is torsion-free, it is enough to check linearly independency over $\widetilde{A}\otimes_{\mathbb{Z}}\mathbb{Q}$. We observe that there is a non-vanishing $8\times 8$-minor in the  $48\times 8$-matrix corresponding to the map $\widetilde{A}^8 \to R_{\widetilde{A}}\cong \widetilde{A}^{48}$ given by the invariants above. More precisely, after tensoring with $\mathbb{Q}$, the determinant of the following matrix
\[
\renewcommand\arraystretch{1.8}
\begin{blockarray}{ccccccccc}
& 1 & \sigma_1^2 & \sigma_1^4 & n_4 & n_6 & \sigma_1^2n_4 &\sigma_3^2 & \sigma_1^2n_6 \\
\begin{block}{c(cccccccc)}
  1 & 1 & * & * & * & * & * & * & * \\
  \sigma_1^2 & 0 & 1 & * & * & * & * & * & * \\
  \sigma_1^4 & 0 & 0 & 1 & * & * & * & * & *\\ 
  \sigma_1^2r & 0 & 0 & 0 & 4 & * & * & * & *\\
  \sigma_1^4r & 0 & 0 & 0& 0 & -\frac{33}{32} & -8 & \frac{81}{64} & * \\
  \sigma_1z_2r^2 & 0 & 0& 0& 0& -\frac{21}{4} & 0 & \frac{69}{8} & *\\
  z_2z_3r^2 & 0& 0& 0& 0& \frac{3}{2} & 0 & -\frac{3}{4} & *\\
  \sigma_1^4r^2 & 0 & 0& 0& 0& 0& 0& 0& \frac{5933}{3488}\\
\end{block}
\end{blockarray}
 \]
 is invertible in $\widetilde{A}\otimes_{\mathbb{Z}} \mathbb{Q}$. On the left, we recorded the elements of our chosen basis to which the selected $8$ out of $48$ rows correspond. This shows that the map $(\widetilde{A}\otimes_{\mathbb{Z}} \mathbb{Q})^8 \to R_{\widetilde{A}}\otimes_{\mathbb{Z}} \mathbb{Q}$ given by the invariants above is an inclusion of a direct $\widetilde{A}\otimes_{\mathbb{Z}} \mathbb{Q}$-summand.

Thus, we have shown that the $8$ invariants listed above are $\widetilde{A}$-linearly independent elements of $R_{\widetilde{A}}$, so they generate a free sub-$\widetilde{A}$-module of $R_{\widetilde{A}}$ of rank $8$, which we denote by $V$. 

\item Our next goal is to show that this module $V$ is already all of $S_{\widetilde{A}}$ when tensored with $\Q$. Recall that we identify $S_{\widetilde{A}}$ with $R_{\widetilde{A}}^{(\Z/7)^\times}$ as in \cref{SCasFP} and likewise $S_{\widetilde{A}\tensor \Q}$ with $R_{\widetilde{A}\tensor\Q}^{(\Z/7)^\times}$ .

Moreover, there is an isomorphism $S_{\widetilde{A}\otimes \mathbb{Q}} \cong S_{\widetilde{A}}\otimes_{\mathbb{Z}} \mathbb{Q}$ since $R_{\widetilde{A}}\otimes_{\mathbb{Z}}\mathbb{Q}$ can be written as a directed colimit of the form
\[
R_{\widetilde{A}} \xrightarrow{\cdot 2 \hphantom{\cdot}} R_{\widetilde{A}} \xrightarrow{\cdot 3 \hphantom{\cdot}} \ldots 
\]
and directed colimits commute with finite limits in the finitely presentable category of abelian groups (see e.g.\ \cite[Proposition 1.59]{AdamekRosicky}).

 As the order of $(\Z/7)^{\times}$ is invertible in $\widetilde{A}\otimes\Q$, the invariants $S_{\widetilde{A}\otimes \Q}$ are a direct summand  of the free $\widetilde{A}\otimes \Q$-module $R_{\widetilde{A}\otimes \Q}$ and thus projective. By the Quillen--Suslin Theorem, it implies that $S_{\widetilde{A}\otimes \Q}$ is also free, automatically of rank $8$ as this is true after inverting $\Delta$ by the second part of \cref{SCasFP}. 

Since the map $V\otimes \Q \to R_{\widetilde{A}}\otimes \Q$ is split injective, so is the map $V\otimes \Q \to S_{\widetilde{A}}\otimes \Q$. Since we have shown now both sides to be free $\widetilde{A}\otimes_{\mathbb{Z}} \mathbb{Q}$-modules of rank $8$, this map is also surjective and thus an isomorphism. 

\item In this step, we reduce the proof of the proposition to showing that the map $V\to S_{\widetilde{A}}$ (or to $R_{\widetilde{A}}$) 
is injective when we tensor it with $\mathbb{F}_3$. This will imply the surjectivity of $V\to S_{\widetilde{A}}$. Indeed, let $x\in S_{\widetilde{A}}$ be some element. By the rational statement, we know that there is an element $y$ in $V$  
and $k\in \mathbb{N}$ such that $3^kx=y$. If $k=0$, we are done; otherwise we can conclude that $y$ is mapped to $0$ in $R_{\widetilde{A}}$ after tensoring with $\mathbb{F}_3$, so by injectivity of the map $V\otimes \mathbb{F}_3\to R_{\widetilde{A}}\otimes \mathbb{F}_3$ it can be divided by $3$ in $V$. 
Inductively, this implies the claim.

\item Finally, we show that the map $V\to R_{\widetilde{A}}$ is still injective after tensoring $\mathbb{F}_3$. We do a similar computation for $\mathbb{F}_3$ as we did above rationally. This time, we consider the following $8\times 8$-minor of the $48\times 8$-matrix describing the inclusion $V\to R_{\widetilde{A}}$, again with the corresponding basis elements displayed on the left:
 \[
\renewcommand\arraystretch{1.8}
\begin{blockarray}{ccccccccc}
& 1 & \sigma_1^2 & \sigma_1^4 & n_4 & n_6 & \sigma_1^2n_4 &\sigma_3^2 & \sigma_1^2n_6 \\
\begin{block}{c(cccccccc)}
  1 & 1 & * & * & * & * & * & * & * \\
  \sigma_1^2 & 0 & 1 & * & * & * & * & * & * \\
  \sigma_1^4 & 0 & 0 & 1 & * & * & * & * & *\\ 
  \sigma_1^2r & 0 & 0 & 0 & 4 & \frac{67}{4} \overline{a}_2 & 168 \overline{a}_2 & -\frac{179}{8} \overline{a}_2 & *\\
  \sigma_1^3z_3& 0& 0& 0& \frac{1}{2}& \frac{3}{4}  \overline{a}_2 & 8 \overline{a}_2 &-\frac{11}{8} \overline{a}_2 &*\\
  \sigma_1^4r & 0 & 0 & 0& 0 & -\frac{33}{32} & -8 & \frac{81}{64} & * \\
  \sigma_1^3z_3r & 0 & 0& 0& 0& \frac{13}{4} & 24 & -\frac{33}{8} & *\\
  \sigma_1^4r^2 & 0 & 0& 0& 0& 0& 0& 0& \frac{5933}{3488}\\
\end{block}
\end{blockarray}
 \]
 Its determinant is a rational multiple of $\overline{a}_2$ not divisible by $3$. It shows that the map $V\to R_{\widetilde{A}}$ is still injective after tensoring with $\mathbb{F}_3$, as desired. This completes the proof of $V=S_{\widetilde{A}}$. 
 \end{enumerate}
 \end{proof}

\section{Comodule Structures}\label{sec:Comodule}
\addtocontents{toc}{\protect\setcounter{tocdepth}{1}}
Recall from \cref{StacksHopfElliptic} that we denote by $\widetilde{A}$ the ring $\widetilde{A}=\Z_{(3)}[\overline{a}_2,\overline{a}_4,\overline{a}_6]$. Recall moreover that we obtain a graded Hopf algebroid $(\widetilde{A},\widetilde{\Gamma})$ representing $\MM_{cub, \Z_{(3)}}$, and that quasi-coherent sheaves on $\MM_{cub, \Z_{(3)}}$ are equivalent to graded $(\widetilde{A},\widetilde{\Gamma})$-comodules. Thus it suffices for our main algebraic theorem to provide an isomorphism of certain comodules, which will describe explicitly.  

Throughout this section we will again (implicitly) localize everything at the prime $3$. Moreover, we will denote by $f_n$ the natural map $\MM_1(n) \to \MM_{ell}$ and by $f'_n$ the resulting map $\MM_1(n)_{cub} \to \MM_{cub}$ from the normalization. In the case $n=2$, we will use the abbreviations $f = f_2$ and $f' = f_2'$. Lastly, we use $h_n$ for the natural map $\MM_0(n) \to \MM_{ell}$ and $h'_n$ for the resulting map $\MM_0(n)_{cub} \to \MM_{cub}$.

\subsection{The comodule corresponding to $f_*f^*\OO$}
We will use $\OO$ as a shorthand notation for the structure sheaf on $\mell$ or $\mcub$. Recall that the map~$f'\colon \MM_1(2)_{cub}\to \MM_{cub}$ is affine by construction. This implies that the pushforward sheaf $(f')_*(f')^*\OO$ is quasi-coherent. By the discussion above, it is equivalent to a certain $(\widetilde{A},\widetilde{\Gamma})$-comodule. 

At the prime $3$, the universal elliptic curve with a $\Gamma_1(2)$-structure has an equation of the form 
\[y^2 = x^3+b_2x^2+b_4x \]
with $(0,0)$ being the chosen point of order $2$, resulting in an identification $\MM_1(2) \simeq \Spec \Z_{(3)}[b_2,b_4,\Delta^{-1}]/\GG_m$ (see e.g.\ \cite[Section 1.3]{BehrensModular}). As in \cite[Example 2.1]{MeierDecomposition} one can deduce $\MF_1(2) \cong \Z_{(3)}[b_2,b_4]$. The resulting $\widetilde{A}$-module structure is given by
\[
 \overline{a}_2\mapsto b_2 \mbox{ and } \overline{a}_4 \mapsto b_4 \mbox{ and } \overline{a}_6\mapsto 0.
\]
The corresponding $(\widetilde{A},\widetilde{\Gamma})$-comodule is given by $\widetilde{\Gamma}\otimes_{\widetilde{A}} \MF_1(2)$ with extended comodule structure by \cref{lem:associated}. In this tensor product, we use the right $\widetilde{A}$-module structure of $\widetilde{\Gamma}$. A similar computation to the following appears in \cite{Bauertmf}. 

\begin{Lemma}\label{mf12comodule}
 There is a ring isomorphism $\widetilde{\Gamma} \otimes_{\widetilde{A}} \MF_1(2)\cong \widetilde{A}[r]/(\overline{a}_6+\overline{a}_4r+\overline{a}_2r^2+r^3)$, and the comodule structure is an $\widetilde{A}$-module map determined by $r\mapsto 1\otimes r+r\otimes 1$ (and $\overline{a}_i \mapsto \overline{a}_i\otimes 1$).
 
 Forgetting the ring structure, we can identify this comodule with the free $\widetilde{A}$-module $\widetilde{A}w_1\oplus \widetilde{A}w_2 \oplus \widetilde{A}w_3$ with $(\widetilde{A},\widetilde{\Gamma})$-comodule structure given by
 \[
  \begin{aligned}
   w_1 &\mapsto &&1\otimes w_1,\\
   w_2 &\mapsto && 1\otimes w_2 +r\otimes w_1,\\
   w_3 &\mapsto && 1\otimes w_3+2r\otimes w_2+r^2\otimes w_1. 
  \end{aligned}
 \]

\end{Lemma}

\begin{pf}
 Using the formulae for $\eta_R$, we obtain a ring isomorphism 
 \[
  \widetilde{\Gamma} \otimes_{\widetilde{A}} \MF_1(2)\cong \widetilde{A}[r,b_2,b_4]/(R), 
 \]
where the relations $R$ are generated by 
\[
\begin{aligned}
&\overline{a}_2+3r=b_2,\\
&\overline{a}_4+2\overline{a}_2r+3r^2=b_4,\\
&\overline{a}_6+\overline{a}_4r+\overline{a}_2r^2+r^3=0.
\end{aligned}
\]
This immediately implies the first claim. 

The first statement about the comodule structure is straightforward since~$\widetilde{\Gamma} \otimes_{\widetilde{A}} \MF_1(2)$ carries the extended comodule structure.

For the second description of the comodule structure, observe that there is an isomorphism of $\widetilde{A}$-modules
\[
\widetilde{A}w_1\oplus \widetilde{A}w_2 \oplus \widetilde{A}w_3 \to \widetilde{A}[r]/(\overline{a}_6+\overline{a}_4r+\overline{a}_2r^2+r^3)
\]
given by $w_i\mapsto r^{i-1}$. Thus, to identify the comodule structure, we only need to compute it on $1,r,r^2$ on the right-hand side and transfer it via this isomorphism, using the compatibility of comodule structure with the ring structure of $\widetilde{A}[r]/(\overline{a}_6+\overline{a}_4r+\overline{a}_2r^2+r^3)$. This yields the claim. 
\end{pf}

We will now identify the dual $\mathcal{H}om_{\OO}(f_*f^*\OO, \OO)$ of the vector bundle $f_*f^*\OO$ on $\MMb_{ell}$, which will be useful in the next section. Actually, we will identify rather $\mathcal{H}om_{\OO}(f'_*(f')^*\OO, \OO)$ instead, working on $\MM_{cub}$. As $\MMb_{ell}$ sits as an open substack in $\MM_{cub}$, this directly implies the computation of $\mathcal{H}om_{\OO}(f_*f^*\OO, \OO)$.

For the identification, we use the translation into comodules. While this translation is valid for an arbitrary graded Hopf algebroid, we formulate it in terms of $(\widetilde{A}, \widetilde{\Gamma})$. Given a graded left $(\widetilde{A},\widetilde{\Gamma})$-comodule $M$ that is finitely generated free as a $\widetilde{A}$-module, we can define a right comodule structure on $\Hom_{\widetilde{A}}(M,\widetilde{A})$, whose coaction is given as in \cite[Definition A1.1.6]{RavenelAlt} by the composite of
\[\Hom_{\widetilde{A}}(M,\widetilde{A}) \xrightarrow{\widetilde{\Gamma} \tensor (-)} \Hom_{\widetilde{A}}(\widetilde{\Gamma}\tensor_{\widetilde{A}} M, \widetilde{\Gamma}) \xrightarrow{\Psi_M^*}\Hom_{\widetilde{A}}(M, \widetilde{\Gamma}) \]
with the inverse of the isomorphism 
\[\Hom_{\widetilde{A}}(M, \widetilde{A}) \tensor_{\widetilde{A}} \widetilde{\Gamma} \xrightarrow{\cong}\Hom_{\widetilde{A}}(M, \widetilde{\Gamma}).\]
Using the conjugation $c$ on $\widetilde{\Gamma}$ we can transform this into a left comodule. 

\begin{lemma}
Let $(M, \psi_M)$ be the graded left $(\widetilde{A},\widetilde{\Gamma})$-comodule corresponding to a vector bundle $\FF$ on $\MM_{cub}$ under the equivalence from \cref{AtildevsQCoh}. Then the graded left $(\widetilde{A},\widetilde{\Gamma})$-comodule $\Hom_{\widetilde{A}}(M,\widetilde{A})$ corresponds under the same equivalence to the sheaf $\mathcal{H}om_{\OO_{\MM_{cub}}}(\FF,\OO_{\MM_{cub}})$. 
\end{lemma}
\begin{proof}
Recall that $M = \FF(\Spec \At)$ and the maps $\eta_L, \eta_R\colon \widetilde{A} \to \widetilde{\Gamma}$ define isomorphisms 
\[\Gt\tensor_{\At} M \xrightarrow{\cong} \FF(\Spec \Gt) \xleftarrow{\cong} M \tensor_{\At} \Gt.\]
The composite of the isomorphisms $T$ (from right to left) with the natural morphism $i\colon M\to M\tensor_{\At} \Gt$ is by construction $\psi_M$.

One can show that the sheaf $\mathcal{H}om_{\OO_{\MM_{cub}}}(\FF,\OO_{\MM_{cub}})$ corresponds to the comodule $\Hom_{\At}(M, \At)$ with the structure map $\psi_{\Hom}$ that makes the following diagram commute:
\[
\begin{tikzcd}
\Hom_{\At}(M, \At) \arrow[d, dashed, "\psi_{\Hom}"] \arrow[r]& \Hom_{\Gt}(M\tensor_{\At}\Gt, \Gt) \arrow[d, "(T^{-1})^*"]\\
\Gt\tensor_{\At}\Hom_{\At}(M, \Gt) \arrow[r, "\cong"]&\Hom_{\Gt}(\Gt\tensor_{\At} M, \Gt)
\end{tikzcd}
\]

There is a further isomorphism $\Gt \tensor_{\At} M \cong M \tensor_{\At} \Gt$, interchanging the two factors and applying the conjugation to $\Gt$. As the conjugation on $\Gt$ interchanges the roles of $\eta_L$ and $\eta_R$, we can identify the composite
\[\Gt \tensor_{\At} M  \cong M \tensor_{\At} \Gt \xrightarrow{T} \Gt \tensor_{\At} M \cong M \tensor_{\At} \Gt\]
with $T^{-1}$. 

Using this ingredient, a lengthy diagram chase shows that $\psi_{\Hom}$ agrees with the comodule structure map of $\Hom_{\widetilde{A}}(M,\widetilde{A})$ described above. 
\end{proof}

\begin{lemma}\label{lem:dual}
The dual of $f'_*(f')^*\OO$ is isomorphic to $f'_*(f')^*\omline^{\tensor (-4)}$. 
\end{lemma}
\begin{proof}
Observe that the conjugation on $\widetilde{\Gamma}$ is given by $\aq_i \mapsto \eta_R(\aq_i)$ and $r \mapsto -r$ because it corresponds on the level of represented functors to inverting an isomorphism between Weierstra\ss{} curves. Inserting this into the above description of the internal hom and using \cref{mf12comodule}, we arrive at the following left comodule structure on $\widetilde{A}w_1^*\oplus \widetilde{A}w_2^* \oplus \widetilde{A}w_3^*$: 
\begin{alignat*}{3}
\widetilde{A}w_1^*\oplus \widetilde{A}w_2^* \oplus \widetilde{A}w_3^* &\longrightarrow &\widetilde{\Gamma}\otimes_{\widetilde{A}} \left(\widetilde{A}w_1^*\oplus \widetilde{A}w_2^* \oplus \widetilde{A}w_3^*\right)&\\
w_1^* &\mapsto & 1\otimes w_1^* - r\otimes w_2^* +r^2 \otimes w_3^*&\\
w_2^* &\mapsto & \hphantom{1\otimes w_1^* +}1\otimes w_2^* - 2r\otimes w_3^*&\\
w_3^*&\mapsto & \hphantom{1\otimes w_1^* +1\otimes w_2^* +}1 \otimes w_3^* &.
\end{alignat*}
Looking more closely shows that this comodule is actually isomorphic to $\widetilde{A}w_1\oplus \widetilde{A}w_2 \oplus \widetilde{A}w_3$ as an ungraded comodule from \cref{mf12comodule} via the following isomorphism:
\[
w_1 \mapsto w_3^*, \quad w_2 \mapsto -\frac{1}{2}w_2^*, \quad w_3 \mapsto w_1^*. 
\]
This map shifts grading by $4$. As $\omline^{\tensor (-4)}$ is the line bundle corresponding to the shift $A[-4]$ under the equivalence between quasi-coherent sheaves on $\MM_{cub}$ and graded $(A,\Gamma)$-comodules, this implies by the monoidality of the equivalence in \cref{associatedstackQCoh} an isomorphism $\mathcal{H}om_{\OO}(f_*'(f')^*\OO, \OO) \cong f_*'(f')^*\OO \tensor \omline^{\tensor (-4)}$. By the projection formula this yields the result. 
\end{proof}

\subsection{The comodules corresponding to $(f_7)_*(f_7)^*\OO$ and $(h_7)_*(h_7)^*\OO$}
Recall from \Cref{prop:IdentifyMF} the $3$-local isomorphism $\MF_1(7)\cong \Z_{(3)}[z_1,z_2,z_3]/(\sigma_2)$. Again by \cref{lem:associated} we obtain that the quasi-coherent sheaf $(f'_7)_*(f'_7)^*\OO$ on $\MM_{cub}$ corresponds to the extended $(\widetilde{A},\widetilde{\Gamma})$-comodule structure on $R_{\widetilde{A}} \cong \widetilde{\Gamma} \otimes_{\widetilde{A}} \MF_1(7)$. Recall further from the beginning of \cref{sec:Splitting} the induced $\widetilde{A}$-module structure on $\mf[r]$ from the identification $R_{\widetilde{A}} \cong \mf[r]$.
\begin{Lemma}\label{RBcomodule}
Consider the natural morphism $f_7\colon \MM_1(7)_{cub} \to \MM_{cub}$. 
 Under the ring isomorphism $R_{\widetilde{A}}\cong \Z_{(3)}[z_1,z_2,z_3,r]/(\sigma_2)$ the $(\widetilde{A},\widetilde{\Gamma})$-comodule structure corresponding to $(f_7)_*(f_7)^*\OO$ is completely determined by
 \[
    \begin{aligned}
   z_i &\mapsto &&1\otimes z_i, \mbox{ for }i\in\{1,2,3\},\\
   r &\mapsto && 1\otimes r +r\otimes 1.  
  \end{aligned}
 \]
\end{Lemma}

Recall that we identified $S_{\widetilde{A}}\cong (R_{\widetilde{A}})^{(\Z/7)^{\times}}$ (cf.\ \cref{SCasFP}) in \cref{SBbasis} as ${\widetilde{A}}$-module  with a free $8$-dimensional ${\widetilde{A}}$-module with basis 
\[
1, \sigma_1^2, \sigma_1^4, n_4, \sigma_1^2n_4, n_6, \sigma_1^2n_6, \sigma_3^2. 
\]
As $(h_7')_*(h_7')^*\OO$ injects into $(f_7)_*(f_7)^*\OO$, the corresponding comodule structure on $S_{\widetilde{A}}$ is that of a sub comodule of $R_{\widetilde{A}}$. Concretely, we obtain the following:
\begin{Lemma}\label{SBcomodule}
 The graded $(\widetilde{A},\widetilde{\Gamma})$-comodule structure on $S_{\widetilde{A}}$ is given by
  \[
  \begin{aligned}
   1 &\mapsto&& 1\otimes 1, \\
   \sigma_1^2 &\mapsto&& 1\otimes \sigma_1^2,\\
   \sigma_1^4  &\mapsto&& 1\otimes \sigma_1^4,\\
   n_4    &\mapsto&& 1\otimes n_4+r\otimes\sigma_1^2, \\
   \sigma_1^2n_4&\mapsto&& 1\otimes \sigma_1^2n_4+r\otimes\sigma_1^4,\\
   n_6  &\mapsto&& 1\otimes n_6 + 2r\otimes n_4 +r^2 \otimes \sigma_1^2,\\
   \sigma_1^2n_6 &\mapsto&& 1\otimes \sigma_1^2n_6 + 2r\otimes \sigma_1^2n_4 +r^2 \otimes \sigma_1^4,\\
   \sigma_3^2 &\mapsto&& 1\otimes \sigma_3^2.
  \end{aligned}
 \]

\end{Lemma}
\begin{proof}
 This follows from \Cref{RBcomodule} and \Cref{SBbasis} by a straightforward computation.
\end{proof}

\subsection{The conclusion}
We continue to work $3$-locally.
\begin{prop}\label{prop:comoduleiso}
 There is an isomorphism of graded $(\widetilde{A},\widetilde{\Gamma})$-comodules 
 \[
  \widetilde{A}\oplus (\widetilde{\Gamma}\otimes_{\widetilde{A}} \MF_1(2))[2] \oplus (\widetilde{\Gamma}\otimes_{\widetilde{A}} \MF_1(2))[4]\oplus \widetilde{A}[6] \to S_{\widetilde{A}}, 
 \]
given by 
\[
 \begin{aligned}
  1_{\widetilde{A}} &\mapsto && 1, \\
  w_1[2] &\mapsto && \sigma_1^2,\\
  w_2[2] &\mapsto && n_4,\\
  w_3[2] &\mapsto && n_6,\\
  w_1[4] &\mapsto && \sigma_1^4,\\
  w_2[4] &\mapsto && \sigma_1^2n_4,\\
  w_3[4] &\mapsto && \sigma_1^2n_6,\\
  1_{\widetilde{A}}[6] &\mapsto && \sigma_3^2. 
 \end{aligned}
\]

\end{prop}

\begin{proof}
 This follows by inspection from \Cref{mf12comodule} and \Cref{SBcomodule}.
\end{proof}

This implies our main algebraic theorem by the equivalence of $(\widetilde{A},\widetilde{\Gamma})$-comodules and quasi-coherent sheaves on $\MM_{cub,\mathbb{Z}_{(3)}}$ from \cref{AtildevsQCoh}:
\begin{Theorem}\label{thm:AlgMain}
There is $3$-locally an isomorphism 
\[
\begin{aligned}
(h_7')_*\OO_{\MM_0(7)_{cub}} \cong& \OO_{\MM_{cub}} \oplus \omline^{\tensor (-6)} \oplus \left((f')_*\OO_{\MM_1(2)_{cub}} \tensor \omline^{\tensor (-2)}\right)\\
&\oplus \left((f')_*\OO_{\MM_1(2)_{cub}} \tensor\omline^{\tensor (-4)}\right)
\end{aligned}
\]
of vector bundles on $\MM_{cub}$.
\end{Theorem}
By restricting to the open substack $\MMb_{ell,(3)}$, this implies Theorem \ref{thm:MainTheorem}. 

\section{Topological conclusions}\label{sec:TopConclusions}
\addtocontents{toc}{\protect\setcounter{tocdepth}{2}}
Recall from \cite{TMFbook} that one obtains the spectrum $\Tmf$ as the global sections of a sheaf of $E_\infty$-ring spectra $\OO^{top}$ on the \'etale site of $\MMb_{ell}$. Given any sheaf of spectra $\FF$ on the \'etale site of any Deligne--Mumford stack $\XX$, there is a \emph{descent spectral sequence}
\[H^q(\XX; \pi_p\FF) \Rightarrow \pi_{p-q}(\FF(\XX)),\]
where $\pi_*\FF$ denotes the \emph{sheafification} of the naive presheaf of homotopy groups \cite[Chapter 5]{TMFbook}. We have $\pi_{2p-1}\OO^{top}= 0$ and $\pi_{2p}\OO^{top} \cong \omline^{\tensor p}$ and in particular $\pi_0\OO^{top} \cong \OO_{\MMb_{ell}}$. Thus the descent spectral sequence takes the form
\[H^q(\MMb_{ell}; \omline^{\tensor p}) \Rightarrow \pi_{2p-q}\Tmf.\]
In general, the edge homomorphism takes the form $\pi_n(\FF(\XX)) \to (\pi_n\FF)(\XX)$. In the case of $\OO^{top}$, this produces a morphism $\pi_{2n}\Tmf \to \MF_n(\SL_2(\Z);\Z)$, which is not an isomorphism integrally even for $n\geq 0$.

Actually, the approach of \cite[Chapter 12]{TMFbook} defines sheaves of $E_\infty$-ring spectra $\OO^{top}_R$ on $\MMb_{ell,R}$ for every localization $R$ of the integers by varying the set of primes in the arithmetic square following Remark 1.6 in op.\ cit. By construction, $\pi_*\OO^{top}_R$ is again concentrated in even degrees with $\pi_{2k}\OO^{top}_R$ being the pullback of $\omline^{\tensor k}$ to $\MMb_{ell,R}$. As $R$ is a filtered colimit over the integers, we can form the analogous filtered homotopy colimit over $\Tmf$ to obtain a spectrum $\Tmf_R$ with $\pi_*\Tmf_R \cong (\pi_*\Tmf)\tensor R$. As homotopy colimits do not commute with global sections in general, we have to prove the following lemma about the global sections $\Gamma(\OO^{top}_R) = \OO^{top}_R(\MMb_{ell,R})$. 
\begin{lemma}
The map $\Tmf \to \Gamma(\OO^{top}_R)$ factors over an equivalence $\Tmf_R \to \Gamma(\OO^{top}_R)$. 
\end{lemma}
\begin{proof}
The map $(\MMb_{ell, R},\OO^{top}_R) \to (\MMb_{ell},\OO^{top})$ induces a map of descent spectral sequences in the opposite direction. As $R$ is flat over $\Z$ and cohomology commutes with flat base change, this map of spectral sequences is just tensoring with $R$. The map converges moreover to a map $\pi_*\Tmf \to \pi_*(\OO^{top}_R(\MMb_{ell,R}))$ and we claim that the induced map 
\[\pi_*\Tmf \tensor R \cong \pi_*\Tmf_R \to \pi_*\Gamma(\OO^{top}_R)\]
is an isomorphism. This is true because the $E_\infty$-pages of these descent spectral sequences are concentrated in finitely many lines, either by computation \cite{Konter} or conceptually as in \cite[Theorem 3.14]{MathewThick}. As $\Gamma(\OO^{top}_R)$ is an $R$-local spectrum, the map $\Tmf \to \Gamma(\OO^{top}_R)$ factors over a map $\Tmf_R \to \Gamma(\OO^{top}_R)$ that induces exactly the isomorphism above on $\pi_*$ and is thus an equivalence.
\end{proof}

To avoid cluttering the notation, we will set $\MMb_{ell} = \MMb_{ell,R}$ and $\Tmf = \Tmf_R$ etc. in the following.

We will work in the homotopy category of $\OO^{top}$-modules. We denote the derived smash product over $\OO^{top}$ by $\tensor_{\OO^{top}}$ and the internal Hom in this category by $\mathcal{H}om_{\OO^{top}}$ (see \cite[Section 2.2]{MeierTMFLevel} for details on the latter). Given two $\OO^{top}$-modules $\FF$ and $\cG$, we denote by  $[\FF,\cG]^{\OO^{top}}$ the morphism set in the homotopy category and this coincides with $\pi_0$ of the global sections $\mathrm{Hom}_{\OO^{top}}(\FF,\cG)$ of the sheaf of spectra $\mathcal{H}om_{\OO^{top}}(\FF,\cG)$. 

\begin{Definition}
 An $\OO^{top}$-module $\F$ is \emph{locally free of rank $n$} if there is an \'etale covering $\{U_i \to \MMb_{ell}\}$ such that $\F$ restricted to $U_i$ is equivalent to $\bigoplus_n \OO^{top}|_{U_i}$. 
\end{Definition}

\begin{lemma}\label{lem:locfree}
Let $\F$ and $\cG$ be $\OO^{top}$-module and assume $\FF$ to be locally free.
\begin{enumerate}
 \item The homotopy groups $\pi_p\FF$ are zero for $p$ odd and isomorphic to $\pi_0\FF \tensor \omline^{\tensor \frac{p}2}$ for $p$ even. 
 \item The map $\pi_p\mathcal{H}om_{\OO^{top}}(\FF,\cG) \to \mathcal{H}om_{\OO_{\MMb_{ell}}}(\pi_0\FF, \pi_p\cG)$ is an isomorphism for every $p\in\Z$.
\end{enumerate}
\end{lemma}
\begin{proof}
The sheaf $\pi_p\FF$ vanishes for $p$ odd as it vanishes locally. For $p$ even, we can write $\pi_p\FF \cong \pi_0 \Sigma^{-p}\FF \cong \pi_0\left(\Sigma^{-p}\OO^{top} \wedge_{\OO^{top}}\FF\right)$. The map 
\[\omline^{\tensor \frac{-p}2}\tensor \pi_0\FF \cong \pi_0\Sigma^{-p}\OO^{top} \tensor_{\pi_0\OO^{top}} \pi_0\FF \to  \pi_0\left(\Sigma^{-p}\OO^{top} \wedge_{\OO^{top}}\FF\right) \cong \pi_0 \Sigma^{-p}\FF\]
is an isomorphism as it is an isomorphism locally when $\FF \simeq (\OO^{top})^n$.

For the second part, we argue similarly that the map 
\[\pi_p\mathcal{H}om_{\OO^{top}}(\FF,\cG) \to \mathcal{H}om_{\OO_{\MMb_{ell}}}(\pi_0\FF, \pi_p\cG)\]
is an isomorphism as it is one locally when $\FF \simeq (\OO^{top})^n$.
\end{proof}
A related lemma to the following already appears in \cite[Lemma 2.2.2]{BehrensOrmsby}.

\begin{lemma}\label{lem:trace}
Let $\AA$ be a sheaf of $\OO^{top}$-algebras on $\MMb_{ell}$ that is locally free of rank $n$ as an $\OO^{top}$-module. There is a trace map
$$\tr_{\AA}\colon \AA \to \OO^{top}$$
such that the composite $\tr_{\AA}u$ with the unit map 
$$u\colon \OO^{top} \to \AA$$
equals multiplication by $n$. 
\end{lemma}
\begin{proof}
Consider the composite
$$\tr_{\AA}\colon \AA \to \mathcal{H}om_{\OO^{top}}(\AA,\AA) \xleftarrow{\simeq} \AA \tensor_{\OO^{top}} \mathcal{H}om_{\OO^{top}}(\AA,\OO^{top}) \xrightarrow{\ev} \OO^{top}.$$
Here, the middle map is an equivalence because $\AA$ is locally free of finite rank. We claim that the composite 
$$\tr_{\AA}u\colon \OO^{top} \to \OO^{top}$$ 
equals multiplication by $n$. 

Note first that the map
\[\pi_0\colon [\OO^{top},\OO^{top}]^{\OO^{top}} \to \Hom_{\pi_0\OO^{top}}(\pi_0\OO^{top}, \pi_0\OO^{top})
\]
is a bijection. Indeed, the source agrees with $\pi_0\Gamma(\OO^{top}) = \pi_0\Tmf$ and the morphism is the edge homomorphism of the descent spectral sequence for 
\[\mathcal{H}om_{\OO^{top}}(\OO^{top},\OO^{top}) \simeq \OO^{top}.\] 
It can be deduced from \cite[Section 3, Figure 11, Figure 26]{Konter} 
that this edge homomorphism is an isomorphism, i.e.\ that the $E_\infty$-term contains in the zeroth column only a $\Z$ in line $0$ and nothing above it (see also the proof of \cite[Lemma 4.9]{HillMeier} for a different approach).  Thus, it is enough to show that $\tr_{\AA}u$ is multiplication by $n$ on $\pi_0$.

As $\AA$ is locally free, 
\[\pi_0\AA \tensor_{\OO_{\MMb_{ell}}}\pi_0\mathcal{H}om_{\OO^{top}}(\AA,\OO^{top}) \to \pi_0 (\AA \tensor_{\OO^{top}} \mathcal{H}om_{\OO^{top}}(\AA,\OO^{top})) \]
is locally and hence globally an isomorphism. Note that the source is naturally isomorphic to $\pi_0\AA \tensor_{\OO_{\MMb_{ell}}}\mathcal{H}om_{\OO_{\MMb_{ell}}}(\pi_0\AA,\OO_{\MMb_{ell}})$ by the discussion by \cref{lem:locfree}.
Using these isomorphisms it can be checked that $\pi_0\tr_{\AA}\colon \pi_0\AA \to \pi_0\OO^{top}$ agrees with the trace map of $\pi_0\AA$ over $\pi_0\OO^{top} = \OO_{\MMb_{ell}}$. Its precomposition with $\pi_0u$ equals $n$ as it does locally (since we get exactly the trace of the identity map of a free module of rank $n$). This shows the claim. 
\end{proof}

Now we assume that $\frac12\in R$. We will need the following variant of \cite[Lemma 5.2.2]{MeierDoktorarbeit}. Recall that we denote by $f$ the natural map $\MMb_1(2) \to \MMb_{ell}$. By \cite{HillLawson}, we have a sheaf of $E_\infty$-ring spectra $\OO^{top}_{\MMb_1(n)}$ on the \'etale site of every $\MMb_1(n)$. 
We denote by $f_*f^*\OO^{top}$ the sheaf $f_*\OO^{top}_{\MMb_1(n)}$ on $\MMb_{ell}$, i.e.\ the one associating with every \'etale map $U \to \MMb_{ell}$ the $E_\infty$-ring spectrum $\OO^{top}_{\MMb_1(2)}(U\times_{\MMb_{ell}}\MMb_1(2))$. By the proof of \cite[Theorem 3.5]{MeierTMFLevel} the odd homotopy of $f_*f^*\OO^{top}$ vanishes and $\pi_{2i}f_*f^*\OO^{top} \cong f_*f^*\omline^{\tensor i}$. This implies in paticular that $f_*f^*\OO^{top}$ is locally free.

\begin{lemma}\label{lem:canrealize}
 Let $\F$ be a locally free $\OO^{top}$-module on $\MMb_{ell}$ of finite rank. Let $g_{alg}\colon f_*f^*\omline^{\tensor (-i)} \to \pi_0\F$ be a split injection. Then $g_{alg}$ can be uniquely realized by a split map
 \[g\colon \Sigma^{2i}f_*f^*\OO^{top} \to \F\]
  with $\pi_0 g = g_{alg}$. 
\end{lemma}
\begin{proof}
Using \cref{lem:locfree} and the projection formula, we reduce to the case $i=0$ by possibly suspending $\FF$ to simplify notation. 

 We will use our earlier identification in \cref{lem:dual} of the dual of the vector bundle $f_*f^*\OO_{\MMb_{ell}} \cong f_*\OO_{\MMb_1(2)}$ with $\omline^{\tensor (-4)} \tensor_{\OO_{\MMb_{ell}}} f_*\OO_{\MMb_1(2)} \cong f_*f^*\omline^{\tensor (-4)}$. 
As $\pi_p\FF$ is locally free and using \cref{lem:locfree}, this implies
\begin{align*}
\pi_{p}\mathcal{H}om_{\OO^{top}}(f_*f^*\OO^{top},\F) &\cong \mathcal{H}om_{\OO_{\MMb_{ell}}}(f_*f^*\OO_{\MMb_{ell}}, \pi_p\FF)\\
&\cong f_*f^*\omline^{\tensor (-4)}\tensor_{\MMb_{ell}}\pi_p\FF \\
&\cong f_*f^*(\omline^{\tensor (-4)}\tensor_{\MMb_{ell}}\pi_p\FF).
\end{align*}
As $f$ is affine and every quasi-coherent sheaf on $\MMb_1(2)$ has cohomology at most in degrees $0$ and $1$ by \cite[Proposition 2.4(4)]{MeierDecomposition}, the descent spectral sequence 
$$H^q(\MMb_{ell}; \pi_p\mathcal{H}om_{\OO^{top}}(f_*f^*\OO^{top},\F)) \Rightarrow \pi_{p-q}\Hom_{\OO^{top}}(f_*f^*\OO^{top},\F)$$
is concentrated in the lines $0$ and $1$. Moreover, the $E_2$-term is zero for $p$ odd and thus the edge homomorphism 
$$[f_*f^*\OO^{top},\FF]^{\OO^{top}} = \pi_0\Hom_{\OO^{top}}(f_*f^*\OO^{top},\F) \to \Hom_{\OO_{\MMb_{ell}}}(f_*f^*\OO_{\MMb_{ell}},\pi_0\F)$$
is an isomorphism. 

Similarly, one shows that 
$$[\F,f_*f^*\OO^{top}]^{\OO^{top}} = \pi_0\Hom_{\OO^{top}}(\F, f_*f^*\OO^{top}) \to \Hom_{\OO_{\MMb_{ell}}}(\pi_0\F,f_*f^*\OO_{\MMb_{ell}})$$
is an isomorphism. The lemma follows.
\end{proof}

Recall that we denote the natural map $\MMb_0(7) \to \MMb_{ell}$ by $h$. By the work of \cite{HillLawson}, we have a sheaf of $E_\infty$-ring spectra $\OO^{top}_{\MMb_0(7)}$ on the \'etale site of $\MMb_0(7)$ and we denote by $h_*h^*\OO^{top}$ its pushforward to $\MMb_{ell}$ along $h$. 

\begin{Theorem}\label{thm:topmain}
We can decompose $\Tmf_0(7)_{(3)}$ as a $\TMF_{(3)}$-module into
\[\Tmf_{(3)} \oplus \Sigma^4\Tmf_1(2)_{(3)} \oplus \Sigma^8\Tmf_1(2)_{(3)}\oplus L,\]
where $L \in \Pic(\Tmf_{(3)})$, i.e.\ $L$ is an invertible $\Tmf_{(3)}$-module. There is a corresponding splitting
\[ h_*h^*\OO^{top}_{(3)} \simeq \OO^{top}_{(3)} \oplus \Sigma^4f_*f^*\OO^{top}_{(3)}\oplus \Sigma^8f_*f^*\OO^{top}_{(3)} \oplus \L\]
for a certain invertible $\OO^{top}_{(3)}$-module $\L$.
\end{Theorem}
\begin{proof}Throughout this proof, we will implicitly localize at $3$. 
By Lemma \ref{lem:trace}, the unit map $\OO^{top} \to h_*h^*\OO^{top}$ splits off as an $\OO^{top}$-module; denote the cofiber by $\FF$. Note that $\pi_k\FF = 0$ for $k$ odd. By Theorem \ref{thm:AlgMain}, 
$$\pi_0\FF \cong \omline^{\tensor (-6)}\oplus f_*f^*\omline^{\tensor (-2)} \oplus f_*f^*\omline^{\tensor (-4)}.$$
By Lemma \ref{lem:canrealize}, we obtain a decomposition $$\F \cong \L \oplus \Sigma^4f_*f^*\OO^{top} \oplus \Sigma^8f_*f^*\OO^{top}$$
with $\pi_0\L \cong \omline^{\tensor (-6)}$. As a summand of a locally free module, $\L$ is locally free as well and thus an invertible $\OO^{top}$-module as it has rank $1$. We obtain our result by taking global sections because the global sections of $h_*h^*\OO^{top}$ are $\Tmf_0(7)$. To see that $L = \Gamma(\L)$ is an invertible $\Tmf$-module, we use that the global sections functor
$$\Gamma\colon \mathrm{QCoh}(\MMb_{ell},\OO^{top}) \to \Tmf\mathrm{-mod}$$
is a symmetric monoidal equivalence of $\infty$-categories by one of the main results of \cite{MathewMeier}. 
\end{proof}

In \cref{AppB} we will identify this invertible $\TMF_{(3)}$-module $L$ precisely. In particular, we will show that $L_{\Q} \simeq \Sigma^{12}\Tmf_{\Q}$ and $L\wedge_{\Tmf}\TMF \simeq \Sigma^{36}\TMF$. This implies directly that $L$ is not just a suspension of $\Tmf_{(3)}$. Moreover, we obtain together with the $8$-fold periodicity of $\TMF_1(2)$ the following corollary. 

\begin{Cor}
We can decompose $\TMF_0(7)_{(3)}$ as a $\TMF_{(3)}$-module into
\[ \TMF_{(3)} \oplus \Sigma^4\TMF_1(2)_{(3)} \oplus \TMF_1(2)_{(3)} \oplus \Sigma^{36}\TMF_{(3)}.\]
\end{Cor}


\appendix

\section{Modular forms and $q$-expansions}\label{qExpAppendix}
The aim of this appendix is to review several different definitions of modular forms (complex-analytic, in the sense of Katz and via stacks) and compare them via explicit isomorphisms. Moreover, we will repeat this for modular forms with respect to the congruence subgroup $\Gamma_1(n)$ and the corresponding algebraic definition via the moduli stack of elliptic curves with level structure. We have no claim of originality here. The main reason for writing this appendix anyhow is the existence of two different versions of level structures, often called \emph{naive} and \emph{arithmetic}, whose precise relationship has confused at least the authors in the past. In particular, we will deduce a $q$-expansion principle for the naive level structure, namely Theorem \ref{thm:q-exp}. 

We have based our treatment on \cite{DiamondIm}, \cite{DiamondShurman}, \cite{Katz} and \cite[Section 2]{KatzRealEisenstein}, of which we recommend especially the first two as an introduction to modular forms. We also refer to \cite{ConradRamanujan} for a thorough treatment of the geometry on the analytic side.

\subsection{Modular forms}\label{sec:ModForms}
In this section, we will give three definitions of modular forms and compare them. 

\subsubsection{Analytic definition of modular forms}\label{sec:analytic}
We start by recalling the classical definition of modular forms. Let $f$ be first any $1$-periodic holomorphic function $\bH \to \C$. Then there is a well-defined holomorphic function $g\colon \mathbb{D}\setminus\{0\} \to \C$  satisfying $f(z)=g(e^{2\pi i z})$, where $\mathbb{D}$ denotes the open unit disk. We say that $f$ is \emph{holomorphic/meromorphic at $\infty$} if and only if $g$ can be extended holomorphically/meromorphically to $0$. In these cases, we call the Laurent expansion of $g$ at $0$ the \emph{classical $q$-expansion of $f$} (at $\infty$).

Given a matrix $\gamma=\begin{pmatrix}a & b\\ c &d\end{pmatrix} \in \operatorname{GL}_2(\R)$ with positive determinant, an integer $k$ and an arbitrary function $f\colon \mathbb{H}\to \mathbb{C}$, one defines a new function $f[\gamma]_k$ as follows:
\begin{alignat*}{2}
f[\gamma]_k\colon&\mathbb{H}\to \mathbb{C}\\
 &z \mapsto (cz+d)^{-k}f\left(\frac{az+b}{cz+d}\right).
\end{alignat*}
By \cite[Lemma 1.2.2]{DiamondShurman}, we have $(f[\gamma]_k)[\gamma'_k] = f[\gamma\gamma']_k$ for $\gamma, \gamma'\in \SL_2(\Z)$ and the same proof works actually for arbitrary $\gamma, \gamma'\in \operatorname{GL}_2(\R)$ of positive determinant. We say that a holomorphic function $f\colon \bH \to \C$ for a fixed $k$ is \emph{holomorphic/meromorphic at all cusps} if $f[\gamma]_k$ is $1$-periodic and holomorphic/meromorphic at $\infty$ for all $\gamma \in \SL_2(\Z)$. 

Let $\Gamma_1(n) \subset \SL_2(\Z)$ be the subgroup of matrices that reduce to a matrix of the form $\begin{pmatrix}1&\ast\\0 & \ast\end{pmatrix}$ modulo $n$. Note that for $n=1$, we obtain $\Gamma_1(1) = \SL_2(\Z)$. We denote by $\MF_k(\Gamma_1(n); \mathbb{C})$  the set of holomorphic functions $f\colon\bH\to\C$ that satisfy
\begin{align}\label{eq:transformation}
 f\left(\frac{az+b}{cz+d}\right)=(cz+d)^kf(z) \quad \text{ for every } z\in \bH \text{ and } \begin{pmatrix} a & b\\ c& d\end{pmatrix}\in \Gamma_1(n)
\end{align}
and are holomorphic at all cusps. Note that it is automatic that $f[\gamma]_k$ is $1$-periodic for all $\gamma \in \SL_2(\Z)$ as $\gamma\begin{pmatrix}1&1\\0&1\end{pmatrix}\gamma^{-1} \in \Gamma_1(n)$. Elements of $\MF_k(\Gamma_1(n); \mathbb{C})$ are called \emph{holomorphic modular forms of weight $k$} for $\Gamma_1(n)$. If we instead require $f$ to be meromorphic at all cusps, we speak of \emph{meromorphic modular forms of weight $k$} and denote the set of these by $\MFC_k(\Gamma_1(n);\mathbb{C})$.

For a subring $R_0 \subset \C$, we denote by $\MFC_k(\Gamma_1(n); R_0)$ the subset of $\MFC_k(\Gamma_1(n); \mathbb{C})$ of modular forms with coefficients of classical $q$-expansion of $f$ lying in $R_0$ and we use the notation $\MF_k(\Gamma_1(n);R_0)$ analogously. 

We note that the multiplication of functions induces a multiplication on the direct sum $\MFC(\Gamma_1(n);R_0) = \bigoplus_{k\in \Z} \MFC_k(\Gamma_1(n); R_0)$, making it into a graded ring of modular forms. The $q$-expansion defines a ring homomorphism $\MFC(\Gamma_1(n); R_0) \to R_0((q))$.

An important example of a modular form is the modular discriminant $\Delta \in \MF_{12}(\SL_2(\Z);\Z)$ with $q$-expansion $q -24q^2 + \cdots$. From the $q$-expansion we see that for every meromorphic modular form $f\in \MFC_k(\SL_2(\Z); R_0)$, there is a $k>0$ such that $\Delta^kf$ is a holomorphic modular form. Moreover, $\Delta$ vanishes nowhere on the upper half-plane \cite[Corollary 1.4.2]{DiamondShurman} so that $\Delta^{-1}$ is a meromorphic modular form over $\Z$ again. We see that $\MF(\SL_2(\Z);R_0)[\Delta^{-1}] \to \MFC(\SL_2(\Z);R_0)$ is an isomorphism. As $\Delta[\gamma]_{12} = \Delta$ for all $\gamma \in \SL_2(\Z)$, we can repeat the argument above to see that $\MF(\Gamma_1(n);R_0)[\Delta^{-1}] \to \MFC(\Gamma_1(n);R_0)$ is an isomorphism for all $n$ and similarly for other congruence subgroups of $\SL_2(\Z)$.

\subsubsection{Algebro-geometric definitions of modular forms}
For the algebro-geometric definitions of modular forms, we will concentrate in this part on the situation without level, i.e.\ the one corresponding to modular forms for $\SL_2(\Z)$. We denote for a (generalized) elliptic curve $p\colon E \to T$ the quasi-coherent sheaf $p_*\Omega^1_{E/T}$ by $\omega_E$. For the definition of a generalized elliptic curve see \cite[Definition 1.12]{DeligneRapoport}.
\begin{Prop}[{\cite[Proposition II.1.6]{DeligneRapoport}}]\label{prop:omegaE}
 Let $p\colon E\to T$ be a generalized elliptic curve, and denote its chosen section by $e\colon T\to E$. Then the sheaf $\omega_E = p_*\Omega^1_{E/T}$ is a line bundle on $T$. Moreover, the adjunction counit
\[
 p^*p_*\Omega^1_{E/T}\to \Omega^1_{E/T}
\]
is an isomorphism and thus $p_*\Omega^1_{E/T} \cong e^*\Omega^1_{E/T}$.
\end{Prop}
An \emph{invariant differential} for $E$ is a nowhere vanishing section of $\Omega^1_{E/T}$ or equivalently a trivialization of $\omega_E$. 

Our second definition of modular forms will define them as a certain kind of natural transformations. Fix a commutative ring $R_0$. For any $R_0$-algebra $R$, denote by $\Ell^1(R)$ the set of isomorphism classes of pairs $(E,\omega)$ consisting of an elliptic curve $E$ over $R$ together with an invariant differential. This defines (together with pullback of elliptic curves and of invariant differentials) a functor 
\[
 \Ell^1(-)\colon (\operatorname{AffSch}/\Spec(R_0))^{\operatorname{op}} \to \operatorname{Sets}. 
\]
As in \cite[ Section 1.1]{Katz}, we can consider a notion of a modular form of level $1$ and weight $k$ over $R_0$ as the subset of the set of natural transformations $f\in \Nat^{R_0}(\Ell^1(-), \Gamma(-))$ with the following scaling property: For any $R_0$-algebra $R$, elliptic curve with chosen invariant differential $(E,\omega)$ and any $\lambda\in R^{\times}$, we have 
\begin{align}\label{eq:scaling}
 f(E,\lambda\omega)=\lambda^{-k}f(E,\omega). 
\end{align}
Denote the set of such natural transformations by $\Nat^{R_0}_{k}(\Ell^1(-), \Gamma(-))$. Also here, the direct sum $\bigoplus_{k \in \Z}\Nat^{R_0}_{k}(\Ell^1(-), \Gamma(-))$ carries a multiplication by multiplying values in the target. This multiplication gives again a definition of a graded ring of modular forms.\\

For the third definition, let $\MM_{ell,R_0}$ be the moduli stack of elliptic curves over $\Spec(R_0)$ (see e.g.\ \cite{DeligneRapoport} or \cite{M-OMell}). On its big \'{e}tale site, one defines a line bundle $\omline = \omline_{R_0}$ as follows. For a morphism $t\colon T\to \MM_{ell,R_0}$ from a scheme $T$, let $p\colon E\to T$ be the corresponding elliptic curve with unit section $e$. We associate with $(T,t)$ the line bundle $\omega_E$ on $T$. To check that this actually defines a line bundle consider a cartesian square
\[
\begin{tikzcd}
E'\arrow[d,"p'" swap] \arrow[r, "\widetilde{f}"]  & E \arrow[d, "p"]\\
T'\arrow[r, "f" swap] & T
\end{tikzcd}
\]
with unit section $e'\colon T' \to E'$. We obtain a chain of natural isomorphisms
\begin{equation}\label{eq:canonicaliso}
    f^*\omega_E \cong f^*e^*\Omega^1_{E/T} \cong (e')^*\tilde{f}^*\Omega^1_{E/T} \cong (e')^*\Omega^1_{E'/T'} \cong \omega_{E'}
\end{equation}
as required. 

The third definition of the meromorphic modular forms over $R_0$ of weight $k$ is $H^0(\MM_{ell,R_0};\omline^{\tensor k}_{R_0})$. Here, the direct sum $\bigoplus_{k \in\Z} H^0(\MM_{ell,R_0};\omline^{\tensor k}_{R_0})$ carries a multiplication inherited from the tensor algebra $\bigoplus_{k\in\Z} \omline^{\tensor k}_{R_0}$, defining also here a graded ring of modular forms. Sometimes it is convenient to reinterpret this ring as $H^0(\MM_{ell,R_0}^1, \OO_{\MM_{ell,R}^1})$, where $\MM_{ell,R_0}^1$ is the relative spectrum of $\bigoplus_{i\in\Z}\omline_{R_0}^{\tensor i}$ \cite[Section 12.1]{GoertzWedhorn}.

\subsubsection{Comparision of definitions of modular forms}\label{sec:compare}
We start by comparing the two algebro-geometric definitions.
\begin{Prop}\label{sectionsvsnat}
There is a natural isomorphism 
\[
\alpha\colon H^0(\MM_{ell,R_0}, \omline_{R_0}^{\otimes k}) \to \Nat^{R_0}_{k}(\Ell^1(-), \Gamma(-)).
\]
Moreover, on the direct sum for all $k\in \Z$, the map $\alpha$ induces an isomorphism of graded rings.
\end{Prop}

\begin{proof}
There is an easy map 
\[
\alpha\colon H^0(\MM_{ell,R_0}, \omline_{R_0}^{\otimes k}) \to \Nat^{R_0}_{k}(\Ell^1(-), \Gamma(-)),
\]
constructed as follows. Start with an element $f\in H^0(\MM_{ell,R_0}, \omline_{R_0}^{\otimes k})$, an $R_0$-algebra $R$ and an elliptic curve $E/R$ together with an invariant differential $\omega$. If $E$ is classified by $\varphi\colon\Spec(R)\to \MM_{ell,R_0}$, we have $\varphi^*(\omline_{R_0}^{\otimes k}) = \omega_E^{\otimes k}$. By pulling back, $f$ defines an element in $\Gamma(\varphi^*(\omline_{R_0}^{\otimes k}))$, which via the isomorphism $\omega^{\otimes k}$ from $\OO_R^{\otimes k}$ to $\omega_E^{\otimes k}$ is identified with 
\[
 \Gamma(\varphi^*(\omline_{R_0}^{\otimes k})) = \Gamma(\omega_E^{\otimes k})\cong \Gamma(\OO_R^{\otimes k})\cong \Gamma(\OO_R)=R. 
\]
Define $\alpha(f)(E,\omega)$ to be the image in $R$ of the element defined by $f$ in the left-hand side. The naturality of $\alpha(f)$ is clear. 
Replacing $\omega$ by $\lambda\omega$ for $\lambda\in R^{\times}$ multiplies the chosen isomorphism above by $\lambda^k$, so we obtain
\[
 \alpha(f)(E, \lambda\omega)=\lambda^{-k}\alpha(f)(E,\omega).
\]

Let us sketch why $\alpha$ is an isomorphism. By definition, the section $f$ corresponds to a compatible choice of sections in $H^0(T;\omega_E^{\tensor k})$ for all $T \to \MM_{ell,R_0}$ classifying an elliptic curve $E/T$. 
As $\omega_E$ is locally trivial, $f$ is uniquely determined by its values on those $T$ where $\omega_E$ is already trivial and $T = \Spec R$ is affine and every coherent choice of values on such $T$ induces a section of $\omline^{\tensor k}_{R_0}$. For such $T$, a section of $\omega_E^{\tensor k}$ corresponds exactly to associating with each trivialization $\omega$ of $\omega_E$ an element $f(E, \omega)$ such that $f(E,\lambda\omega) = \lambda^{-k}f(E,\omega)$. This describes $\Nat^{R_0}_{k}(\Ell^1(-), \Gamma(-))$. 
\end{proof}

Next, we exhibit the map which will turn out to be an isomorphism between the algebraic geometric definitions and the complex analytic ones. 
\begin{Prop}
For any subring $R_0$ of $\C$ define
\[
\beta\colon  \Nat^{R_0}_{k}(\Ell^1(-), \Gamma(-))\to \MFC_k(\SL_2(\Z), R_0), 
\] as follows. For any $f\in  \Nat^{R_0}_{k}(\Ell^1(-), \Gamma(-))$ and any $\tau\in \bH$, set 
\[
\beta(f)(\tau)=f(\C/\Z \oplus \Z\tau, dz) \in \C.
\] 
Then $\beta$ is a natural isomorphism, and induces an isomorphism of graded rings on the direct sum for all $k\in\Z$. 
\end{Prop}
We will check \eqref{eq:transformation} for $\beta(f)$. Let $\begin{pmatrix} a & b\\ c& d\end{pmatrix}\in \SL_2(\Z)$ be given. Observe that we have a biholomorphism
\[
 \begin{aligned}
  \psi \colon \C/\left(\Z\cdot 1\oplus \Z\tau\right) &\to&& \C/\left(\Z\cdot 1\oplus \Z\frac{a\tau+b}{c\tau+d}\right),\\
 {[z]} &\mapsto && \left[\frac{z}{c\tau+d}\right]
 \end{aligned}
\]
and by GAGA thus an isomorphism of the associated algebraic curves. 
Since $f$ is well-defined on isomorphism classes, the scaling property implies
\[
\begin{aligned}
 \beta(f)\left(\frac{a\tau+b}{c\tau+d}\right)&= f\left(\C/\left(\Z\cdot 1\oplus \Z\frac{a\tau+b}{c\tau+d}\right), dz\right)\\
&=(c\tau+d)^{k} f(\C/\Z\cdot 1\oplus \Z\tau, dz)=(c\tau+d)^{k} \beta(f)(\tau).
\end{aligned}
\]
We will come back to the question why $\beta(f)$ is a holomorphic in the interior and meromorphic at the cusps and why $\beta$ is an isomorphism in \cref{sec:Comparisons} and \cref{sec:qexpansions}. 

\subsection{Level structures}
Throughout this section, let $R_0$ be a $\Z[\frac1n]$-algebra. While we gave the analytic definition of modular forms for $\Gamma_1(n)$ already above, there are two different corresponding algebro-geometric notions, based on \emph{naive} and \emph{arithmetic} level structures. 

\subsubsection{Naive level structures}\label{sec:naive}
\begin{Definition}[{\cite[Construction 4.8]{DeligneRapoport}}]
 For an $R_0$-algebra $R$, let $\Ell^1_{\Gamma_1(n)}(R)$ denote the set of isomorphism classes of triples $(E, \omega, j)$, where $E$ is an elliptic curve over $R$, further $\omega$ is a chosen trivialization of the line bundle $\omega_E$, and $j\colon \mathbb{Z}/n\mathbb{Z}_{R} \to E$ is a morphism of group schemes over $\Spec(R)$ and a closed immersion. This morphism $j$ is called a \emph{$\Gamma_1(n)$-level structure}. 
\end{Definition} 

Recall that $\mathbb{Z}/n\mathbb{Z}_{R}=\coprod_{\mathbb{Z}/n\mathbb{Z}}\Spec(R)$ as a scheme, with the obvious map to $\Spec R$ and group structure coming from the group structure on $\mathbb{Z}/n\mathbb{Z}$. The group structure on the elliptic curve is explained in \cite[Section 2.1]{KatzMazur}. We can identify $j$ with the image $P = j(1)\in E(R)$ since it determines $j$ completely.

\begin{remark}
We should remark that this variant of level structures is often called ``naive'' in the literature. Note also that the analogous definition in \cite[Section 8.2]{DiamondIm}, looks slightly different, but is equivalent by using that being closed immersion can be checked for proper schemes on geometric points. 
\end{remark}

Using again the scaling condition \eqref{eq:scaling} we can define $\Nat^{R_0}_k(\Ell^1_{\Gamma_1(n)}(-), \Gamma(-))$ analogously to our definition without level in Section \ref{sec:ModForms}.

We can also define a moduli stack $\MM_1(n)$ classifying elliptic curves over $\Z[\frac1n]$-schemes with $\Gamma_1(n)$-level structure. We obtain a morphism $f_n\colon \MM_1(n) \to \MM_{ell}$ by forgetting the level structure. As in Section \ref{sec:compare} we obtain a comparison isomorphism
$$\alpha\colon H^0(\MM_1(n); \omline^{\tensor k}) \to \Nat^{R_0}_k(\Ell^1_{\Gamma_1(n)}(-), \Gamma(-));$$
here and in the following we will abuse notation to denote the pullback of $\omline$ to $\MM_1(n)$ by $\omline$ as well. 
 
There are different ways to compare modular forms with and without level structure. The particular form of compatibility we want to use is expressed in the following commutative diagram.

\begin{equation*}
 \begin{tikzcd}[column sep=2cm]
\Nat^{R_0}_k(\Ell^1(-), \Gamma(-))\arrow[d, "{(\mathbb{C}/\mathbb{Z}+\tau\mathbb{Z}, dz)}" ] \arrow[r, "{(E, P) \mapsto E/\langle P \rangle}" ] & \Nat^{R_0}_k(\Ell^1_{\Gamma_1(n)}(-), \Gamma(-)) \arrow[d, "{(\mathbb{C}/\mathbb{Z}+n\tau\mathbb{Z}, dz, \tau)}" ]\\
 \MFC_k(\SL_2(\Z), R_0) \arrow[r, hook]& \MFC(\Gamma_1(n), R_0) 
 \end{tikzcd}
\end{equation*}
We refer to \cite[Example 4.40]{AbelianVarieties} for the fact that the quotient of an elliptic curve by a finite subgroup scheme is an elliptic curve again. Moreover, we will denote the right vertical morphism by $\beta_1$. The reason for our particular choice of $\beta_1$ might become clearer in the next subsection and even clearer when we discuss $q$-expansions. That $\beta_1$ actually lands in $\MFC(\Gamma_1(n), R_0)$ will follow from \cref{sec:Comparisons} and \cref{sec:qexpansions}.

\begin{remark}\label{rem:equivariance}
 The group $(\Z/n)^\times$ acts on $\Ell^1_{\Gamma_1(n)}(-)$ by multiplication on the point of order $n$. Moreover, if we denote by $\Gamma_0(n)\subset \SL_2(\Z)$ the subgroup of matrices $\begin{pmatrix}a&b\\c&d\end{pmatrix}$ with $c$ divisible by $n$, the quotient group  $\Gamma_1(n)\!\setminus\! \Gamma_0(n)$ acts on $\MFC(\Gamma_1(n), \C)$ as follows. For $g\in \MFC_k(\Gamma_1(n), \C)$ and $\gamma \in \Gamma_0(n)$, we define the action by $g. [\gamma] = g[\gamma]_k$ in the sense of \cref{sec:analytic}. The map 
 \[\Gamma_1(n)\!\setminus\! \Gamma_0(n) \to (\Z/n)^\times, \qquad \begin{pmatrix}a&b\\c&d\end{pmatrix} \mapsto a\]
 is an isomorphism and under this isomorphism $\beta_1$ is equivariant.
 
 To be compatible with \cite[Section 5.2]{DiamondShurman}, we will actually work with the \emph{opposite} convention though. This means that we will act with the \emph{inverse} of an element of $(\Z/n)^\times$ on $\Ell^1_{\Gamma_1(n)}(-)$ and $\MM_1(n)$ and use the identification 
  \[\Gamma_1(n)\!\setminus\! \Gamma_0(n) \xrightarrow{\cong} (\Z/n)^\times, \qquad \begin{pmatrix}a&b\\c&d\end{pmatrix} \mapsto d.\]
  By the above, this makes $\beta_1$ into an equivariant map as well and this will the equivariance we will use throughout this document. 
\end{remark}
 
\subsubsection{Arithmetic level structures}\label{sec:Arithmetic}
Now we would like to discuss a different variant of level structures, called ``arithmetic'' in the literature. 
\begin{Definition}
 For an $R_0$-algebra $R$, let $\Ell^1_{\Gamma_{\mu}(n)}(R)$ denote the set of isomorphism classes of triples $(E, \omega, \iota)$, where $E$ is an elliptic curve over $R$, again $\omega$ is a chosen trivialization of the line bundle $\omega_E$, and $\iota \colon \mu_{n,R} \to E$ is a morphism of group schemes over $\Spec(R)$ and a closed immersion. Here, $\mu_{n,R}$ is a group scheme given by the spectrum of the bialgebra $R[t]/(t^n-1)$ with comultiplication determined by $t\mapsto t\otimes t$. The morphism $\iota$ is called an \emph{arithmetic (or $\Gamma_{\mu}(n)$-) level structure} on $E$. 
\end{Definition}
 For a $\Z\left[\frac{1}{n}, \zeta_n\right]$-algebra $R$, the group schemes $\mu_{n,R}$ and $\mathbb{Z}/n\mathbb{Z}_R$ are isomorphic, but this is not true in general. 
 
 We can define the set of weight $k$ modular forms with arithmetic level structure to be $ \Nat^{R_0}_k(\Ell^1_{\Gamma_{\mu}(n)}(-), \Gamma(-))$ with the same scaling condition as before. Likewise, we can define a moduli stack $\MM_{\mu}(n)$ of elliptic curves with $\Gamma_{\mu}(n)$-level structure (over bases with $n$ invertible). As before we obtain a comparison isomorphism
 $$\alpha\colon H^0(\MM_{\mu}(n); \omline^{\tensor k}) \to \Nat^{R_0}_k(\Ell^1_{\Gamma_{\mu}(n)}(-), \Gamma(-)),$$
 where we abuse notation again to denote the pullback of $\omline$ to $\MM_{\mu}(n)$ by $\omline$ as well.

We need to discuss a relation between $\Gamma_1(n)$- and $\Gamma_{\mu}(n)$-level structures. After base change to a $\Z\left[\frac{1}{n}, \zeta_n\right]$-algebra $R$, the stacks $\MM_{\mu}(n)$ and $\MM_1(n)$ become equivalent over $\MM_{ell, R}$ via the isomorphism $\mu_{n,R} \cong \mathbb{Z}/n\mathbb{Z}_R$. Less obviously, there is also a different equivalence between $\MM_{\mu}(n)$ and $\MM_1(n)$ that does not require any base change, but changes the underlying elliptic curve. To that purpose we recall the Weil paring \cite[Section 2.8]{KatzMazur}
\[e_n\colon E[n](S)\times E[n](S) \to \mathbb{G}_{m,S}(S)\]
for an elliptic curve $E/S$.
Here, $E[n]$ denotes the $n$-torsion $E\times_E S$, using the multiplication-by-$n$ morphism $[n]\colon E \to E$ and the unit morphism $S\to E$ in the pullback. Using these ingredients, we add in the following lemma some details to the treatment in \cite[Section 2.3]{KatzRealEisenstein}.
\begin{lemma}\label{lem:phiequivalence}
There is an equivalence $\varphi\colon \MM_{1}(n) \to \MM_{\mu}(n)$ sending $(E\to S,P)$ to $(E/\langle P \rangle\to S, \delta)$, where $\delta$ can be described as follows: For $\zeta\in \mu_n(S)$, choose $Q\in E[n](S)$ such that $e_n(P,Q) = \zeta^{-1}$. Then $\delta(\zeta) = \pi(Q)$ for $\pi\colon E \to E/\langle P\rangle$.
\end{lemma}
\begin{proof}With notation as in the statement of the lemma, we define $\delta\colon \mu_{n,S} \to E/\langle P \rangle$ as follows: As explained in \cite[Section 2.8]{KatzMazur} there is a bilinear pairing
\begin{equation}\langle -,-\rangle_{\pi}\colon \ker(\pi)\times \ker(\pi^t) \to \mathbb{G}_{m,S}\end{equation}
 of abelian group schemes for $\pi\colon E \to E/\langle P\rangle$ the projection and $\pi^t$ the dual isogeny. By \cite[2.8.2.1]{KatzMazur} and because $\ker(\pi) = \langle P\rangle \cong (\Z/n)_S$, this induces a chain of isomorphisms 
 \begin{equation}\label{eq:pitcomp}\ker(\pi^t) \to \Hom_{S-\mathrm{gp}}(\ker(\pi),\mathbb{G}_{m,S})\xrightarrow{\ev_P} \mu_{n,S}.\end{equation}
 The map $\delta$ is the composition of the inverse of this isomorphism with the natural inclusion $\ker(\pi^t) \to E/\langle P \rangle$ \emph{composed with $[-1]$}. The reasons for composing with $[-1]$ will be apparent in the example below. 
 
An analogous construction dividing out $\mu_{n,S}$ provides an inverse of $\varphi$. To see this, we are using that in the situation above, $(E/\langle P \rangle)/\delta \cong E/E[n]$, and the isomorphism $E/E[n] \cong E$ induced by $[n]$, the multiplication-by-$n$ morphism. Thus, $\varphi\colon \MM_1(n) \to \MM_{\mu}(n)$ is an equivalence of stacks.

One can compute $\varphi$ in terms of the Weil pairing as follows: As $\pi\pi^t = [n]$, we obtain from \cite[2.8.4.1]{KatzMazur} that $\langle P, \pi(Q) \rangle_{\pi}$ for $Q\in E[n](S)$ can be computed as $e_n(P,Q)$. Consider now $\zeta \in \mu_n(S)$. The inverse of the composition \eqref{eq:pitcomp} sends $\zeta$ to $\pi(Q)$ for some $Q \in E[n](S)$ with $e_n(P,Q) = \zeta$. We obtain $e_n(P,-Q) = \zeta^{-1}$ showing the result.
 \end{proof}

\begin{example}\label{exa:Tate}
Let $E = \C/(\Z +n \tau\Z)$ be an elliptic curve over $\Spec \C$ with chosen $n$-torsion point $\tau$. We claim that $\varphi(E,\tau) = (\C/\Z+\tau\Z, \zeta_n\mapsto \frac1n)$ with $\zeta_n = e^{\frac{2\pi i}{n}}$. 
Indeed, we have $e_n(\tau, \frac1n) = \zeta_n^{-1}$ by \cite[2.8.5.3]{KatzMazur} and thus $\zeta_n$ has to be send to $\frac1n$ as claimed by the preceding lemma. 
\end{example}

The example implies directly the following lemma. 

\begin{lemma}\label{KatztoComplex}
The following diagram commutes: 
\begin{equation*}
 \begin{tikzcd}[column sep=1cm]
  H^0(\MM_{\mu}(n)_{R_0}, \omline^{\otimes k}) \arrow[r, "\varphi^*"]  \arrow[d, "\alpha"] & H^0(\MM_{1}(n)_{R_0}, \omline^{\otimes k})\arrow[d, "\alpha"] \\
\Nat^{R_0}_k(\Ell^1_{\Gamma_{\mu}(n)}(-), \Gamma(-))  \arrow[r, "\varphi^*"]  \arrow[dr, "{(\mathbb{C}/\mathbb{Z}+\tau\mathbb{Z}, dz, \zeta_n \mapsto \frac{1}{n})}" swap] & \Nat^{R_0}_k(\Ell^1_{\Gamma_1(n)}(-), \Gamma(-)) \arrow[d, "{(\mathbb{C}/\mathbb{Z}+n\tau\mathbb{Z}, dz, \tau)}" ]\\
 &\MFC(\Gamma_1(n); R_0)
 \end{tikzcd}
\end{equation*}
\end{lemma}
We will denote the diagonal arrow by $\beta_{\mu}$ and it will follow from \cref{sec:Comparisons} and \cref{sec:qexpansions} that $\beta_{\mu}$ actually lands in $\MFC(\Gamma_1(n); R_0)$.

\subsubsection{Compactifications and comparison of algebraic and analytic theory}\label{sec:Comparisons}
In this section we discuss the comparison of the algebraic and the analytic theory. The basic sources are \cite{DeligneRapoport} and \cite{ConradRamanujan} and we will just give a short summary. We will use the compactifications $\MMb_1(n)$ of $\MM_1(n)$ as recalled in the beginning of \cref{sec:ModularForms}. It is shown in \cite[Section IV]{DeligneRapoport} that $\MMb_1(n) \to \Spec \Z[\tfrac1n]$ is proper and smooth of relative dimension $1$. 

For $n\geq 5$, the stack $\MMb_1(n)$ is representable by a projective scheme (see e.g.\ \cite{MeierDecomposition}). It is shown in \cite[Thm. 2.2.2.1]{ConradRamanujan} that the Riemann surface associated with $\MMb_1(n)_{\C}$ is isomorphic to a more classical construction, namely the compactification $X_1(n)$ of the quotient $Y_1(n)$ of the upper half-plane $\mathbb{H}$ by $\Gamma_1(n)$. Indeed, Conrad shows that both $\MMb_1(n)_{\C}$ and $X_1(n)$ classify generalized elliptic curves over complex analytic spaces with $\Gamma_1(n)$-level structure. The family of elliptic curves $(\C/\Z+n\tau\Z, \tau)$ with $\Gamma_1(n)$-level structure over $\mathbb{H}$ descends to $Y_1(n)$ and extends to $X_1(n)$. (Indeed, Conrad considers the universal family $(\C/\Z+\tau\Z, \frac1n)$ as in \cite[Section 2.1.3]{ConradRamanujan}, but the choice of $e^{2\pi i/n}$ as an $n$-th root of unity allows us to consider the automorphism $\MM_1(n)_{\C} \xrightarrow{\varphi} \MM_{\mu}(n)_{\C} \simeq \MM_1(n)_{\C}$ that carries one family of elliptic curves into the other as follows from Example \ref{exa:Tate}.) This specifies an isomorphism from $X_1(n)$ to the Riemann surface associated with $\MMb_1(n)_{\C}$, and by restriction to the locus where the fibers of the universal generalized elliptic curve are smooth, also an isomorphism from $Y_1(n)$ to the Riemann surface associated with $\MM_1(n)_{\C}$. More information about $Y_1(n)$ and $X_1(n)$ can be found in \cite{ConradRamanujan} and in \cite[Chapter 2]{DiamondShurman}.

We will abuse notation again and denote by $\omline$ the line bundle on $X_1(n)$ corresponding to the analytification of $\omline$ on $\MMb_1(n)_{\C}$ under the isomorphism above and likewise its restriction to $Y_1(n)$. By GAGA \cite[Th\'eor\`eme 1]{GAGA}, the morphism $H^0(\MMb_1(n)_{\C}; \omline^{\tensor k}) \to H^0(X_1(n);\omline^{\tensor k})$ is an isomorphism. 
Given a section of $\omline^{\tensor k}$ on $Y_1(n)$ we can pull it back along $\pi\colon \mathbb{H} \to Y_1(n)$ and obtain a holomorphic function on $\mathbb{H}$ by trivializing $\pi^*\omline$ via $dz$. It is shown in \cite[Lemma 1.5.7.2]{ConradRamanujan} that the image consists exactly of those holomorphic functions on $\mathbb{H}$ satisfying the transformation formula \eqref{eq:transformation} for modular forms of weight $k$ for $\Gamma_1(n)$. Moreover, Conrad shows that the image of $H^0(X_1(n);\omline^{\tensor k}) \hookrightarrow H^0(Y_1(n);\omline^{\tensor k})$ corresponds exactly to the \emph{holomorphic} modular forms of weight $k$ for $\Gamma_1(n)$.

In summary, we obtain an isomorphism $\psi\colon H^0(\MMb_1(n)_{\C};\omline^{\tensor k}) \cong \MF_k(\Gamma_1(n);\C)$. Unraveling the definitions from \cite[Section 1.5.1 and 1.5.2]{ConradRamanujan} shows that this is compatible with our comparison map
$$\beta_1\alpha\colon H^0(\MM_1(n)_{\C}; \omline^{\tensor k} )\xrightarrow{\cong} \Nat^{\C}_k(\Ell_{\Gamma_1(n)}(-),\Gamma(-)) \to \{\text{functions on }\H\}.$$
We will argue why $\beta_1\alpha$ actually takes values in $\MFC_k(\Gamma_1(n);\C)$ (as claimed before) and why with this target $\beta_1\alpha$ becomes an isomorphism. 

The modular form $\Delta\in \MF_{12}(\SL_2(\Z);\Z) \subset \MF_{12}(\Gamma_1(n);\C)$ (see \cref{sec:analytic}) corresponds to a holomorphic section of $\omline^{\tensor 12}$ on $X_1(n)$ with simple zeros at all cusps, i.e.\ at all those points in the complement of $Y_1(n)$; this can be seen by considering the $q$-expansion of $\Delta$ and the construction of the $X_1(n)$. Thus, $H^0(X_1(n); \omline^{\tensor *})[\Delta^{-1}]$ corresponds exactly to those holomorphic sections of $\omline^{\tensor *}$ on $Y_1(n)$ that can be meromorphically extended to $X_1(n)$. This in turn corresponds exactly to the (algebraic) sections of $\omline^{\tensor *}$ on $\MM_1(n)_{\C}$. This implies an identification $H^0(\MM_1(n)_{\C}; \omline^{\tensor *}) \cong H^0(\MMb_1(n)_{\C}; \omline^{\tensor *})[\Delta^{-1}]$. Under this identification, $\beta_1\alpha$ can be written as
\[H^0(\MMb_1(n)_{\C};\omline^{\tensor *})[\Delta^{-1}] \underset{\psi}{\xrightarrow{\cong}} \MF_*(\Gamma_1(n); \omline^{\tensor *})[\Delta^{-1}] \cong \MFC_*(\Gamma_1(n);\C),\]
(followed by the inclusion into $\{\text{functions on }\H\}$). This shows our claims.  

For $n<5$, $\MMb_1(n)$ is no longer a scheme. In this case, one can analogously use a GAGA theorem for stacks as, for example, proven in \cite{PortaYu}. In our situation the proof should be considerably simplified though as $\MMb_1(n)_{\C}$ has a finite faithfully flat cover by a scheme (e.g.\ by $\MMb_1(5n)_{\C}$) and one should be able to deduce a sufficiently strong GAGA theorem just by descent from the scheme case. 

\subsection{The Tate curve}\label{sec:AppTate}
In this section, we will discuss the Tate curve, which will give us an algebraic way to define $q$-expansions of modular forms. For an alternative treatment we refer e.g.\ to \cite[Section 8.8]{KatzMazur}. We first discuss the situation over the complex numbers. 

\begin{Theorem}[{\cite[Theorem V.1.1]{SilvermanAdvanced}}] \label{TateWeierstrass}
For any $q,u\in \C$ with $|q|<1$, define the following quantities:
\[
 \begin{aligned}
  \sigma_k(n)&=&&\sum_{d|n} d^k,\\
  s_k(q)&=&&\sum_{n\geq 1} \sigma_k(n)q^n=\sum_{n\geq 1} \frac{n^kq^n}{1-q^n},\\
  a_4(q)&=&& -5s_3(q),\\
  a_6(q)&=&&-\frac{5s_3(q)+7s_5(q)}{12},\\
  X(u,q)&=&&\sum_{n\in\Z} \frac{q^nu}{(1-q^nu)^2}-2s_1(q),\\
  Y(u,q)&=&&\sum_{n\in\Z} \frac{(q^nu)^2}{(1-q^nu)^3}+s_1(q).
 \end{aligned}
\]

\begin{enumerate}
 \item  Then the equation
\begin{equation}\label{TateWeierstrassFormula}
 y^2+xy=x^3+a_4(q)x+a_6(q)
\end{equation}
defines an elliptic curve $E_q$ over $\C$, and $X,Y$ define a complex analytic isomorphism 
\[
 \begin{aligned}
  \C^{\times}/q^{\Z} &\to && E_q\\
  u &\mapsto && \begin{cases}
                 (X(u,q), Y(u,q)), \mbox{ if } u\notin q^{\Z},\\
		   O,\mbox{ if }u \in q^{\Z}
                \end{cases}
 \end{aligned}
\]
\item The power series $a_4(q)$ and $a_6(q)$ define holomorphic functions on the open unit disk $\mathbb{D}$.
\item As power series in $q$, both $a_4(q),a_6(q)$ have integer coefficients. 
 
\item The discriminant of $E_q$ is given by 
\[
 \Delta(q)=q\prod_{n\geq 1} (1-q^n)^{24} \in \mathbb{Z}\llbracket q\rrbracket.
\]

\item Every elliptic curve over $\C$ is isomorphic to $E_q$ for some $q$ with $|q|<1$. 

\end{enumerate}
\end{Theorem}

Let $\Conv \subset \Z ((q))$ be the subset of ``convergent'' Laurent series, i.e.\ those that define meromorphic functions on $\mathbb{D}$ that are holomorphic away from $0$; in particular, $a_4, a_6\in\Conv$. As $\Delta(q)$ is non-vanishing for $q\ne 0$ in $\mathbb{D}$, it defines an invertible element in $\Conv$ and thus we can use the Weierstra\ss{} equation \eqref{TateWeierstrassFormula} to define an elliptic curve $\Tate(q)$ over $\Conv$. For our computations in \cref{sec:TateCurve} it will be convenient to consider the analogously defined ring $\Conv_{q^n} \subset \Z((q^n))$ with the Tate curve $\Tate(q^n)$ defined by $a_4(q^n)$ and $a_6(q^n)$ over it.

Let $q_0\in\mathbb{D}$ be a nonzero point and consider the morphism $\ev_{q_0}\colon \Conv \to \C$. By the theorem above, we see that the analytic space associated with $\ev_{q_0}^*\Tate(q)$ is isomorphic to $\C^\times/q_0^\Z$. The invariant differential $\eta^{can}$ associated to the Weierstra\ss{} equation corresponds under this isomorphism to $\frac{du}{u}$, as can be shown by elementary manipulations using \cite[Section V.1]{SilvermanAdvanced}.

Next, we want to describe a group homomorphism $\iota\colon \mu_{n,\Conv[\frac1n]} \to \Tate(q)_{\Z[\frac1n]}$ for $n\geq 2$. We first define a morphism $\iota_{\zeta}\colon \mu_{n,\Conv[\frac1n, \zeta_n]} \to \Tate(q)_{\Z[\frac1n, \zeta_n]}$. As $\mu_n$ is isomorphic to $\Z/n$ over $\Z[\frac1n,\zeta_n]$, it suffices to specify an $n$-torsion point in $\Tate(q)(\Conv[\frac1n, \zeta_n])$ as the image of $\zeta_n$; we take $(X(\zeta_n,q), Y(\zeta_n, q))$. As $X$ and $Y$ have integer coefficients, we see that for every ring automorphism $\sigma$ of $\Z[\frac1n,\zeta_n]$, we have $\iota_{\zeta}(\sigma(\zeta_n)) = \sigma(\iota_{\zeta}(\zeta_n))$. Thus, Galois descent implies that $\iota_{\zeta}$ descends to a morphism $\iota\colon \mu_{n,\Conv[\frac1n]} \to \Tate(q)_{\Z[\frac1n]}$. Note that we can check that this is indeed a group homomorphism into the $n$-torsion by evaluating at infinitely many points in $\mathbb{D}$. For a nonzero $q_0\in\mathbb{D}$, this $\iota$ corresponds under the 
isomorphism of $\ev_{q_0}^*\Tate(q)$ with $\C^\times/q_0^\Z$ exactly to the composite $\mu_n(\C) \to \C^\times \to \C^\times/q_0^\Z$. Note that 
$\iota$ defines a $\Gamma_{\mu}(n)$-structure on $\Tate(q)_{\Z[\frac1n]}$. 

As a last point, we mention that for a subring $R\subset \C$ containg $\zeta_n$, the $n$-torsion $\Tate(q^n)_R[n]$ is isomorphic to $(\Z/n)_{\Conv_{q^n,R}}^2$ as it has rank $n^2$ over $\Conv_R$ \cite[Theorem 2.3.1]{KatzMazur} and we can specify $n^2$ points by $(X(\zeta_n^aq^b, q^n), Y(\zeta_n^aq^b, q^n))$, where $0\leq a, b \leq n-1$.

\subsection{$q$-expansions}\label{sec:qexpansions}
Our goal in this subsection is to define the $q$-expansion both in the holomorphic and in the algebraic context, to compare them and to obtain a $q$-expansion principle. 

Consider a modular form $f$ in $\MFC(\Gamma_1(n);\C)$ for $n\geq 1$. We recall that $f$ factors through a meromorphic function $g\colon \mathbb{D} \to \C$ on the open unit disk with only possible pole in 0; more precisely, we have $g(q) = f(z)$, where $q = q(z) = e^{2\pi i z}$. Taylor expansion of $g$ at $0$ yields the classical $q$-expansion
$$\Phi^{hol}\colon \MFC_k(\Gamma_1(n);\C) \to \C((q)).$$

Let us fix for the whole subsection a $\Z[\frac1n]$-subalgebra $R_0\subset \C$. On the algebraic side, we obtain a map 
$$\Phi^{\mu, R_0}\colon \Nat^{R_0}_k(\Ell^1_{\Gamma_{\mu}(n)}(-), \Gamma(-)) \to R_0((q))$$
by evaluating the natural transformation at the pullback to $R_0((q))$ of the Tate curve $(\Tate(q), \eta_{can}, \iota)$ from the last section. 

We want to show that $\Phi^{hol}$ and $\Phi^{\mu,\C}$ correspond to each other under $\beta_{\mu}$. Both have actually image in $\Conv_{R_0} \subset \C((q))$. Thus we can check the agreement of $\Phi^{hol}\beta_{\mu}$ with $\Phi^{\mu,\C}$ after postcomposing these two maps with $\ev_{q_0}\colon \Conv_{R_0} \to  \C$ for infinitely many $q_0 \in \mathbb{D} \setminus \{0\}$. 

To that purpose, choose $h \in \Nat^{R_0}_k(\Ell^1_{\Gamma_{\mu}(n)}(-), \Gamma(-))$ and $\tau_0\in\H$ with $e^{2\pi i \tau_0} = q_0$. The exponential defines an isomorphism
\[(\C/(\Z+\tau_0\Z), dz, \zeta_n\mapsto \frac1n) \cong (\C^\times/q_0^\Z, \frac{du}u, \iota^{can}),\]
of elliptic curves with invariant differential and arithmetic level structure, where $\iota^{can}$ denotes the composition $\mu_n(\C) \to \C^\times \to \C^\times/q_0^\Z$.
This implies
\[\ev_{q_0}\Phi^{hol}\beta_{\mu}(h) = h(\C^\times/q_0^\Z, \frac{du}u,\iota^{can}).\]

On the other hand, $\ev_{q_0}\Phi^{\mu,\C}(h)$ is by definition the evaluation of $h$ at
\[
(\ev_{q_0}^*\Tate(q), \ev_{q_0}^*\eta^{can}, \ev_{q_0}^*\iota)
\] 
and we have seen in the last section that this triple is isomorphic to $(\C^\times/q_0^\Z, \frac{du}u, \iota^{can})$. Thus, the following triangle commutes indeed:

\[
\begin{tikzcd}
\Nat^{R_0}_k(\Ell^1_{\Gamma_{\mu}(n)}(-), \Gamma(-)) \arrow[r, "{\Phi^{\mu,\C}}"]\arrow[d, "\beta_{\mu}"]& \C((q)) \\
\MFC_k(\Gamma_1(n),\C) \arrow[ur, "\Phi^{hol}" swap]
\end{tikzcd}
\]

We obtain the $q$-expansion morphism
$$\Phi^{1,R_0}\colon \Nat^{R_0}_k(\Ell^1_{\Gamma_1(n)}(-), \Gamma(-)) \to \Conv_{R_0}$$
as the composition $\Phi^{\mu,R_0}\alpha(\varphi^*)^{-1}\alpha^{-1}$, where $\varphi$ is as in Subsection \ref{sec:Arithmetic}. 

\begin{lemma}
Assume that $R_0 \subset \C$ and let $q_0\neq 0$ be a point in the open unit disk. Evaluating at $q_0$ yields a morphism $\ev_{q_0}\colon \Conv_{R_0} \to \C$. Then
$$\ev_{q_0}\Phi^{1,R_0}(h) = h(\C^\times/q_0^{n\Z}, \frac{du}u, q_0)$$ 
for every $h \in \Nat^{R_0}_k(\Ell^1_{\Gamma_1(n)}(-), \Gamma(-))$ and thus $\Phi^{1,R_0}(h)$ is the Taylor expansion of 
\[q \mapsto h(\C^\times/q^{n\Z}, \frac{du}u, q)\]
at $0$.
\end{lemma}
\begin{proof}
It suffices to show that
$$\varphi(\C^\times/q_0^{n\Z}, \frac{du}u, q_0) = (\C^\times/q_0^\Z, \frac{du}u, \iota^{can}).$$
This follows from Example \ref{exa:Tate}.
\end{proof}

Note that these discussions show that $\beta_1$ and $\beta_{\mu}$ actually have target in the ring $\MFC(\Gamma_1(n); R_0)$, i.e.\ that the classical $q$-expansion of $\beta_1$ of a modular form over $R_0$ actually has coefficients in $R_0$ and similarly for $\beta_{\mu}$. 

\begin{Theorem}[$q$-expansion principle]\label{thm:q-exp}
Let $R_0$ be a subring of $\C$. 
The morphisms
$$\beta_{\mu}\colon \Nat^{R_0}_k(\Ell^1_{\Gamma_{\mu}(n)}(-), \Gamma(-)) \to \MFC(\Gamma_1(n); R_0)$$
and  
$$\beta_1\colon \Nat^{R_0}_k(\Ell^1_{\Gamma_1(n)}(-), \Gamma(-)) \to \MFC(\Gamma_1(n); R_0)$$
are isomorphisms. In other words: If the coefficients of the $q$-expansion of a complex modular form are in $R_0$, it is actually already defined over $R_0$. 
\end{Theorem}
\begin{proof}
By the considerations above, it suffices to show the first statement. For $R_0 = \C$, this was discussed in Subsection \ref{sec:Comparisons}. The general case follows by the $q$-expansion principle as stated in \cite[Theorem 12.3.4]{DiamondIm}. 
\end{proof}

\subsection{Summary}\label{sec:summary} 
Let $R$ be any $\Z[\frac1n]$-algebra. We can define holomorphic modular forms for $\Gamma_1(n)$ of weight $k$ over $R$ as $H^0(\MMb_1(n)_R; \omline^{\tensor k})$ and meromorphic modular forms as $H^0(\MM_1(n)_R;\omline^{\tensor k})$. We have a morphism 
$\Spec \C \to \MM_1(n)$ classifying the elliptic curve $\C/\Z+n\tau\Z$ with chosen point $\tau$ of order $n$. Pulling $f \in H^0(\MM_1(n);\omline^{\tensor k})$ back to $\Spec \C$ and using the trivialization $\omline^{\tensor k}$ induced by the choice of differential $dz$, defines a holomorphic function of $\tau \in \mathbb{H}$ that is a meromorphic modular form for $\Gamma_1(n)$ in the classical sense. This defines isomorphisms
$$\beta_1\colon H^0(\MM_1(n)_{\C}; \omline^{\tensor k}) \to \MFC_k(\Gamma_1(n);\C)$$ 
and 
\[\beta_1\colon H^0(\MMb_1(n)_{\C}; \omline^{\tensor k}) \to \MF_k(\Gamma_1(n);\C).\]
 The $q$-expansion of $\beta_1(f)$ lies in $R\subset \C$ if and only if $f$ is in the image of the injection
 $$H^0(\MM_1(n)_R; \omline^{\tensor k}) \to H^0(\MM_1(n)_{\C}; \omline^{\tensor k}).$$
 As a last point, we consider the $\mathbb{G}_m$-torsor $\MMb_1^1(n) \to \MMb_1(n)$ that is the relative $\Spec$ of the quasi-coherent algebra
 $\bigoplus_{k\in\Z} \omline^{\tensor k}$. By construction, 
 \[H^0(\MMb_1^1(n)_{R};\OO_{\MMb_1^1(n)}) \cong \bigoplus_{k\in\Z}H^0(\MMb_1(n)_{R}; \omline^{\tensor k}) \cong \bigoplus_{k\in\Z} \MF_k(\Gamma_1(n); R). \]

\section{The invertible summand \texorpdfstring{in $\Tmf_0(7)$}{}\\ (joint with Martin Olbermann)}\label{AppB}
We recall from \cref{thm:topmain} that $\Tmf_0(7)_{(3)}$ splits as a $\Tmf$-module as
\[\Tmf_{(3)} \oplus \Sigma^4\Tmf_1(2)_{(3)} \oplus \Sigma^8\Tmf_1(2)_{(3)}\oplus L,\]
where $L \in \Pic(\Tmf_{(3)})$. The goal of this appendix is to determine $L$. The necessary computations of $\pi_*\Tmf_0(7)$ were obtained by Martin Olbermann. It turns out that for the purposes of this article, we need only a small part of these computations, which the authors of the main part of this article extracted from Olbermann's computations. 

We recall from \cite{MathewStojanoska} that $\Pic(\Tmf_{(3)}) \cong \Z \oplus \Z/3$. More precisely, their computation shows that the morphisms
\[\Z/72 \to \Pic(\TMF_{(3)}), \quad [k] \mapsto \Sigma^k\TMF\]
and
\[\Z \to \Pic(\Tmf_{\Q}),\quad k \mapsto \Sigma^k \Tmf_{\Q}\]
are isomorphisms and moreover that 
\[\Pic(\Tmf_{(3)}) \to \Z/72 \times_{\Z/24} \Z \subset \Pic(\TMF_{(3)}) \times \Pic(\Tmf_{\Q})\]
is an isomorphism as well. (While this last fact is not explicitly stated in \cite{MathewStojanoska}, it is clearly visible in the proof of their Theorem B.)

As described in \cref{thm:topmain}, we obtain $L$ as the global sections of an invertible $\OO^{top}$-module $\mathcal{L}$ on $\MMb_{ell,(3)}$ with $\pi_0\mathcal{L} \cong \omline^{\tensor (-6)}$. As $\pi_0\Sigma^{12}\OO^{top} \cong \omline^{\tensor (-6)}$ as well and global sections define an equivalence
$$\Gamma\colon \mathrm{QCoh}(\MMb_{ell, (3)},\OO^{top}) \to \Tmf_{(3)}\mathrm{-mod}$$
of $\infty$-categories \cite{MathewMeier}, we see that the image of $L$ in $\Pic(\Tmf_{\Q})$ is $\Sigma^{12}\Tmf_{\Q}$. We will show:

\begin{Prop}
The image $L[\Delta^{-1}]$ of $L$ in $\Pic(\TMF_{(3)})$ is $\Sigma^{36}\TMF_{(3)}$ and hence $L \simeq \Sigma^{36}\Gamma(\mathcal{J}^{\tensor (-1)})$ in the notation from \cite[Construction 8.4.2]{MathewStojanoska}.
\end{Prop}
In the following, we will leave the localization at $3$ for the moduli stacks, rings of modular forms and variants of $\TMF$ implicit to avoid clutter in the notation. We already know from the discussion above that $L[\Delta^{-1}] \simeq \Sigma^k\TMF$ for $k=12, 36$ or $60$. Moreover, the descent spectral sequence for $L[\Delta^{-1}]$ embeds into that of $\TMF_0(7)$ as a summand. Recall that the latter has $E_2$-term $H^*(\MM_0(7);\omline^{\tensor *})$. Since $\MM_0(7)$ has the $(\Z/7)^\times$-Galois cover $\MM_1(7)$, we use the definition of \v{C}ech cohomology to identify this $E_2$-term with 
$H^*((\Z/7)^\times; \MFC_1(7))$, where $\MFC_1(7)$ is used as our abbreviation for $\MFC(\Gamma_1(7);\Z[\frac17])$.
Actually, as 
\[\OO^{top}(\MM_1(7)^{\times_{\MM_0(7)}k}) \simeq \prod_{(\zsk)^{\times k}}\TMF_1(7)\]
and $\MM_1(7)$ is affine, we obtain even an identification of the cosimplicial objects defining the descent spectral sequence for $\TMF_0(7)$ and the homotopy fixed point spectral sequence for $\TMF_1(7)^{h\zsk} \simeq \TMF_0(7)$, and hence an isomorphism of these spectral sequences. 

Moreover, the descent spectral sequence for $L[\Delta^{-1}]$ has $E_2$-term isomorphic to $H^0(\MM_{ell}; \omline^{\tensor (* - 6)})$. Under these identifications, the embedding of descent spectral sequences sends $1 \in H^0(\MM_{ell}; \omline^{\tensor (6-6)})$ to $\sigma_3^2 \in H^0((\Z/7)^\times; \MFC_1(7)_6)$. This follows after identification of source and target with the primitive elements in the $(\widetilde{A},\widetilde{\Gamma})$ comodules $\widetilde{A}$ and $S_{\widetilde{A}}$ from \cref{prop:comoduleiso}. As $d_5(\Delta) = \alpha\beta^2$ and $d_5(\Delta^2) = -\Delta\alpha\beta^2$, while $d_5(1) = 0$ in the descent spectral sequence for $\TMF$ itself \cite{Konter}, it suffices to show the following lemma. 

\begin{lemma}
In the descent spectral sequence for $\TMF_0(7)$, the class $\Delta\sigma_3^2 \in H^0(\MM_0(7); \omline^{\tensor (-6)})$ has a trivial $d_5$-differential.
\end{lemma}
\begin{proof}
Our first tool is the map of descent spectral sequences from that for $TMF$ to that for $TMF_0(7)$, which on the $0$-line of the $E_2$-term is a map 
\begin{equation}\label{eq:DSSmap} H^0(\MM_{ell}; \omline^{\otimes *}) \to H^0(\MM_0(7); \omline^{\otimes *}). \end{equation}
Recall from above that
$H^0(\MM_0(7); \omline^{\otimes *}) \cong H^0(\zsk; \MFC_1(7))$. \Cref{Z6invariants} implies that we can express every element in the invariants as a polynomial in $\sigma_1, \sigma_3$ and $p=z_1^2z_2+z_2^2z_3+z_3^2z_1$. As every element in $H^0(\MM_{ell}; \omline^{\otimes *})$ is a polynomial in the $a_i$, \cref{prop:coordinatesm1n} and \cref{Alphas} give us a concrete way to calculate the map  \eqref{eq:DSSmap}. In particular, we obtain
\[\Delta \mapsto -\sigma_3^3p-8\sigma_3^4. \]
By the splitting from \cref{thm:MainTheorem} and the known differentials from the descent spectral sequence of $\TMF$, we see that there are no differentials shorter than a $d_5$ in the descent spectral sequence for $\TMF_0(7)$. In particular, we obtain that $\Delta$ is a $d_i$-cycle for $i<5$ in the descent spectral sequence for $\TMF_0(7)$, but $d_5(\Delta) = \alpha\beta^2$ (where we use the same notation for the images of $\Delta$ and $\alpha\beta^2$ in the descent spectral sequence for $\TMF_0(7)$ as in that for $\TMF$).

Our second tool is the transfer
\[\Tr\colon \MFC_1(7) = H^0(\{e\}; \MFC_1(7)) \to H^0((\Z/7)^\times; \MFC_1(7)), \qquad x \mapsto \sum_{g \in (\Z/7)^\times} gx. \]
We have $\Tr(x)y = \Tr(x \res(y)) = 0$ for all $x \in \MFC_1(7)$ and $y \in H^*((\Z/7)^\times, \MFC_1(7))$ with $\ast > 0$ by \cite[Formula V.3.8]{Brown}.
In particular, these elements act trivially on $H^*((\Z/7)^\times, \MFC_1(7))$ for $\ast > 0$. As $3$ is in the image of $\Tr$, we see in particular that  $H^*((\Z/7)^\times, \MFC_1(7))$ for $\ast > 0$ is $3$-torsion.

Moreover we claim that all elements in the image of the transfer $\Tr$
are permanent cycles in the homotopy fixed point spectral sequence (or, equivalently, the descent spectral sequence) converging to $\pi_*\TMF_0(7) = \pi_*\TMF_1(7)^{h(\Z/7)^\times}$. Indeed: Consider the $(\Z/7)^\times$-equivariant map $a\colon \zsk_+ \wedge \TMF_1(7) \to \TMF_1(7)$ induced by $\id_{\TMF_1(7)}$, where the action on the source is only on $\zsk_+$. On homotopy groups, this induces the map $(x_g)_{g\in\zsk} \mapsto \sum_{g\in\zsk} gx_g$. Thus, the map that $a$ induces on homotopy fixed point spectral sequences agrees in the $0$-line exactly with the transfer $\Tr$ under the identification $\MFC_1(7) \cong H^0(\zsk; \bigoplus_{\zsk} \pi_*\TMF_1(7))$. As the homotopy fixed point spectral sequence for $\zsk_+\wedge \TMF_1(7)$ is concentrated in the $0$-line, this implies that every element in the image of $\Tr$ is a permanent cycle. In particular, $\sigma_3p = \Tr(\frac{z_1^3z_2^2z_3}2)$ implies that $d_5(\sigma_3p) = 0$.  

Taken together, these tools imply the following computation:
\begin{align*}
    \alpha\beta^2 &= d_5(\Delta) \\
    &= d_5(-\sigma_3^3p-8\sigma_3^4) \\
    &= -\sigma_3^2d_5(\sigma_3p+8\sigma_3^2)-d_5(\sigma_3^2)(\sigma_3p+8\sigma_3^2)\\
    &= -8\sigma_3^2d_5(\sigma_3^2)-8\sigma_3^2d_5(\sigma_3^2)\\
    &= -\sigma_3^2d_5(\sigma_3^2).
\end{align*}
In total, we obtain $\sigma_3^{2}d_5(\sigma_3^2) = -\alpha\beta^2$ and deduce
\begin{align*}
    d_5(\Delta\sigma_3^2) &= \alpha\beta^2 \sigma_3^2 +\Delta d_5(\sigma_3^2) \\
    &= \alpha\beta^2 \sigma_3^2-8\sigma_3^4d_5(\sigma_3^2) \\
    &= 9\alpha\beta^2 \sigma_3^2=0.
\end{align*}
\end{proof}

\bibliographystyle{plain}
\bibliography{refTMF}

\begin{thebibliography}{10}

\bibitem{AdamekRosicky}
{Ad\'amek, Ji\v{r}\'\i{} and Rosick\'y, Ji\v{r}\'\i }.
\newblock {\em Locally presentable and accessible categories}, volume 189 of
  {\em London Mathematical Society Lecture Note Series}.
\newblock Cambridge University Press, Cambridge, 1994.

\bibitem{Bauertmf}
Tilman Bauer.
\newblock Computation of the homotopy of the spectrum {\tt tmf}.
\newblock In {\em Groups, homotopy and configuration spaces}, volume~13 of {\em
  Geom. Topol. Monogr.}, pages 11--40. Geom. Topol. Publ., Coventry, 2008.

\bibitem{BehrensModular}
Mark Behrens.
\newblock A modular description of the {$K(2)$}-local sphere at the prime 3.
\newblock {\em Topology}, 45(2):343--402, 2006.

\bibitem{BehrensOrmsby}
Mark Behrens and Kyle Ormsby.
\newblock On the homotopy of {$Q(3)$} and {$Q(5)$} at the prime 2.
\newblock {\em Algebr. Geom. Topol.}, 16(5):2459--2534, 2016.

\bibitem{BobGoerss}
Irina Bobkova and Paul~G Goerss.
\newblock {Topological resolutions in $K(2)$-local homotopy theory at the prime
  $2$}.
\newblock {\em Journal of Topology}, 11(4):918--957, 2018.

\bibitem{Bou90}
A.~K. Bousfield.
\newblock A classification of {$K$}-local spectra.
\newblock {\em J. Pure Appl. Algebra}, 66(2):121--163, 1990.

\bibitem{Brown}
Kenneth~S. Brown.
\newblock {\em Cohomology of groups}, volume~87 of {\em Graduate Texts in
  Mathematics}.
\newblock Springer-Verlag, New York, 1994.
\newblock Corrected reprint of the 1982 original.

\bibitem{BrunsHerzog}
Winfried Bruns and J\"urgen Herzog.
\newblock {\em Cohen-{M}acaulay rings}, volume~39 of {\em Cambridge Studies in
  Advanced Mathematics}.
\newblock Cambridge University Press, Cambridge, 1993.

\bibitem{Buzzard}
Kevin Buzzard.
\newblock {Weight one Eisenstein series}.
\newblock
  \url{http://wwwf.imperial.ac.uk/~buzzard/maths/research/notes/weight_one_eisenstein_series.pdf}.

\bibitem{BuzzardWeight1}
Kevin Buzzard.
\newblock Computing weight one modular forms over {$\Bbb C$} and
  {$\overline{\Bbb F}_p$}.
\newblock In {\em Computations with modular forms}, volume~6 of {\em Contrib.
  Math. Comput. Sci.}, pages 129--146. Springer, Cham, 2014.

\bibitem{Conrad}
Brian Conrad.
\newblock Arithmetic moduli of generalized elliptic curves.
\newblock {\em J. Inst. Math. Jussieu}, 6(2):209--278, 2007.

\bibitem{ConradRamanujan}
Brian Conrad.
\newblock {Modular forms and the Ramanujan conjecture}.
\newblock
  \url{https://www.math.leidenuniv.nl/~edix/public_html_rennes/brian.ps}, To
  appear.

\bibitem{dapres}
P.~Deligne.
\newblock Courbes elliptiques: formulaire d'apr\`es {J}. {T}ate.
\newblock pages 53--73. Lecture Notes in Math., Vol. 476, 1975.

\bibitem{DeligneRapoport}
P.~Deligne and M.~Rapoport.
\newblock Les sch\'emas de modules de courbes elliptiques.
\newblock In {\em Modular functions of one variable, {II} ({P}roc. {I}nternat.
  {S}ummer {S}chool, {U}niv. {A}ntwerp, {A}ntwerp, 1972)}, pages 143--316.
  Lecture Notes in Math., Vol. 349. Springer, Berlin, 1973.

\bibitem{DiamondIm}
Fred Diamond and John Im.
\newblock Modular forms and modular curves.
\newblock In {\em Seminar on {F}ermat's {L}ast {T}heorem ({T}oronto, {ON},
  1993--1994)}, volume~17 of {\em CMS Conf. Proc.}, pages 39--133. Amer. Math.
  Soc., Providence, RI, 1995.

\bibitem{DiamondShurman}
Fred Diamond and Jerry Shurman.
\newblock {\em A first course in modular forms}, volume 228 of {\em Graduate
  Texts in Mathematics}.
\newblock Springer-Verlag, New York, 2005.

\bibitem{TMFbook}
Christopher~L. Douglas, John Francis, Andr{\'e}~G. Henriques, and Michael~A.
  Hill, editors.
\newblock {\em Topological modular forms}, volume 201 of {\em Mathematical
  Surveys and Monographs}.
\newblock American Mathematical Society, Providence, RI, 2014.

\bibitem{AbelianVarieties}
Bas Edixhoven, Gerard van~der Geer, and Ben Moonen.
\newblock Abelian varieties.
\newblock \url{https://www.math.ru.nl/~bmoonen/research.html}, 2007.

\bibitem{Eisenbud}
David Eisenbud.
\newblock {\em Commutative algebra}, volume 150 of {\em Graduate Texts in
  Mathematics}.
\newblock Springer-Verlag, New York, 1995.
\newblock With a view toward algebraic geometry.

\bibitem{GoertzWedhorn}
Ulrich G\"ortz and Torsten Wedhorn.
\newblock {\em Algebraic geometry {I}}.
\newblock Advanced Lectures in Mathematics. Vieweg + Teubner, Wiesbaden, 2010.
\newblock Schemes with examples and exercises.

\bibitem{EGAIV}
A.~Grothendieck.
\newblock {\'El\'ements de g\'eom\'etrie alg\'ebrique. {IV}. {\'E}tude locale
  des sch\'emas et des morphismes de sch\'emas }.
\newblock {\em Inst. Hautes \'Etudes Sci. Publ. Math.}, (32):361, 1967.

\bibitem{Hartshorne}
Robin Hartshorne.
\newblock {\em Algebraic geometry}.
\newblock Springer-Verlag, New York, 1977.
\newblock Graduate Texts in Mathematics, No. 52.

\bibitem{HartshorneReflexive}
Robin Hartshorne.
\newblock Stable reflexive sheaves.
\newblock {\em Math. Ann.}, 254(2):121--176, 1980.

\bibitem{HillLawson}
Michael Hill and Tyler Lawson.
\newblock Topological modular forms with level structure.
\newblock {\em Inventiones mathematicae}, pages 1--58, 2015.

\bibitem{HillMeier}
Michael~A. Hill and Lennart Meier.
\newblock The {$C_2$}-spectrum {${\rm Tmf}_1(3)$} and its invertible modules.
\newblock {\em Algebr. Geom. Topol.}, 17(4):1953--2011, 2017.

\bibitem{HoveyMoritaHopf}
Mark Hovey.
\newblock Morita theory for {H}opf algebroids and presheaves of groupoids.
\newblock {\em Amer. J. Math.}, 124(6):1289--1318, 2002.

\bibitem{HusemollerEllCurves}
Dale Husem{\"o}ller.
\newblock {\em Elliptic curves}, volume 111 of {\em Graduate Texts in
  Mathematics}.
\newblock Springer-Verlag, New York, second edition, 2004.
\newblock With appendices by Otto Forster, Ruth Lawrence and Stefan Theisen.

\bibitem{Katz}
Nicholas~M. Katz.
\newblock {$p$}-adic properties of modular schemes and modular forms.
\newblock In {\em Modular functions of one variable, {III} ({P}roc. {I}nternat.
  {S}ummer {S}chool, {U}niv. {A}ntwerp, {A}ntwerp, 1972)}, pages 69--190.
  Lecture Notes in Mathematics, Vol. 350. Springer, Berlin, 1973.

\bibitem{KatzRealEisenstein}
Nicholas~M. Katz.
\newblock {$p$}-adic interpolation of real analytic {E}isenstein series.
\newblock {\em Ann. of Math. (2)}, 104(3):459--571, 1976.

\bibitem{KatzMazur}
Nicholas~M. Katz and Barry Mazur.
\newblock {\em Arithmetic moduli of elliptic curves}, volume 108 of {\em Annals
  of Mathematics Studies}.
\newblock Princeton University Press, Princeton, NJ, 1985.

\bibitem{Konter}
Johan Konter.
\newblock The homotopy groups of the spectrum {$Tmf$}.
\newblock {\em arXiv preprint arXiv:1212.3656}, 2012.

\bibitem{LangMod}
Serge Lang.
\newblock {\em Introduction to modular forms}, volume 222 of {\em Grundlehren
  der Mathematischen Wissenschaften [Fundamental Principles of Mathematical
  Sciences]}.
\newblock Springer-Verlag, Berlin, 1995.
\newblock With appendixes by D. Zagier and Walter Feit, Corrected reprint of
  the 1976 original.

\bibitem{Laumon}
G\'erard Laumon and Laurent Moret-Bailly.
\newblock {\em Champs alg\'ebriques}, volume~39 of {\em Ergebnisse der
  Mathematik und ihrer Grenzgebiete. 3. Folge. A Series of Modern Surveys in
  Mathematics [Results in Mathematics and Related Areas. 3rd Series. A Series
  of Modern Surveys in Mathematics]}.
\newblock Springer-Verlag, Berlin, 2000.

\bibitem{QingLiu}
Qing Liu.
\newblock {\em Algebraic geometry and arithmetic curves}, volume~6 of {\em
  Oxford Graduate Texts in Mathematics}.
\newblock Oxford University Press, Oxford, 2002.
\newblock Translated from the French by Reinie Ern\'e, Oxford Science
  Publications.

\bibitem{MathewThick}
Akhil Mathew.
\newblock A thick subcategory theorem for modules over certain ring spectra.
\newblock {\em Geom. Topol.}, 19(4):2359--2392, 2015.

\bibitem{Mathewtmf}
Akhil Mathew.
\newblock The homology of $tmf$.
\newblock {\em Homology, Homotopy and Applications}, 18(2):1--29, 2016.

\bibitem{MathewMeier}
Akhil Mathew and Lennart Meier.
\newblock Affineness and chromatic homotopy theory.
\newblock {\em J. Topol.}, 8(2):476--528, 2015.

\bibitem{MathewStojanoska}
Akhil Mathew and Vesna Stojanoska.
\newblock The {P}icard group of topological modular forms via descent theory.
\newblock {\em Geom. Topol.}, 20(6):3133--3217, 2016.

\bibitem{MeierDoktorarbeit}
Lennart Meier.
\newblock United elliptic homology (thesis).
\newblock \url{http://hss.ulb.uni-bonn.de/2012/2969/2969.htm}, 2012.

\bibitem{MeierVB}
Lennart Meier.
\newblock Vector bundles on the moduli stack of elliptic curves.
\newblock {\em J. Algebra}, 428:425--456, 2015.

\bibitem{MeierDecomposition}
Lennart Meier.
\newblock Additive decompositions for rings of modular forms.
\newblock {\em arXiv preprint arXiv:1710.03461}, 2017.

\bibitem{MeierTMFLevel}
Lennart Meier.
\newblock Topological modular forms with level structure: decompositions and
  duality.
\newblock {\em arXiv preprint arXiv:1806.06709}, 2018.

\bibitem{M-OMell}
Lennart Meier and Viktoriya Ozornova.
\newblock {Moduli stack of elliptic curves}.
\newblock
  \url{http://www.ruhr-uni-bochum.de/imperia/md/images/mathematik/folder/lehrstuhlxiii/modulistack.pdf},
  2017.

\bibitem{Nagata}
Masayoshi Nagata.
\newblock {\em Local rings}.
\newblock Interscience Tracts in Pure and Applied Mathematics, No. 13.
  Interscience Publishers a division of John Wiley \& Sons\, New York-London,
  1962.

\bibitem{Naumann}
Niko Naumann.
\newblock The stack of formal groups in stable homotopy theory.
\newblock {\em Adv. Math.}, 215(2):569--600, 2007.

\bibitem{PortaYu}
Mauro Porta and Tony~Yue Yu.
\newblock Higher analytic stacks and {GAGA} theorems.
\newblock {\em Adv. Math.}, 302:351--409, 2016.

\bibitem{QuillenProjective}
Daniel Quillen.
\newblock Projective modules over polynomial rings.
\newblock {\em Invent. Math.}, 36:167--171, 1976.

\bibitem{RavenelAlt}
Douglas~C. Ravenel.
\newblock {\em Complex cobordism and stable homotopy groups of spheres}, volume
  121 of {\em Pure and Applied Mathematics}.
\newblock Academic Press, Inc., Orlando, FL, 1986.

\bibitem{RustomGenerators}
Nadim Rustom.
\newblock Generators of graded rings of modular forms.
\newblock {\em J. Number Theory}, 138:97--118, 2014.

\bibitem{RustomGeneratorsRelations}
Nadim Rustom.
\newblock Generators and relations of the graded algebra of modular forms.
\newblock {\em Ramanujan J.}, 39(2):315--338, 2016.

\bibitem{GAGA}
Jean-Pierre Serre.
\newblock {G\'eom\'etrie alg\'ebrique et g\'eom\'etrie analytique}.
\newblock {\em Ann. Inst. Fourier, Grenoble}, 6:1--42, 1955--1956.

\bibitem{SilvermanAEC}
Joseph~H. Silverman.
\newblock {\em The arithmetic of elliptic curves}, volume 106 of {\em Graduate
  Texts in Mathematics}.
\newblock Springer-Verlag, New York, 1992.
\newblock Corrected reprint of the 1986 original.

\bibitem{SilvermanAdvanced}
Joseph~H. Silverman.
\newblock {\em Advanced topics in the arithmetic of elliptic curves}, volume
  151 of {\em Graduate Texts in Mathematics}.
\newblock Springer-Verlag, New York, 1994.

\bibitem{stacks-project}
The {Stacks Project Authors}.
\newblock {\itshape {S}tacks Project}.
\newblock \url{http://stacks.math.columbia.edu}, 2017.

\bibitem{SuslinProjective}
A.~A. Suslin.
\newblock Projective modules over polynomial rings are free.
\newblock {\em Dokl. Akad. Nauk SSSR}, 229(5):1063--1066, 1976.

\bibitem{Cesna}
K\k{e}stutis \v{C}esnavi\v{c}ius.
\newblock A modular description of {$\mathcal{X}_0(n)$}.
\newblock {\em Algebra Number Theory}, 11(9):2001--2089, 2017.

\bibitem{VistoliDescent}
Angelo Vistoli.
\newblock Grothendieck topologies, fibered categories and descent theory.
\newblock In {\em {Fundamental algebraic geometry}}, volume 123 of {\em Math.
  Surveys Monogr.}, pages 1--104. Amer. Math. Soc., Providence, RI, 2005.

\end{thebibliography}
\end{document}